\documentclass{amsart}

\usepackage[subsec,skipfonts,nodate]{enrabbrev}
\usepackage{eucal}

\newcommand{\dR}{\txt{dR}}
\newcommand{\DR}{\txt{DR}}
\newcommand{\dAff}{\txt{dAff}}
\newcommand{\dSt}{\txt{dSt}}
\newcommand{\Span}{\txt{Span}}
\DeclareMathOperator{\Spec}{Spec}
\newcommand{\SPAN}{\txt{SPAN}}
\newcommand{\COSPAN}{\txt{COSPAN}}
\newcommand{\OSPAN}{\overline{\SPAN}}
\newcommand{\OCOSPAN}{\overline{\COSPAN}}
\newcommand{\Cospan}{\txt{Cospan}}
\newcommand{\CAlg}{\txt{CAlg}}

\newcommand{\Mix}{\txt{$\epsilon$-dg}^{\txt{gr}}}
\newcommand{\bbS}{\bbSigma}
\newcommand{\bbL}{\bbLambda}

\newcommand{\hbbS}{\widehat{\bbS}}
\newcommand{\Dop}{\simp^{\op}}
\newcommand{\Dn}{\simp^{n}}
\newcommand{\Dnop}{\simp^{n,\op}}
\newcommand{\DnIop}{\simp^{n,\op}_{/I}}
\newcommand{\LnIop}{\bbL^{n,\op}_{/I}}

\newcommand{\eg}{e.g.\@}
\newcommand{\ie}{i.e.\@}

\DeclareMathOperator{\Tw}{Tw}
\newcommand{\Twl}{\Tw^{\ell}}

\newcommand{\cA}{\mathcal{A}}
\newcommand{\cC}{\mathcal{C}}
\newcommand{\bD}{\mathbb{D}}
\newcommand{\cD}{\mathcal{D}}
\newcommand{\bE}{\mathbb{E}}
\newcommand{\cE}{\mathcal{E}}

\newcommand{\bL}{\mathbb{L}}
\newcommand{\cM}{\mathcal{M}}

\newcommand{\bP}{\mathbb{P}}
\newcommand{\cP}{\mathcal{P}}
\newcommand{\cS}{\mathcal{S}}
\newcommand{\bT}{\mathbb{T}}

\newcommand{\cAcl}{\mathcal{A}^{2, \mathrm{cl}}}
\newcommand{\cdga}{\mathrm{cdga}}
\newcommand{\Catpo}{\mathrm{Cat}_\infty^{\mathrm{po}}}
\newcommand{\Cois}{\txt{Cois}}
\newcommand{\CoisCorr}{\txt{CoisCorr}}
\newcommand{\CoisCorrns}{\CoisCorr_{n}^{s}}
\newcommand{\CoisCorrnsnd}{\CoisCorr_{n}^{s,\txt{nd}}}
\newcommand{\Comm}{\txt{Comm}}
\newcommand{\CompCorr}{\mathrm{CompCorr}}
\newcommand{\CompCorrns}{\CompCorr_{n}^{s}}
\newcommand{\CompCorrnsnd}{\CompCorr_{n}^{s,\txt{nd}}}
\newcommand{\compat}{\mathrm{compat}}
\newcommand{\CSS}{\txt{CSS}}
\newcommand{\dArt}{\mathrm{dArt}}
\newcommand{\ddr}{\mathrm{d}_{\mathrm{dR}}}
\newcommand{\dg}{\txt{\textbf{dg}}}

\newcommand{\IsotCorr}{\mathrm{IsotCorr}}
\newcommand{\IsotCorrns}{\IsotCorr^{s}_{n}}

\newcommand{\Lagns}{\txt{Lag}_{n}^{s}}
\newcommand{\LMod}{\mathrm{LMod}}
\newcommand{\Mor}{\txt{Mor}}
\newcommand{\Pois}{\txt{Pois}}
\newcommand{\pt}{*}

\newcommand{\QCoh}{\mathrm{QCoh}}
\newcommand{\red}{\mathrm{red}}
\newcommand{\Sym}{\mathrm{Sym}}

\newcommand{\CbbS}{\mathbb{X}}

\theoremstyle{definition}
\newtheorem{assumption}[thm]{Assumption}

\usepackage{xstring}
\renewcommand{\eprint}[1]{\IfBeginWith{#1}{arXiv}{\href{https://arxiv.org/abs/#1}{#1}}{\href{#1}{#1}}}

\usepackage{subfiles}

\author{Valerio Melani}
\address{Dipartimento di Matematica, Universit\`a di Pisa, 
  Pisa, Italy}
\email{valerio.melani@unipi.it}

\author{Pavel Safronov}

\address{Institut f\"{u}r Mathematik, Universit\"{a}t Z\"{u}rich, Zurich, Switzerland}
\email{pavel.safronov@math.uzh.ch}

\title{Shifted Coisotropic Correspondences}

\date{\today.  The work of P.S.\ was supported by the NCCR SwissMAP grant of the Swiss
 National Science Foundation.
 The position of R.H.\ was funded through the grant IBS-R003-D1 of the
 Institute for Basic Science of the Republic of Korea.}

\begin{document}
\begin{abstract}
  We define (iterated) coisotropic correspondences between derived
  Poisson stacks, and construct symmetric monoidal higher categories
  of derived Poisson stacks where the $i$-morphisms are given by
  $i$-fold coisotropic correspondences. Assuming an expected
  equivalence of different models of higher Morita categories, we
  prove that all derived Poisson stacks are fully dualizable, and so
  determine framed extended TQFTs by the Cobordism Hypothesis. Along
  the way we also prove that the higher Morita category of
  $E_{n}$-algebras with respect to coproducts is equivalent to the
  higher category of iterated cospans.
\end{abstract}

\maketitle



\tableofcontents

\section{Introduction}
\subsection{Canonical Relations}
\enlargethispage{1ex}
In symplectic geometry
Weinstein~\cite{WeinsteinSymplCat,WeinsteinSymplCat2} has proposed
that the ``correct'' notion of morphisms between two symplectic
manifolds $(X,\omega_X)$ and $(Y, \omega_Y)$ should be
\emph{Lagrangian correspondences} (also known as \emph{canonical
  relations}), \ie{} Lagrangian submanifolds of
$(X \times Y, \omega_X - \omega_Y)$. As one piece of evidence for this
claim, it is a well-known fact that a smooth map $X \to Y$ is a
symplectomorphism if and only if its graph is a Lagrangian
correspondence. Under certain transversality hypotheses, it is
possible to compose Lagrangian correspondences by taking an
intersection, and Weinstein suggested that a ``category'' of
symplectic manifolds and Lagrangian correspondences should in some
sense be a natural domain for geometric quantization.  However, in
general it is not possible to compose Lagrangian correspondences
(though see \cite{WehrheimWoodward} for a way to partially circumvent
this problem in the context of Floer theory).

Poisson geometry can be viewed as a generalization of symplectic
geometry where we weaken the non-degeneracy condition. In this
context, the analogue of Lagrangian correspondences between Poisson
manifolds $(X,\pi_X)$ and $(Y, \pi_Y)$ are the \emph{coisotropoic
  correspondences}, \ie{} coisotropic submanifolds of
$(X \times Y, \pi_X - \pi_Y)$. A map $X \to Y$ can be shown to be a
Poisson morphism if and only if its graph is a coisotropic
correspondence, and Weinstein~\cite{WeinsteinCoisotropic} proved that
under suitable transversality hypotheses these too can be composed by
taking intersections.

Symplectic and Poisson structures are also important in algebraic
geometry, and here similar problems arise. Indeed, to define
symplectic structures on an algebraic scheme $X$, one requires the
cotangent sheaf $\Omega^1_X$ to be a vector bundle, which means that
the scheme has to be smooth. However, intersections of smooth schemes
are not smooth in general. More generally, we may consider a version
of symplectic structures where we replace the cotangent sheaf
$\Omega^1_X$ by the cotangent complex $\bL_X$, in which case we
require the cotangent complex to be perfect. However, we again run
into the problem that the cotangent complexes of intersections of
schemes with perfect cotangent complexes are in general not themselves
perfect. Extra structures on derived Lagrangian intersections of
symplectic schemes have been studied in \cite{BehrendFantechi} and on
derived coisotropic intersections of Poisson schemes in
\cite{BaranovskyGinzburg}.

A way to deal with the problem of non-transverse intersections is to
work in the setting of \emph{derived} algebraic geometry, where
derived schemes with perfect cotangent complexes \emph{are} stable
under intersections. (Foundational references on derived algebraic
geometry include
\cite{HAG2,GaitsgoryRozenblyum1,GaitsgoryRozenblyum2,SAG}.) In this
setting, analogues of symplectic and Lagrangian structures on derived
schemes and, more generally, derived Artin stacks, have been
introduced by Pantev, To\"en, Vaqui\'e, and Vezzosi~\cite{PTVV}, while
analogues of Poisson and coisotropic structures were introduced by the
same authors together with Calaque~\cite{CPTVV} (see also
\cite{MelaniSafronov2}). More precisely, in the derived setting
differential forms naturally form a bicomplex, which allows us to
consider \emph{shifted} versions of all of these structures (so that,
for example, an $s$-shifted symplectic form gives an equivalence
$\mathbb{T}_{X} \simeq \mathbb{L}_{X}[s]$ with a shift by some integer
$s$, instead of an equivalence $\mathbb{T}_{X} \simeq \mathbb{L}_{X}$
between the tangent and cotangent complexes of a derived stack $X$).

The goal of the present paper is to introduce a notion of (iterated)
shifted coisotropic correspondences between shifted Poisson stacks and
construct higher categories whose objects are shifted Poisson stacks
and whose (higher) morphisms are (iterated) shifted coisotropic
correspondences.

\subsection{The 1-Category of Derived Poisson Stacks}
\enlargethispage{1ex}
Before we describe the contents of this paper in more detail, it is
helpful to first discuss the simplest case of our construction,
namely the 1-category $h\CoisCorr^{s}_{1}$ of $s$-shifted coisotropic correspondences.

For this we must first give a brief sketch of the definition of
$s$-shifted Poisson structures on derived stacks, due to
Calaque--Pantev--To\"en--Vaquié--Vezzosi \cite{CPTVV}. These authors
associate to every derived stack $X$ a certain symmetric monoidal
stable \icat{} $\mathcal{M}_{X}$ (thought of as the \icat{} of
quasi-coherent complexes on the de Rham stack of $X$), contravariantly
functorial in $X$, together with a commutative algebra
$\cP_X^\infty \in \mathcal{M}_{X}$ (which is only a lax functor of
$X$). (We will review
this formalism in more detail in \S\ref{sect:derivedstacks}.) An
$s$-shifted Poisson structure on $X$ is then defined to be a lift of the commutative
algebra structure on $\cP_X^\infty$ to a $\mathbb{P}_{s+1}$-algebra,
where $\mathbb{P}_{s+1}$ is the operad of dg Poisson algebras with a
bracket of degree $-s$.

\begin{remark}
  In other words, we can define an $\infty$-groupoid of $s$-shifted Poisson
  structures on $X$ as the pullback
  \[
  \begin{tikzcd}
    \Pois(X, s) \arrow{r} \arrow{d} &
    \Alg_{\mathbb{P}_{s+1}}(\mathcal{M}_{X}) \arrow{d} \\
    \{ \cP_X^\infty\} \arrow{r} & \txt{CAlg}(\mathcal{M}_{X}).
  \end{tikzcd}
\]
\end{remark}

\begin{remark}
  The additivity theorem for Poisson algebras proved by the third
  author \cite[Theorem 2.22]{SafronovAdditivity} and, independently, Rozenblyum, says that $\mathbb{E}_{n}$-algebras in $\Alg_{\mathbb{P}_{s}}(\mathcal{C})$ are the same thing as
  $\mathbb{P}_{n+s}$-algebras in $\mathcal{C}$. Since
  $\mathbb{E}_{n}$-algebras in commutative algebras are just
  commutative algebras, we could have equivalently used
  $\mathbb{E}_{n}$-algebras in $\mathbb{P}_{s-n+1}$-algebras in
  $\mathcal{M}_{X}$ in our definition of Poisson structures above.
\end{remark}

Derived stacks endowed with a $s$-shifted Poisson structure are the objects of the category $h\CoisCorr^{s}_{1}$. Next, to describe its morphisms, we outline the definition of a \emph{coisotropic correspondence}
between two $s$-shifted Poisson stacks $X$ and $X'$. This is first of
all given by a span $X \xfrom{f} Y \xto{g} X'$ of derived stacks. This
induces a cospan of commutative algebras in $\mathcal{M}_{Y}$,
\[ f^{*}\cP_X^\infty \to  \cP_Y^\infty \from g^{*}\cP_{X'}^\infty.\]
Since the symmetric monoidal structure on commutative algebras is
cocartesian, we can equivalently view this cospan as giving $\cP_Y^\infty$
the structure of an $(f^{*}\cP_X^\infty, g^{*}\cP_{X'}^\infty)$--bimodule in
$\txt{CAlg}(\mathcal{M}_{Y})$. A coisotropic correspondence is then a
lift of this structure to a bimodule in the symmetric monoidal \icat{}
$\Alg_{\mathbb{P}_{s}}(\mathcal{M}_{Y})$, where we view the
$\mathbb{P}_{s+1}$-algebra structures on $\cP_X^\infty$ and $\cP_{X'}^\infty$ as
associative algebras in $\mathbb{P}_{s}$-algebras.
\begin{remark}
  If we have chosen $s$-shifted Poisson structures for $X$ and $X'$,
  this means that the $\infty$-groupoid of compatible coisotropic
  structures on the span is given by the pullback
  \[
  \begin{tikzcd}
    \txt{Cois}_{X,X'}(f, g; s) \arrow{r} \arrow{d} &
    \txt{Mod}_{f^{*}\cP_X^\infty,g^{*}\cP_{X'}^\infty}(\Alg_{\mathbb{P}_{s}}(\mathcal{M}_{Y})) \arrow{d} \\
    \{ \cP_Y^\infty \} \arrow{r} & \txt{Mod}_{f^{*}\cP_X^\infty,g^{*}\cP_{X'}^\infty}(\txt{CAlg}(\mathcal{M}_{Y})).
  \end{tikzcd}
  \]
\end{remark}

The coisotropic correspondences are the morphisms in the category
$h\CoisCorr^{s}_{1}$ of $s$-shifted coisotropic correspondences. To compose two
coisotropic correspondences given by spans
\[
  \begin{tikzcd}
    {} & Y \arrow{dl} \arrow{dr} & & Y'\arrow{dl} \arrow{dr} \\
    X & & X' & & X''
  \end{tikzcd}
\]
we first compose the spans in the usual way, by
forming a pullback
\[
\begin{tikzcd}
{}    & & Z \arrow{dr} \arrow{dl}\\ 
    & Y \arrow{dr} \arrow{dl}& & Y'\arrow{dr} \arrow{dl}\\
  X & & X' & & X''.\\
\end{tikzcd}
\]
This pullback induces a pushout square in
$\txt{CAlg}(\mathcal{M}_{Z})$ (see Proposition~\ref{propn:Ppullbacks})
\[
\begin{tikzcd}
\cP_{X'}^\infty \arrow{r} \arrow{d} & \cP_Y^\infty \arrow{d}\\  
\cP_{Y'}^\infty \arrow{r} & \cP_Z^\infty,
\end{tikzcd}
\]
\ie{} $\cP_Z^\infty \simeq \cP_Y^\infty \otimes_{\cP_{X'}^\infty} \cP_{Y'}^\infty$ (where we omit
notation for the pullbacks to $Z$). To compose two coisotropic
correspondences we take the corresponding relative
tensor product of the bimodules $\cP_Y^\infty$ and $\cP_{Y'}^\infty$
in $\mathbb{P}_{s}$-algebras. This can be interpreted as forming a
composite in the Morita category of algebras and bimodules in
$\Alg_{\mathbb{P}_{s}}(\mathcal{M}_{Z})$ --- this has associative
algebras as objects, with morphisms from $A$ to $B$ given by
$(A, B)$-bimodules and composition given by taking relative tensor
products.

To construct the \icatl{} extension of this category we need a more
structured way of defining it. For this we consider a general notion
of ``spans with coefficients''. If $\mathcal{C}$ is an \icat{} with
pullbacks, then given a functor $F \colon \mathcal{C}^{\op} \to \CatI$,
we can define an \icat{} $\txt{Span}_{1}(\mathcal{C}; F)$ of spans
with coefficients in $F$ such that:
\begin{itemize}
\item an object of $\txt{Span}_{1}(\mathcal{C}; F)$ is a pair $(c \in
  \mathcal{C}, x \in F(c))$,
\item a morphism from $(c,x)$ to $(c',x')$ is a span $c \xfrom{f} d \xto{g}
  c'$ in $\mathcal{C}$ together with a morphism $\phi \colon F(f)(x) \to F(g)(x')$ in
  $F(d)$,
\item given another morphism from $(c',x')$ to $(c'', x'')$
  corresponding to a span $c' \xfrom{f'} d' \xto{g'} c''$ and a
  morphism $\psi \colon F(f')(x') \to F(g')(x'')$, their
  composite is given by composing the spans by taking a pullback
  \[
    \begin{tikzcd}
      {} & & e \arrow{dl}[above left]{h} \arrow{dr}{k} \\
      {} & d \arrow{dl}[above left]{f} \arrow{dr}{g}& &
      d'\arrow{dl}[above left]{f'} \arrow{dr}{g'} \\
      c & & c' & & c'',
    \end{tikzcd}
  \]
  and then composing $F(h)(\phi) \colon F(fh)(x) \to F(gh)(x')
  \simeq F(f'k)(x')$ with $F(k)(\psi) \colon F(f'k)(x')
  \to F(g'k)(x'')$ in $F(e)$.
\end{itemize}
We can apply this to the functors $\mathfrak{P}_{1}^{s},\mathfrak{C}_{1}
\colon \txt{dSt}^{\op} \to \Cat$ given by
$\mathfrak{P}_{1}^{s}(X) =
\mathfrak{alg}_{1}(\Alg_{\mathbb{P}_{s}}(\mathcal{M}_{X}))$ and
$\mathfrak{C}_{1}(X) = \mathfrak{alg}_{1}(\CAlg(\mathcal{M}_{X}))$, where
$\mathfrak{alg}_{1}(\mathcal{C})$ denotes the Morita \icat{} of a
monoidal \icat{} $\mathcal{C}$ \cite{nmorita}. The forgetful functor
from Poisson algebras to commutative algebras induces a functor
\[
  \Span_{1}(\dSt; \mathfrak{P}_{1}^{s})
  \to \Span_{1}(\dSt; \mathfrak{C}_{1}).
\]
Moreover, using the section $\cP_X^\infty \in \CAlg(\mathcal{M}_{X})$ we can
define a functor
$\Span_{1}(\dSt) \to \Span_{1}(\dSt; \mathfrak{C}_{1})$
 which takes
$X \in \dSt$ to $(X, \cP_X^\infty)$ and a span $X \xfrom{f} Z \xto{g} Y$ to
itself plus $\cP_Z^\infty$ viewed as an
$f^{*}\cP_X^\infty$--$g^{*}\cP_Y^\infty$--bimodule. This allows us to define the
\icat{} $\CoisCorr^{s}_{1}$ as the pullback
\[
  \begin{tikzcd}
    \CoisCorr^{s}_{1} \arrow{d} \arrow{r} &
    \Span_{1}(\dSt; \mathfrak{P}_{1}^{s}) \arrow{d}\\
  \Span_{1}(\dSt) \arrow{r} & \Span_{1}(\dSt; \mathfrak{C}_{1}).
  \end{tikzcd}
\]

\subsection{Overview of Results}
In \S\ref{subsec:spancoeff} we use the higher categories of ``spans
with local systems'' defined in \cite{spans} to construct the \icats{}
$\Span_{1}(\mathcal{C}; F)$ as well as their higher-dimensional
cousins $\Span_{n}(\mathcal{C}; F)$, where $F$ is a functor from
$\mathcal{C}$ to the \icat{} of $(\infty,n)$-categories. We then want
to define the $(\infty,n)$-category $\CoisCorrns$ as
a pullback
\[
  \begin{tikzcd}
    \CoisCorrns \arrow{d} \arrow{r} &
    \Span_{n}(\dSt; \mathfrak{P}_{n}^{s}) \arrow{d}\\
  \Span_{n}(\dSt) \arrow{r} & \Span_{n}(\dSt; \mathfrak{C}_{n}).
  \end{tikzcd}
\]
where the functors
$\mathfrak{P}_{n}^{s},\mathfrak{C}_{n} \colon \txt{dSt}^{\op} \to
\Cat_{(\infty,n)}$ are given by
$\mathfrak{P}_{n}^{s}(X) =
\mathfrak{alg}_{n}(\Alg_{\mathbb{P}_{s+1-n}}(\mathcal{M}_{X}))$ and
$\mathfrak{C}_{n}(X) =
\mathfrak{alg}_{n}(\CAlg(\mathcal{M}_{X}))$, with
$\mathfrak{alg}_{n}(\mathcal{C})$ denoting the Morita
$(\infty,n)$-category of $\mathcal{C}$~\cite{nmorita}. However, we
need to do some work to construct the functor $\Span_{n}(\dSt)
\to \Span_{n}(\dSt; \mathfrak{C}_{n})$; for this we prove two
results that may be of independent interest.
\begin{thm}[See Corollary \ref{cor:SPANCOSPAN}]
  Let $\Catpo$ be the subcategory of $\CatI$ whose objects
  are \icats{} with pushouts, and whose morphisms are functors that
  preserve these. Given a functor $F \colon \mathcal{C}^{\op} \to
  \Catpo$ we can form the functor $\Cospan_{n}(F) \colon
  \mathcal{C}^{\op} \to \Cat_{(\infty,n)}$. There is an equivalence of
  $(\infty,n)$-categories
  \[ \Span_{n}(\mathcal{C}; \Cospan_{n}(F)) \simeq
    \Cospan_{n}(\mathcal{F}),\]
  where $\mathcal{F} \to \mathcal{C}^{\op}$ is the cocartesian
  fibration for $F$.
\end{thm}

\begin{thm}[See Corollary \ref{cor:cospanmorita}]
  Suppose $\mathcal{C}$ is an \icat{} with finite colimits. Then there
  is an equivalence of $(\infty,n)$-categories
  \[ \Cospan_{n}(\mathcal{C}) \simeq
    \mathfrak{alg}_{n}(\mathcal{C}^{\amalg}).\]
\end{thm}

Together, these two results lead to a simplified description of
$\Span_{n}(\dSt; \mathfrak{C}_{n})$, which allows us to prove our main
result:
\begin{thm}[See Theorem \ref{thm:coiscorradjoints}]
  There is a symmetric monoidal $(\infty,n)$-category
  $\CoisCorrns$ whose objects are derived stacks with
  $s$-shifted Poisson structures, and whose $i$-morphisms are $i$-fold
  coisotropic correspondences. Assuming all $\mathbb{E}_{n}$-algebras
  are fully dualizable, all objects of this $(\infty,n)$-category are
  fully dualizable.
\end{thm}

It was recently proved by Gwilliam and Scheimbauer
\cite{GwilliamScheimbauer} that the Morita $(\infty,n)$-category has
duals; however, they use a geometric model of this
$(\infty,n)$-category, which is not yet known to be equivalent to the
algebraic model we use. Assuming this comparison (more precisely, see
Conjecture \ref{conj:moritaadjoints}), as well as the Cobordism
Hypothesis, we have:
\begin{cor}
  Every $s$-shifted derived Poisson stack $X$ determines a framed $n$-dimensional
  extended topological quantum field theory
  \[ \txt{Bord}^{\txt{fr}}_{0,n} \to \CoisCorrns.\]
\end{cor}

Note that the $(\infty, n)$-category of $s$-shifted Lagrangian
correspondences $\Lagns$ has recently been defined in
\cite{aksz}.

It is known \cite{CPTVV, PridhamPoisson} that $s$-shifted Poisson
structures satisfying a non-degeneracy condition are equivalent to
$s$-shifted symplectic structures in the sense of \cite{PTVV}, and
similarly that non-degenerate coisotropic structures are equivalent to
Lagrangian structures \cite{PridhamLagrangian, MelaniSafronov2}. In 
\S\ref{sect:LagrangianCorrespondences} we explain how we expect these
equivalences to generalize to relate $\CoisCorrns$ to a symmetric monoidal $(\infty,
n)$-category $\Lagns$ of $s$-shifted symplectic stacks and iterated
Lagrangian correspondences, which is constructed in forthcoming work
of the first author with Calaque and Scheimbauer~\cite{aksz}.

\section{Categorical Preliminaries}
In this section we carry out the preliminary categorical constructions
we require. We begin by briefly reviewing the definitions of (and
fixing our notation for) iterated Segal spaces in
\S\ref{subsec:revseg}, and then recalling the construction of higher
categories of spans from \cite{spans} in \S\ref{subsec:revspans}. In
\S\ref{subsec:spancoeff} we use this to introduce higher categories of
spans with coefficients in an $(\infty,n)$-category.
For the case of spans with coefficients in cospans we then provide a simpler
description of this construction in \S\ref{subsec:spancoeffcospan}.
In \S\ref{subsec:revmorita} we recall the definition of the higher Morita
category of $E_{n}$-algebras from \cite{nmorita}, which we use in
\S\ref{subsec:cospanmor} to prove that the higher category of cospans
is a higher Morita category.

\subsection{Review of Iterated Segal Spaces}\label{subsec:revseg}
The goal of this subsection is to provide a brief review of the theory
of iterated Segal spaces, which was introduced by Barwick in
\cite{BarwickThesis}; iterated Segal spaces will be our model for
\incats. Our discussion here is mainly intended to fix the notation we
use in the rest of the paper; we refer the reader to
\cite{spans}*{\S\S 3, 4, 7, 11} for further details and motivation.

\begin{defn}\label{def:inert}
  We write $\simp$ for the usual simplex category, with objects the
  ordered sets $[n] := \{0,1,\ldots,n\}$ and order-preserving
  functions as morphisms. A morphism $\phi \colon [n] \to [m]$ in
  $\simp$ is called \emph{inert} if it is the inclusion of a
  sub-interval, \ie{} if $\phi(i) = \phi(0)+i$ for all $i$, and
  \emph{active} if it preserves the end points, \ie{} if $\phi(0) = 0$
  and $\phi(n) = m$. We write $\simp_{\txt{int}}$ for the subcategory
  of $\simp$ containing only the inert maps.
\end{defn}

\begin{notation}
For all $n$, we have maps in $\simp$
\[ \sigma_i: [0] \to [n],\qquad \rho_i:[1] \to [n] \] where $\sigma_i$
($0 \leq i \leq n$) sends $0$ to $i$ and $\rho_i$ ($0<i\leq n$) sends
$0$ and $1$ to $i-1$ and $i$, respectively.
\end{notation}

\begin{defn}\label{defn:catobj}
  Let $\mathcal{C}$ be an \icat{} with pullbacks. A \emph{category object} in $\cC$ is a functor $X: \simp^{\op} \to \cC$ such that the natural morphisms induced by the maps $\sigma_i$ and $\rho_i$ 
  \[ X_n \to X_1 \times_{X_0} \cdots \times_{X_0} X_1 \]
  are equivalences in $\cC$, for all $n$. We let $\Cat(\cC)$ denote
  the full subcategory of $\Fun(\Dop, \mathcal{C})$ spanned by the
  category objects. 
\end{defn}

The above definition can be iterated, leading us to the following notion:
\begin{defn}\label{defn:ncatobj}
  Let $\mathcal{C}$ be an \icat{} with pullbacks. A \emph{$n$-uple
    category object} in $\cC$ is defined inductively as a category
  object in the \icat{} of $(n-1)$-category objects. We let
  $\Cat^n(\cC)$ denote the \icat{} of $n$-uple category objects in
  $\cC$, viewed as a full subcategory of $\Fun(\Dnop,
  \mathcal{C})$. If $\mathcal{C}$ is the \icat{} $\mathcal{S}$ of
  spaces, we refer to $n$-uple category objects as \emph{$n$-uple
    Segal spaces}.
\end{defn}

Among the $n$-uple Segal spaces we can single out those that describe
$(\infty,n)$-categories by imposing constancy conditions:
\begin{defn}
  Let $\cC$ be an \icat{} with pullbacks. A \emph{1-fold Segal object}
  in $\cC$ is simply a category object in $\cC$. We now say
  inductively that an \emph{$n$-fold Segal object} in $\cC$ is an
  $n$-uple category object $X$ in $\cC$ such that
  \begin{itemize}
  \item the restriction $X_{0, \bullet,\ldots, \bullet} \in \Cat^{n-1}(\cC)$ is constant;
  \item the restrictions $X_{k, \bullet,\ldots, \bullet} \in \Cat^{n-1}(\cC)$ are $(n-1)$-fold Segal objects for all $k$. 
  \end{itemize}
  We denote by $\Seg_n(\cC)$ the full subcategory of $\Cat^n(\cC)$
  spanned by $n$-fold Segal objects. If $\mathcal{C}$ is the \icat{}
  $\mathcal{S}$ of spaces, we refer to $n$-fold Segal objects as
  \emph{$n$-fold Segal spaces}.
\end{defn}

By definition, the category $\Seg_n(\cC)$ comes equipped with an inclusion functor to $\Cat^n(\cC)$.
\begin{propn}\label{propn:Useg}
  Let $\cC$ be an \icat{} with pullbacks. The inclusion
  $\Seg_n(\cC) \to \Cat^n(\cC)$ admits a right adjoint, which will be
  denoted by $U_{\Seg}^n$.
\end{propn}
We refer to \cite[Proposition 4.12]{spans} for a proof. If $X$ is an $n$-fold Segal object in $\cC$, we refer to $U_{\Seg}^nX$ as the \emph{underlying $n$-uple category object} of $X$.

To obtain the correct \icat{} of $(\infty,n)$-categories we must
invert the fully faithful and essentially surjective morphisms. By
results of Rezk~\cite{RezkCSS} in the case $n =1$ and
Barwick~\cite{BarwickThesis} in general this localization is given by
the full subcategory of \emph{complete} objects, defined as follows:
\begin{defn}
  Let $X$ be a $n$-fold Segal space. We inductively say that $X$ is
  \emph{complete} if
  \begin{itemize}
  \item the Segal space $X_{\bullet, 0, \ldots, 0}$ is complete in the
    sense of \cite{RezkCSS};
  \item the $(n-1)$-fold Segal space $X_{1, \bullet, \ldots, \bullet}$ is complete.
  \end{itemize}
  We denote by $\CSS_{n}(\cS)$ the full subcategory of $\Seg_n(\cS)$
  spanned by complete $n$-fold Segal spaces.
\end{defn}
We also denote $\CSS_{n}(\mathcal{S})$ by $\Cat_{(\infty,n)}$; this
\icat{} is equivalent to those of other descriptions of
$(\infty,n)$-categories by \cite{BarwickSchommerPriesUnicity}.

\begin{notation}
  Let $\mathcal{D}$ be an $n$-fold Segal space, and let $x$ and $y$ be
  objects of $\mathcal{D}$ (\ie{} points of
  $\mathcal{D}_{0,\ldots,0}$). Then the $(n-1)$-fold Segal space
  $\mathcal{D}(x,y)$ of morphisms from $x$ to $y$ is defined by the
  pullback square
  \[
    \begin{tikzcd}
      \mathcal{D}(x,y) \arrow{r} \arrow{d} & \mathcal{D}_{1}\arrow{d} \\
      \{(x,y)\} \arrow{r} & \mathcal{D}_{0} \times \mathcal{D}_{0}
    \end{tikzcd}
  \]
  of $(n-1)$-fold Segal spaces.
\end{notation}

Since (complete) $n$-fold Segal spaces are models for
\incats{}, it is natural to consider a notion of monoidal structures
on these objects.
\begin{defn}\label{defn:assmonoid}
  Let $\cC$ be an \icat{} with finite products. An \emph{associative
    monoid} in $\cC$ is a functor $A: \simp^{\op} \to \cC$ such that
  the natural maps
  \[ A_n \to A_1 \times \cdots \times A_1 \] are equivalences for all
  $n$. We denote by $\Mon(\cC)$ the full subcategory of
  $\Fun(\simp^{\op}, \cC)$ spanned by associative monoids. Monoids in
  the categories $\Seg_n(\cS)$ or $\CSS^n(\cS)$ will be called
  \emph{monoidal} $n$-fold (complete) Segal spaces.
\end{defn}

We can once again iterate the above definition:
\begin{defn}\label{defn:uplemonoid}
  Inductively, a \emph{$k$-uple monoid} in $\cC$ is simply defined to
  be a $(k-1)$-uple monoid in $\Mon(\cC)$, and we denote by
  $\Mon_k(\cC)$ the category of $k$-uple monoids in $\cC$. The
  $k$-uple monoids in $\Seg_n(\cS)$ or in $\CSS_n(\cS)$ are called
  \emph{$k$-uply monoidal} $n$-fold (complete) Segal spaces.
\end{defn}

\begin{remark}
	Note that iterating Definition \ref{defn:assmonoid} does not produce any new operation, but instead adds commutativity constraints on the already existing one. This is analogous to the Eckmann-Hilton argument and to the Dunn-Lurie additivity (see also Remark \ref{rmk:EkvsDn} below).
\end{remark}

Notice that there are natural functors
$\Mon_k(\cC) \to \Mon_{k-1}(\cC)$ for all $k$, which are defined by
sending $X \in \Mon_k(\cC)$ to
$X_{1, \bullet, \ldots, \bullet} \colon \simp^{n-1,\op} \to \cC$.
\begin{defn}
  The \icat{} $\Mon_{\infty}(\cC)$ of \emph{$\infty$-uple monoids} in
  $\cC$ is defined to be the limit of the diagram
  \[ \cdots \to \Mon_k(\cC) \to \Mon_{k-1}(\cC) \to \cdots \to \Mon(\cC)
    \to \cC.  \]
  If $\cC$ is $\Seg_n(\cS)$ or $\CSS_n(\cS)$,
  elements in $\Mon_{\infty}(\cC)$ will be called \emph{$\infty$-uply
    monoidal} $n$-fold (complete) Segal spaces.
\end{defn}

\begin{remark}\label{rmk:EkvsDn}
  $k$-uply and $\infty$-uply monoidal (complete)
  $n$-fold Segal spaces can be equivalently described as
  $\bE_n$-algebras and $\bE_{\infty}$-algebras, in the sense of
  \cite{HA} (we refer to \cite[Proposition 10.12]{spans} for a precise
  statement). Therefore, $k$-uply and $\infty$-uply monoidal
  (complete) $n$-fold Segal spaces will be alternatively called
  \emph{$\bE_k$-monoidal} and \emph{symmetric monoidal} (complete)
  $n$-fold Segal spaces.
\end{remark}

\begin{remark}
  If $\mathcal{D}$ is an $n$-fold Segal space and $x$ is an object of
  $\mathcal{D}$ then $\mathcal{D}(x,x)$ is canonically a monoidal
  $(n-1)$-fold Segal space. This construction can be iterated, so that
  if we have a sequence of $(n+i)$-fold Segal spaces $\mathcal{D}_{i}$
  $(0 \leq i \leq m)$ and objects $x_{i}$ in $\mathcal{D}_{i}$ such that
  $\mathcal{D}_{i}(x_{i},x_{i}) \simeq \mathcal{D}_{i-1}$, then
  $\mathcal{D}_{0}$ is an $\bE_{m}$-monoidal $n$-fold Segal space
  (where we may have $m = \infty$).
\end{remark}

We now briefly recall what it means for an $(\infty,n)$-category to
have adjoints and duals (see \cite{LurieCob} or \cite{spans}*{\S 11} for more details):
\begin{defn}\label{defn:hasduals}
  Let $\mathcal{D}$ be a 2-fold Segal space, and let
  $h_{2}\mathcal{D}$ denote its homotopy 2-category. A 1-morphism in
  $\mathcal{D}$ is a (left or right) \emph{adjoint} if its image in
  $h_{2}\mathcal{D}$ is one. We say that $\mathcal{D}$ \emph{has
    adjoints for 1-morphisms} if every 1-morphism in $\mathcal{D}$ is
  both a left and right adjoint. If $\mathcal{D}$ is an $n$-fold Segal
  space we similarly say that $\mathcal{D}$ has adjoints for
  1-morphisms if its underlying 2-fold Segal space has adjoints for
  1-morphisms; by induction we then say that $\mathcal{D}$ has
  adjoints for $i$-morphisms for $i > 1$ if $\mathcal{D}(x,y)$ has
  adjoints for $(i-1)$-morphisms for all objects $x,y$. If
  $\mathcal{D}$ has adjoints for $i$-morphisms for all $1 \leq i < n$
  we simply say that $\mathcal{D}$ \emph{has adjoints}, while a
  ($k$-uply) monoidal $n$-fold Segal space \emph{has duals} if it has adjoints
  when viewed as an $(n+1)$-fold Segal space.
\end{defn}

We need to know that these properties are preserved under pullbacks,
which is a consequence of the following observation:
\begin{propn}
  Let
  \[
    \begin{tikzcd}
      \cC \arrow{d} \arrow{r} & \cD_1 \arrow{d} \\
      \cD_2 \arrow{r} & \cE
    \end{tikzcd}
  \]
  be a pullback in the \icat{} $\CSS_{2}(\mathcal{S})$ of $(\infty,2)$-categories. Then a morphism in
  $\mathcal{C}$ has a left (right) adjoint \IFF{} its images in
  $\mathcal{D}_{1}$ and $\mathcal{D}_{2}$ have left (right) adjoints.
\end{propn}
\begin{proof}
  Let $\txt{Adj}$ denote the free adjunction 2-category. This is
  described explicitly in \cite{RiehlVerityAdj}, where it is proved
  that an adjunction in an $(\infty,2)$-category $\mathcal{K}$ is
  equivalent to a functor $\txt{Adj} \to \mathcal{K}$, from which it
  is clear that any functor of $(\infty,2)$-categories must preserve
  adjunctions. It thus suffices to show the ``if'' direction, which we
  do for the case of left adjoints.

  Let $l \colon \Delta^{1} \to \txt{Adj}$ denote the inclusion of the
  1-morphism that is a left adjoint. By \cite{RiehlVerityAdj}*{Theorem
    4.4.18}, for any $(\infty,2)$-category $\mathcal{K}$ the fibres of
  \[ l^{*} \colon \Map_{\Cat_{(\infty,2)}}(\txt{Adj}, \mathcal{K}) \to
    \Map_{\Cat_{(\infty,2)}}(\Delta^{1}, \mathcal{K})\] are either
  empty or contractible, and a 1-morphism in $\mathcal{K}$ is a left
  adjoint precisely when the fibre is non-empty. Moreover, our
  pullback square gives a commutative cube
  \[
    \begin{tikzcd}[cramped]
      \Map(\txt{Adj}, \mathcal{C}) \arrow{rr} \arrow{dr} \arrow{dd} & & \Map(\txt{Adj}, \mathcal{D}_{1})
      \arrow{dr} \arrow{dd}
      \\
      & \Map(\txt{Adj}, \mathcal{D}_{2})
       \arrow[crossing over]{rr} & & \Map(\txt{Adj}, \mathcal{E}) \arrow{dd}\\
      \Map(\Delta^{1}, \mathcal{C}) \arrow{rr} \arrow{dr} & & \Map(\Delta^{1}, \mathcal{D}_{1}) \arrow{dr}
      \\
       & \Map(\Delta^{1}, \mathcal{D}_{2})
       \arrow[leftarrow,crossing over]{uu} \arrow{rr} & & \Map(\Delta^{1}, \mathcal{E}),      
    \end{tikzcd}
  \]
  where the top and bottom faces are pullbacks. Given a 1-morphism $f$ in
  $\mathcal{C}$ we get a pullback square of fibres, which shows that
  if the images of $f$ in $\mathcal{D}_{1}$ and $\mathcal{D}_{2}$ are
  left adjoints, then $f$ is a left adjoint.
\end{proof}

Since the notions of ``having duals'' and ``having adjoints'' are
defined in terms of adjunctions in $(\infty,2)$-categories, we get the
following as an immediate consequence:
\begin{cor}\label{cor:hasdualspb}
Let
\[
\begin{tikzcd}
\cC \arrow{d} \arrow{r} & \cD_1 \arrow{d} \\
\cD_2 \arrow{r} & \cE
\end{tikzcd}
\]
be a pullback of symmetric monoidal $(\infty, n)$-categories.
\begin{enumerate}[(i)]
\item If $\cD_1$ and $\cD_2$ have adjoints, then $\cC$ has adjoints.
\item If $\cD_1$ and $\cD_2$ have duals, then $\cC$ has duals.
\end{enumerate}
\label{cor:dualsadjointspullback}
\end{cor}

\begin{remark}
  In the case of duals for objects, this is also a consequence of
  \cite{HA}*{Proposition 4.6.1.11}.
\end{remark}

\subsection{Review of Higher Categories of
  Spans}\label{subsec:revspans}

In this subsection we will review the definition of higher categories
of iterated spans from \cite{spans}.

\begin{defn}
  We write $\bbS^{n}$ for the partially ordered set of pairs $(i,j)$
  with $0 \leq i \leq j \leq n$, with $(i,j) \leq (i',j')$ if
  $i \leq i'$ and $j' \leq j$. A map $\phi \colon [n] \to [m]$ in
  $\simp$ determines a functor $\bbS^{n} \to \bbS^{m}$ taking $(i,j)$
  to $(\phi(i), \phi(j))$, yielding a functor
  $\bbS^{\bullet} \colon \simp \to \Cat$. We write
  $\bbS^{i_{1},\ldots,i_{n}} := \bbS^{i_{1}} \times \cdots \times
  \bbS^{i_{n}}$,
  which gives functors
  $\bbS^{\bullet,\ldots,\bullet} \colon \simp^{n} \to \Cat$. We let
  $\hbbS^{n} \to \simp^{n,\op}$ denote the cartesian fibration for
  this functor.
\end{defn}

\begin{defn}
  If $\mathcal{C}$ is an \icat{}, we let $\OSPAN^{+}(\mathcal{C}) \to
  \simp^{n,\op}$ be the cocartesian fibration for the functor
  $\Fun(\bbS^{\bullet,\ldots,\bullet}, \mathcal{C}) \colon
  \simp^{n,\op} \to \CatI$. We also write $\OSPAN(\mathcal{C}) \to
  \simp^{n,\op}$ for its underlying left fibration, corresponding to
  the functor  $\Map(\bbS^{\bullet,\ldots,\bullet}, \mathcal{C}) \colon
  \simp^{n,\op} \to \mathcal{S}$.
\end{defn}

\begin{remark}\label{rmk:OSPAN+univprop}
The \icat{} $\OSPAN^{+}(\mathcal{C})$ has a universal property by
\cite{freepres}*{Proposition 7.3}: For any \icat{} $\mathcal{K}$ over
$\simp^{n,\op}$ we have a natural equivalence
\[ \Map_{/\simp^{n,\op}}(\mathcal{K}, \OSPAN^{+}(\mathcal{C})) \simeq \Map(\mathcal{K}
\times_{\simp^{n,\op}} \hbbS^{n}, \mathcal{C}).\]
\end{remark}

\begin{defn}
  Let $\bbL^{i}$ denote the full subcategory of $\bbS^{i}$ spanned by
  the pairs $(i,j)$ with $j-i \leq 1$. These subcategories are
  preserved by inert maps in $\simp$, giving a functor
  $\bbL^{\bullet} \colon \simp_{\txt{int}} \to \Cat$. Similarly, we
  define $\bbL^{i_{1},\ldots,i_{n}} := \bbL^{i_{1}}\times \cdots
  \times \bbL^{i_{k}}$, which gives a functor
  $\bbL^{\bullet,\ldots,\bullet} \colon \simp_{\txt{int}}^{n} \to
  \Cat$ with a natural transformation $\bbL^{\bullet,\ldots,\bullet}
  \to \bbS^{\bullet,\ldots,\bullet}|_{\simp_{\txt{int}}^{n}}$.
\end{defn}

\begin{defn}
  Let $\mathcal{C}$ be an \icat{} with pullbacks. We say a functor
  $f \colon \bbS^{i_{1},\ldots,i_{k}}\to \mathcal{C}$ is
  \emph{cartesian} if it is a right Kan extension of its restriction
  to $\bbL^{i_{1},\ldots,i_{k}}$. We write
  $\SPAN^{+}_{n}(\mathcal{C})$ and $\SPAN_{n}(\mathcal{C})$ for the
  full subcategories of $\OSPAN^{+}_{n}(\mathcal{C})$ and
  $\OSPAN_{n}(\mathcal{C})$, respectively, spanned by the cartesian
  functors.
\end{defn}

We then have:
\begin{itemize}
\item The restricted projection
  $\SPAN^{+}_{n}(\mathcal{C}) \to \simp^{n,\op}$ is a cocartesian
  fibration, by \cite{spans}*{Corollary 5.12}.
\item The corresponding functor $\simp^{n,\op} \to \CatI$ is an
  $n$-uple category object, by \cite{spans}*{Proposition 5.14}.
\end{itemize}
Similarly, $\SPAN_{n}(\mathcal{C})$ is an $n$-uple Segal space.

\begin{defn}
  We let $\Span_{n}(\mathcal{C}) := U^n_{\Seg}(\SPAN_n(\cC))$ be the underlying $n$-fold Segal
  space of $\SPAN_{n}(\mathcal{C})$.
\end{defn}

\begin{notation}
  If $\mathcal{C}$ is an \icat{} with pushouts, we also write
  $\COSPAN_{n}(\mathcal{C}) := \SPAN_{n}(\mathcal{C}^{\op})$ and
  $\Cospan_{n}(\mathcal{C}) := \Span_{n}(\mathcal{C}^{\op})$.
\end{notation}

We have the following results from \cite{spans}:
\begin{itemize}
\item The $n$-fold Segal space $\Span_{n}(\mathcal{C})$ is complete,
  by \cite{spans}*{Corollary 8.5}.
\item For objects, $x,y \in \mathcal{C}$, the $(n-1)$-fold Segal space
  of maps $\Span_{n}(\mathcal{C})(x,y)$ is naturally equivalent to
  $\Span_{n-1}(\mathcal{C}_{/x, y})$, by
  \cite{spans}*{Proposition 8.3}. Here $\mathcal{C}_{/x,y} :=
  \mathcal{C}_{/x} \times_{\mathcal{C}} \mathcal{C}_{/y}$ is the
  \icat{} of spans $x \from c \to y$ with $x$ and $y$ fixed.
\item As a consequence, if $\mathcal{C}$ has a terminal object (\ie{}
  $\mathcal{C}$ has all finite limits), then the $(\infty,n)$-category
  $\Span_{n}(\mathcal{C})$ has a natural symmetric monoidal structure,
  as in \cite{spans}*{Proposition 12.1}.
\item The symmetric monoidal $(\infty,n)$-category
  $\Span_{n}(\mathcal{C})$ has duals, by \cite{spans}*{Corollary 12.5}.
\end{itemize}

Following \cite{spans}*{Section 6}, we also consider a variant of the definition of
$\Span_{n}(\mathcal{C})$, giving a higher category of ``iterated spans
with local systems'' in a category object in $\mathcal{C}$:
\begin{notation}
  Let $\Pi \colon \widehat{\bbS} \to \Dop$ be the functor taking
  $([n], (i,j))$ to $[j-i]$, and a morphism $([n], (i,j)) \to ([m],
  (i',j'))$ given by a morphism $\phi \colon [m] \to [n]$ in $\simp$
  such that $(i,j) \leq (\phi(i'),\phi(j'))$ to the morphism $[j'-i']
  \to [j-i]$ given by $s \mapsto \phi(i'+s)-i$. We write $\Pi^{n}$ for
  the product of $n$ copies of $\Pi$, and $\Pi_{I} \colon \bbS^{I} \to
  \Dnop$ for its
  restriction to $\bbS^{I}$.
\end{notation}

\begin{defn}
  Let $\mathcal{C}$ be an \icat{} with pullbacks. Given a functor $F
  \colon \Dnop \to \mathcal{C}$, we write $\OSPAN^{+}_{n}(\mathcal{C};
  F) \to \Dnop$ for the cocartesian fibration corresponding to the
  functor $I \mapsto \Fun(\bbS^{I}, \mathcal{C})_{/F \circ
    \Pi_{I}}$. 
\end{defn}

\begin{remark}
  $\OSPAN^{+}_{n}(\mathcal{C};F)$ can also be described as the \icat{} of
  commutative diagrams
  \[
    \begin{tikzcd}
      \bbS^{I} \arrow{r} \arrow{d}{\Pi_{I}} & \mathcal{C}^{\Delta^{1}}
      \arrow{d}{\txt{ev}_{1}} \\
      \Dnop \arrow{r}{F} & \mathcal{C}
    \end{tikzcd}
  \]
  If we write $\mathcal{C}_{//F}$ for the pullback
  \[
    \begin{tikzcd}
      \mathcal{C}_{//F}  \arrow{d} \arrow{r} &
      \mathcal{C}^{\Delta^{1}} \arrow{d}{\txt{ev}_{1}} \\
      \Dnop \arrow{r}{F} & \mathcal{C},
    \end{tikzcd}
  \]
  this means we can describe $\OSPAN_{n}^{+}(\mathcal{C};F)$ as the pullback
  \[
    \begin{tikzcd}
      \OSPAN_{n}^{+}(\mathcal{C};F) \arrow{r} \arrow{d}&
      \OSPAN^{+}_{n}(\mathcal{C}_{//F}) \arrow{d} \\
      \Dnop \arrow{r} & \OSPAN_{n}^{+}(\Dnop),
    \end{tikzcd}
  \]
  where the bottom horizontal map is the section of
  $\OSPAN_{n}^{+}(\Dnop) \to \Dnop$ corresponding to $\Pi^{n}$ under
  the equivalence
  \[ \Map_{/\Dnop}(\Dnop, \OSPAN^{+}(\Dnop)) \simeq
    \Map(\hbbS^{n}, \Dnop)\]
  of Remark~\ref{rmk:OSPAN+univprop}.
\end{remark}

\begin{defn}
  Suppose $\mathcal{C}$ is an \icat{} with pullbacks and $F \colon
  \Dnop \to \mathcal{C}$ is an $n$-uple category object. (Then
  $\Pi_{I}F \colon \bbS^{I} \to \mathcal{C}$ is cartesian for all $I$
  by \cite{spans}*{Lemma 6.4}.) We define
  $\SPAN_{n}^{+}(\mathcal{C}; F)$ as the pullback
  \[
    \begin{tikzcd}
      \SPAN_{n}^{+}(\mathcal{C};F) \arrow{r} \arrow{d} &
      \OSPAN_{n}^{+}(\mathcal{C}; F) \arrow{d} \\
      \SPAN_{n}^{+}(\mathcal{C}) \arrow{r} & \OSPAN_{n}^{+}(\mathcal{C}).
    \end{tikzcd}
  \]
\end{defn}
Then $\SPAN_{n}^{+}(\mathcal{C};F) \to \Dnop$ is a cocartesian
fibration, being a fibre product of cocartesian fibrations over
$\Dnop$ along functors that preserve cocartesian morphisms. Moreover,
it corresponds to an $n$-uple category object in $\CatI$ by
\cite{spans}*{Proposition 6.7}. We write $\SPAN_{n}(\mathcal{C};F)$
for the underlying left fibration, which corresponds to an $n$-uple
Segal space, and $\Span_{n}(\mathcal{C};F)$ for its underlying
$n$-fold Segal space.

\begin{remark}\label{rmk:spanlocsysfold}
  Using the description of the right adjoint $U_{\Seg}^{n}$ in terms
  of iterated pullbacks in the
  proof of \cite[Proposition 4.12]{spans}, it is easy to see that for
  an $n$-uple category object $F \colon \Dnop \to \mathcal{C}$, we
  have
  \[ \Span_{n}(\mathcal{C}; F) \simeq U^{n}_{\Seg} \SPAN_{n}(\mathcal{C}; F) \simeq
    \Span_{n}(\mathcal{C}; U_{\Seg}^{n}F).\]
\end{remark}

If $\mathcal{C}$ is an \icat{} with finite limits, and $\xi,\eta$ are
objects of $\Span_{n}(\mathcal{C};F)$, corresponding to morphisms $\xi
\colon x
\to F_{0,\ldots,0}$, $\eta \colon y \to F_{0,\ldots,0}$ in
$\mathcal{C}$, then by \cite{spans}*{Proposition 9.3} we can identify the
$(n-1)$-fold Segal space of maps $\Span_{n}(\mathcal{C};F)(\xi,\eta)$
with $\Span_{n-1}(\mathcal{C}; F_{\xi,\eta})$, where $F_{\xi,\eta}$ is
the functor $\simp^{n-1,\op} \to \mathcal{C}$ defined as the pullback
\[
  \begin{tikzcd}
    F_{\xi,\eta} \arrow{r} \arrow{d} & F_{1}\arrow{d} \\
    x \times y \arrow{r}{\xi \times \eta} & F_{0} \times F_{0}.
  \end{tikzcd}
\]
Here it will be convenient to slightly reformulate this, using the
following observation:
\begin{lemma}\label{lem:SpanFslice}
  Suppose $\mathcal{C}$ is an $\infty$-category with pullbacks. Given a functor $F
  \colon \Dnop \to \mathcal{C}_{/x}$ there is a natural equivalence
\[ \Span_{n}(\mathcal{C}_{/x}; F) \simeq \Span_{n}(\mathcal{C}; F). \]
\end{lemma}
\begin{proof}
  The commutative square
\[ 
  \begin{tikzcd}
    (\mathcal{C}_{/x})^{\Delta^{1}} \arrow{r} \arrow{d}{\txt{ev}_{1}} &
    \mathcal{C}^{\Delta^{1}} \arrow{d}{\txt{ev}_{1}}\\
    \mathcal{C}_{/x} \arrow{r} & \mathcal{C}
  \end{tikzcd}
\]
is cartesian; pulling back along $F
\colon \Dnop \to \mathcal{C}_{/x}$ we get a natural equivalence
$\mathcal{C}_{//F} \simeq (\mathcal{C}_{/x})_{//F}$, and hence a
natural equivalence $\OSPAN_{n}(\mathcal{C};F) \simeq
\OSPAN_{n}(\mathcal{C}_{/x};F)$, which restricts to an equivalence
$\Span_{n}(\mathcal{C};F) \simeq \Span_{n}(\mathcal{C}_{/x};F)$.
\end{proof}
As a consequence, we may identify $\Span_{n}(\mathcal{C};F)(\xi,\eta)$
with $\Span_{n-1}(\mathcal{C}_{/x \times y}; F_{\xi,\eta})$; it is
easy to see that this identification is compatible with the
identification $\Span_{n}(\mathcal{C})(x,y) \simeq
\Span_{n-1}(\mathcal{C}_{/x \times y})$.

It follows that if $F$ is a functor to symmetric monoidal $n$-fold
Segal objects in $\mathcal{C}$, then $\Span_{n}(\mathcal{C};F)$ is a
symmetric monoidal $n$-fold Segal space (see \cite{spans}*{Proposition 13.1}).

If $\mathcal{X}$ is an $\infty$-topos, we can make sense of
\emph{complete} $n$-fold Segal objects in $\mathcal{X}$, and of
(symmetric monoidal) $n$-fold Segal objects \emph{having adjoints}
(and \emph{having duals}). We then have
that:
\begin{itemize}
\item If $F \colon \Dnop \to \mathcal{X}$ is complete, then
$\Span_{n}(\mathcal{X};F)$ is a complete $n$-fold Segal space by
\cite{spans}*{Corollary 9.7}.
\item If $F$ has adjoints, then so does $\Span_{n}(\mathcal{X}; F)$,
  by \cite{spans}*{Theorem 3.3}.
\item If $F$ is a symmetric monoidal complete $n$-fold Segal object in
  $\mathcal{X}$ that has duals, then $\Span_{n}(\mathcal{X};F)$ has
  duals.
\end{itemize}
Here we only consider $\mathcal{X}$ of the form
$\mathcal{P}(\mathcal{C})$ for some \icat{} $\mathcal{C}$, in which
case all these notions are given objectwise in $\mathcal{C}$ by the
usual notions for $n$-fold Segal spaces.

\subsection{Spans with Coefficients}\label{subsec:spancoeff}
We now introduce higher categories of spans with coefficients as a
variant of the constructions above:
\begin{defn}
  Suppose $\mathcal{C}$ is a small \icat{}. Given a functor $F \colon
  \mathcal{C}^{\op} \to \Seg_{n}(\mathcal{S})$, we define the $n$-fold
  Segal space of \emph{spans in $\mathcal{C}$ with coefficients in
    $F$} as the pullback
\[
  \begin{tikzcd}
    \Span_{n}(\mathcal{C}; F) \arrow{r} \arrow{d}
     &
    \Span_{n}(\mathcal{P}(\mathcal{C}); F') \arrow{d} \\
    \Span_{n}(\mathcal{C}) \arrow{r} & \Span_{n}(\mathcal{P}(\mathcal{C})),
  \end{tikzcd}
\]
where $F'$ is $F$ regarded as a functor $\Dnop \to
\mathcal{P}(\mathcal{C})$, $\Span_{n}(\mathcal{P}(\mathcal{C}); F')$
is the \icat{} of spans in the $\infty$-topos
$\mathcal{P}(\mathcal{C})$ with coefficients in $F'$, and the bottom horizontal functor is induced by the
Yoneda embedding. We define the variants $\SPAN_{n}(\mathcal{C}; F)$,
etc., similarly.
\end{defn}

\begin{remark}\label{rmk:OSPANcoeffbifib}
  From the definition of $\OSPAN_{n}(\mathcal{P}(\mathcal{C});F')$
  we see that $\OSPAN_{n}(\mathcal{C}; F)$ has the following
  description: its fibre at $I$ is the space of commutative diagrams
\[
  \begin{tikzcd}
  \bbS^{I} \arrow{r} \arrow{d} \arrow[bend right=65]{dd}[left]{\Pi_{I}} & \mathcal{P}(\mathcal{C})^{\Delta^{1}}
  \arrow{d} \\
   \mathcal{C}\times \Dnop  \arrow{r}{Y \times F'} \arrow{d} & \mathcal{P}(\mathcal{C}) \times
  \mathcal{P}(\mathcal{C}) \\
  \Dnop,
  \end{tikzcd}
\]
where $Y$ denotes the Yoneda embedding.
If we define $\mathcal{F} \to \mathcal{C} \times \Dnop$ by the pullback square
\[
  \begin{tikzcd}
    \mathcal{F} \arrow{r} \arrow{d} &
    \mathcal{P}(\mathcal{C})^{\Delta^{1}} \arrow{d} \\
    \mathcal{C} \times \Dnop \arrow{r}{Y \times F'} &
    \mathcal{P}(\mathcal{C}) \times \mathcal{P}(\mathcal{C}),
  \end{tikzcd}
\]
then we can equivalently describe $\OSPAN_{n}(\mathcal{C};F)$ as the
space of commutative diagrams
\[
  \begin{tikzcd}
    \bbS^{I} \arrow{r} \arrow{dr}{\Pi_{I}} & \mathcal{F} \arrow{d}\\
     & \Dnop.
  \end{tikzcd}
\]
(Here $\mathcal{F} \to \mathcal{C} \times \Dnop$ is the bifibration
(see \S\ref{subsec:bifib}) corresponding to $F$ viewed as a functor
$\mathcal{C}^{\op} \times \Dnop \to \mathcal{S}$.)
From this we obtain an alternative definition of
$\SPAN_{n}(\mathcal{C};F)_{I}$ as the pullback
\[
  \begin{tikzcd}
    \SPAN_{n}(\mathcal{C}; F) \arrow{d} \arrow{r} &
    \OSPAN(\mathcal{F}) \arrow{d} \\
    \SPAN_{n}(\mathcal{C}) \arrow{r} &
    \OSPAN_{n}(\mathcal{C}) \times_{\Dnop} \OSPAN_{n}(\Dnop),
  \end{tikzcd}
\]
where the bottom horizontal map is the fibre product of the inclusion
$\SPAN_{n}(\mathcal{C}) \to \OSPAN_{n}(\mathcal{C})$ with the functor
$\Dnop \to \OSPAN_{n}(\Dnop)$ corresponding to $\Pi \colon \widehat{\bbS}
\to \Dnop$.
\end{remark}

\begin{remark}
  Given a functor $F \colon \mathcal{C} \to \Cat^{n}(\mathcal{S})$, it
  follows from Remark~\ref{rmk:spanlocsysfold} that we have
  \[ \Span_{n}(\mathcal{C}; F) \simeq
    U^{n}_{\Seg}\SPAN_{n}(\mathcal{C}; F) \simeq
    \Span_{n}(\mathcal{C}; U^{n}_{\Seg}F).\]
\end{remark}

\begin{lemma}\
  \begin{enumerate}[(i)]
  \item If $F \colon \mathcal{C}^{\op} \to \Seg_{n-1}(\mathcal{S})$ lands in the full
    subcategory $\Cat_{(\infty,n)}$ of complete $n$-fold Segal spaces,
    then $\Span_{n}(\mathcal{C};F)$ is a complete Segal space.
  \item If $F$ is a functor from $\mathcal{C}^{\op}$ to symmetric
    monoidal  $n$-fold Segal spaces, then
    $\Span_{n}(\mathcal{C};F)$ is symmetric monoidal.
  \item If $F$ is a functor from $\mathcal{C}^{\op}$ to
    $(\infty,n)$-categories with adjoints, then
    $\Span_{n}(\mathcal{C};F)$ has adjoints.
   \item If $F$ is a functor from $\mathcal{C}^{\op}$ to
    symmetric monoidal $(\infty,n)$-categories with duals, then
    $\Span_{n}(\mathcal{C};F)$ has duals.
  \end{enumerate}
\label{lm:spandualadjoints}
\end{lemma}
\begin{proof}
  In case (i), $F' \colon \Dnop \to \mathcal{P}(\mathcal{C})$ is a
  complete $n$-fold Segal object of $\mathcal{P}(\mathcal{C})$, so
  $\Span_{n}(\mathcal{P}(\mathcal{C}); F')$ is a complete $n$-fold
  Segal space by \cite{spans}*{Proposition 9.2}. The $n$-fold Segal
  space $\Span_{n}(\mathcal{C};F)$ is therefore complete as the limit of
  a diagram of complete objects, computed in $n$-fold Segal spaces, is
  complete. Similarly, in case (ii) $F'$ is a symmetric monoidal
   $n$-fold Segal object in
  $\mathcal{P}(\mathcal{C})$, and so
  $\Span_{n}(\mathcal{P}(\mathcal{C}); F')$ is symmetric monoidal
   by \cite{spans}*{Proposition 13.1}. Moreover, the
  functors in the pullback square defining $\Span_{n}(\mathcal{C};F)$
  are naturally symmetric monoidal, and the forgetful functor from symmetric
  monoidal $n$-fold Segal spaces to $n$-fold Segal spaces preserves
  limits. Parts (iii) and (iv) follow similarly using
  Corollary~\ref{cor:hasdualspb} together with \cite{spans}*{Theorem
    13.3 and Corollary 13.4}.
\end{proof}

\begin{propn}\label{propn:SpanCFmaps}
  Suppose $\xi,\eta$ are objects of $\Span_{n}(\mathcal{C}; F)$
  corresponding to pairs $(x \in \mathcal{C}, \xi \in F(x))$, $(y \in
  \mathcal{C}, \eta \in F(y))$. If we define $F_{\xi,\eta}\colon \mathcal{C}_{/x,y} \to
  \Seg_{n-1}(\mathcal{S})$ to be the functor that takes $x
  \from z \to y$ to the pullback
  \[
    \begin{tikzcd}
    F_{\xi,\eta}(z) \arrow{d}\arrow{rr} & & F(z)_{1}\arrow{d}\\
    \{(\xi,\eta)\} \arrow{r} & F(x)_{0} \times F(y)_{0} \arrow{r} &
    F(z)_{0} \times F(z)_{0},
  \end{tikzcd}
\]
  then there is a natural equivalence
  of $(n-1)$-fold Segal spaces
  \[ \Span_{n}(\mathcal{C}; F)(\xi,\eta) \simeq
    \Span_{n-1}(\mathcal{C}_{/x,y}; F_{\xi,\eta}).\]
\end{propn}

\begin{proof}
  From the definition of $\Span_{n}(\mathcal{C};F)$ as a pullback it
  follows that we have a pullback square
  \[
    \begin{tikzcd}
      \Span_{n}(\mathcal{C}; F)(\xi, \eta) \arrow{r} \arrow{d}
      & \Span_{n}(\mathcal{P}(\mathcal{C}); F')(\xi',\eta') \arrow{d}
      \\
      \Span_{n}(\mathcal{C})(x,y) \arrow{r} & \Span_{n}(\mathcal{P}(\mathcal{C}))(Y(x),Y(y)),
    \end{tikzcd}
  \]
  where $\xi'$ is the morphism $Y(x) \to F_{0,\ldots,0}$ corresponding
  to $\xi \in F_{0,\ldots,0}(x)$, and similarly for $\eta'$.
  By \cite{spans}*{Proposition 9.3} we can identify
  $\Span_{n}(\mathcal{P}(\mathcal{C}); F')(\xi',\eta')$ with
  $\Span_{n-1}(\mathcal{P}(\mathcal{C}); F'_{\xi',\eta'})$, where
  $F'_{\xi',\eta'}$ is the $(n-1)$-fold Segal object in
  $\mathcal{P}(\mathcal{C})$ defined as the pullback
  \[
    \begin{tikzcd}
      F'_{\xi',\eta'} \arrow{r} \arrow{d} & F_{1}\arrow{d} \\
      Y(x) \times Y(y) \arrow{r}{\xi' \times \eta'} & F_{0} \times F_{0}.
    \end{tikzcd}
  \]
  Here $F'_{\xi',\eta'}$ is naturally a functor $\simp^{n-1,\op} \to
  \mathcal{P}(\mathcal{C})_{/Y(x)\times Y(y)}$, and so by
  Lemma~\ref{lem:SpanFslice} we can equivalently identify this with
  $\Span_{n-1}(\mathcal{P}(\mathcal{C})_{/Y(x)\times Y(y)};
  F'_{\xi',\eta'})$, compatibly with the identification of
  $\Span_{n}(\mathcal{P}(\mathcal{C}))(Y(x),Y(y))$ with
  $\Span_{n-1}(\mathcal{P}(\mathcal{C})_{/Y(x) \times Y(y)})$ from
  \cite{spans}*{Proposition 8.3}. We thus have a pullback square
  \[
    \begin{tikzcd}
      \Span_{n}(\mathcal{C}; F)(\xi, \eta) \arrow{r} \arrow{d}
      & \Span_{n-1}(\mathcal{P}(\mathcal{C})_{/Y(x)\times Y(y)}; F'_{\xi',\eta'}) \arrow{d}
      \\
      \Span_{n-1}(\mathcal{C}_{/x,y}) \arrow{r} &
      \Span_{n-1}(\mathcal{P}(\mathcal{C})_{/Y(x) \times Y(y)}).
    \end{tikzcd}
  \]
  The canonical functor
  $\mathcal{P}(\mathcal{C}_{/x,y}) \to
  \mathcal{P}(\mathcal{C})_{/Y(x)\times Y(y)}$ is an equivalence
  (since $\mathcal{C}_{/x,y} \to \mathcal{C}$ is the right fibration
  for $Y(x) \times Y(y)$), and under this equivalence the functor
  $F'_{\xi',\eta'}$ corresponds to $(F_{\xi,\eta})'$. Our pullback
  square is therefore equivalent to that defining
  $\Span_{n-1}(\mathcal{C}_{/x,y}; F_{\xi,\eta})$, as required.
\end{proof}

\begin{remark}
  In the special case where $\mathcal{C}$ has a terminal object $*$
  and $x \simeq y \simeq *$, so that $\xi$ and $\eta$ are objects of
  $F(*)$, we can identify $F_{\xi,\eta}$ with the functor
  $\mathcal{C}^{\op} \to \Seg_{n-1}(\mathcal{S})$ taking
  $c \in \mathcal{C}$ to the mapping $(n-1)$-fold Segal space
  $F(c)(f^{*}\xi, f^{*}\eta)$ , where $f$ denotes the unique map
  $c \to *$.
\end{remark}

\subsection{Spans with Coefficients in Cospans}\label{subsec:spancoeffcospan}
Suppose $\mathcal{C}$ is an \icat{} with pullbacks, and consider a
functor $F \colon \mathcal{C}^{\op} \to \Catpo$ to the
\icat{} of small \icats{} with pushouts. Then we have a functor
$\COSPAN_{n}(F) \colon \mathcal{C}^{\op} \to \Cat^{n}(\mathcal{S})$,
and we can consider the $n$-uple Segal space
$\SPAN_{n}(\mathcal{C}; \COSPAN_{n}(F))$. Our goal in this subsection
is to give a simpler description of this $n$-uple Segal space:
\begin{propn}\label{propn:spancospandesc}
  Let $\mathcal{F} \to \mathcal{C}^{\op}$ be the cocartesian fibration
  corresponding to a functor $F \colon \mathcal{C}^{\op} \to
  \Catpo$. Then there is a natural equivalence
  \[ \SPAN_{n}(\mathcal{C}; \COSPAN_{n}(F)) \simeq \COSPAN_{n}(\mathcal{F}).\]
\end{propn}

Our starting point is the following description of spans with
coefficients in cospans:
\begin{propn}\label{propn:Ospancospanbifib}
  Given $\phi \colon \mathcal{C}^{\op} \to \CatI^{\txt{po}}$ with
  corresponding cocartesian fibration $\mathcal{F} \to
  \mathcal{C}^{\op}$, then $\OSPAN_{n}(\mathcal{C};
  \OCOSPAN_{n}(\phi))_{I}$ is equivalent to the space of diagrams of
  the form
  \csquare{\Twl(\bbS^I) \times_{\Dnop}
    \hbbS_{n,\op}}{\mathcal{F}}{\bbS^{I,\op}}{\mathcal{C}^{\op}}{\alpha}{}{}{\gamma^{\op}}
 such that $\alpha$ takes every morphism of $\Twl(\bbS^{I})
\times_{\Dnop} \hbbS_{n,\op}$ that lies over a cartesian morphism in
$\hbbS_{n,\op}$ to a cocartesian morphism in $\mathcal{F}$, 
  where $\hbbS_{n,\op} \to \Dnop$ is the cartesian fibration for $I
  \mapsto \bbS^{I,\op}$. Here $\Twl(\bbS^{I})$ denotes the left
  fibration version of the twisted arrow category of $\bbS^{I}$; see
  Definition~\ref{def:Tw}.
\end{propn}

\begin{proof}
  Let $\mathcal{X} \to \mathcal{C} \times \Dnop$ be the bifibration
  corresponding to $(c,I)\mapsto \OCOSPAN_{n}^{+}(\phi(c))_{I}$. Then
  by Remark~\ref{rmk:OSPANcoeffbifib} we can identify $\OSPAN_{n}(\mathcal{C};
  \OCOSPAN_{n}^{+}(\phi))_{I}$ with the space of commutative diagrams
\[
\left\{  \begin{tikzcd}
    \bbS^{I} \arrow{r} \arrow{dr}{\Pi_{I}} & \mathcal{X} \arrow{d}\\
     & \Dnop.
  \end{tikzcd} \right \}
\simeq
\left \{
\begin{tikzcd}
    \bbS^{I} \arrow{r} \arrow{dr} \arrow[bend right=20]{ddr}[below left]{\Pi_{I}} & \mathcal{X} \arrow{d}\\
     & \mathcal{C} \times \Dnop \arrow{d} \\
     & \Dnop
   \end{tikzcd}
 \right\}
\]
Now Corollary~\ref{cor:bifibseclfib} identifies this with the space of
commutative diagrams
\[ \left \{
    \begin{tikzcd}
      \Twl(\bbS^{I}) \arrow{r} \arrow{d} & \mathcal{X}^{\ell}
      \arrow{d} \\
      \bbS^{I,\op} \times \bbS^{I} \arrow{r} \arrow{d} &
      \mathcal{C}^{\op} \times \Dnop  \arrow{d} \\
      \bbS^{I} \arrow{r}{\Pi_{I}} & \Dnop.
    \end{tikzcd}
  \right\}
\]
Here $\mathcal{X}^{\ell} \to \mathcal{C}^{\op} \times \Dnop$ is the
underlying left fibration of the cocartesian fibration for the functor
$(c,I)\mapsto \Fun(\bbS^{I}, \phi(c))$, and so
Corollary~\ref{cor:Funfib} identifies this space with that of
commutative squares
\[ \left \{
    \begin{tikzcd}
      \Twl(\bbS^{I}) \times_{\Dnop} \hbbS_{n,\op} \arrow{r}{\alpha} \arrow{d}
      & \mathcal{F} \arrow{d} \\
      \bbS^{I,\op} \arrow{r} & \mathcal{C}^{\op},
    \end{tikzcd} \right \}
\]
such that $\alpha$ takes every morphism of $\Twl(\bbS^{I})
\times_{\Dnop} \hbbS_{n,\op}$ that lies over a cartesian morphism in
$\hbbS_{n,\op}$ to a cocartesian morphism in $\mathcal{F}$.
\end{proof}

\begin{notation}
  We use the abbreviation
  \[ \CbbS^{I} := \Twl(\bbS^{I}) \times_{\Dnop} \hbbS_{n,\op}.\]
\end{notation}

\begin{cor}
  $\SPAN_{n}(\mathcal{C}; \COSPAN_{n}(\phi))_{I}$ is the space of commutative
  diagrams
  \csquare{\Twl(\bbS^I) \times_{\Dnop}
    \hbbS_{n,\op}}{\mathcal{F}}{\bbS^{I,\op}}{\mathcal{C}^{\op}}{\alpha}{}{}{\gamma^{\op}}
  where
  \begin{enumerate}[(1)]
  \item $\gamma \colon \bbS^{I} \to \mathcal{C}$ is cartesian, \ie{}
    is a right Kan extension of its restriction to $\bbL^{I}$,
  \item $\alpha$ takes every morphism of $\Tw^{\ell}(\bbS^I) \times_{\Dnop}
    \hbbS_{n,\op}$ that lies over a cartesian morphism in $\hbbS_{n,\op}$ to a
    cocartesian morphism in $\mathcal{F}$,
  \item for every morphism $i \colon A \to B$ in $\bbS^{I}$, the
    diagram
    \[ \bbS^{\pi_{I}(B),\op} \simeq \{i\} \times_{\Dnop} \hbbS_{n,\op} \to
      \{\gamma(A)\} \times_{\mathcal{C}^{op}} \mathcal{F} \simeq
      \phi(\gamma(A))\]
    is cocartesian, \ie{} is a left Kan extension of its restriction
    to $\bbL^{\pi_{I}(B),\op}$.\qed
  \end{enumerate}
\end{cor}

In order to use this description to prove
Proposition~\ref{propn:spancospandesc} we need to relate diagrams
of shape $\CbbS^{I}$ to diagrams of shape $\bbS^{I,\op}$ in
$\mathcal{F}$. This we will do in two steps, using the following
explicit description of the category $\CbbS^{n}$:
\begin{lemma}\ 
  \begin{enumerate}[(i)]
  \item The category $\bbS^{n} \times_{\simp^{\op}} \hbbS_{1,\op}$ is
    equivalent to the partially ordered set of quadruples of integers $(a,b,c,d)$,
    $0 \leq a \leq b \leq c \leq d \leq n$, where
    $(a,b,c,d) \leq (a',b',c',d')$ \IFF{}
    \[ a \leq a' \leq b' \leq b \leq c \leq c' \leq d' \leq d.\]
    This corresponds to a cartesian morphism in $\hbbS_{1,\op}$ \IFF{}
    $b'=b, c'=c$, and the projection to $\bbS^{n}$ is given by
    $(a,b,c,d) \mapsto (a,d)$.
  \item The category $\Twl(\bbS^{n})$ is equivalent to the
    partially ordered set of quadruples of integers $(a,b,c,d)$, $0
    \leq a \leq b \leq c \leq d \leq n$, where
    $(a,b,c,d) \leq (a',b',c',d')$ \IFF{}
    \[ a' \leq a \leq b \leq b' \leq c' \leq c \leq d \leq d',\]
    \ie{} the opposite of the partially ordered set in (i). The
    projections $\Twl(\bbS^{n}) \to \bbS^{n,\op},\bbS^{n}$ are given
    by $(a,b,c,d) \mapsto (a,d), (b,c)$, respectively.
  \item The category $\CbbS^{n} \simeq \Tw^{\ell}(\bbS^{n}) \times_{\simp^{\op}}
    \hbbS_{1,\op}$ is equivalent to the partially ordered set of
    sextuples of integers $(a,b,c,d,e,f)$ where $0 \leq a \leq b \leq c \leq d
    \leq e \leq f \leq n$ where $(a,b,c,d,e,f) \leq
    (a',b',c',d',e',f')$ \IFF{}
    \[ a' \leq a \leq b \leq b' \leq c' \leq c \leq d \leq d' \leq e'
      \leq e \leq f \leq f'.\]
    This corresponds to a cartesian morphism in $\hbbS_{1,\op}$ \IFF{}
    $c'=c, d'=d$. The projections to $\Tw^{\ell}(\bbS^{n})$ and
    $\bbS^{n} \times_{\simp^{\op}} \hbbS_{1,\op}$ are given by
    \[ (a,b,c,d,e,f) \mapsto (b,c,d,e),\quad (a,b,d,e),\]
    respectively.
  \end{enumerate}
\end{lemma}
\begin{proof}
  Since $\hbbS_{1,\op} \to \Dnop$ is the cartesian fibration for the
  functor $[n] \in \simp \mapsto \bbS^{n,\op} \in \Cat$, the category
  $\hbbS_{1,\op}$ has objects pairs $([n], (i,j))$ with
  $0 \leq i \leq j \leq n$, with a morphism
  $([n], (i,j)) \to ([m], (i',j'))$ given by a morphism
  $\phi \colon [m] \to [n]$ in $\simp$ such that
  $(i,j) \leq (\phi(i'), \phi(j'))$ in $\bbS^{n,\op}$, \ie{}
  $(\phi(i'), \phi(j')) \leq (i,j)$ in $\bbS^{n}$, or
  $i \leq \phi(i') \leq \phi(j') \leq j$.

  On the other hand, the functor $\Pi_{n} \colon \bbS^{n}\to \Dop$ takes $(i,j)$ to
  $[j-i]$, so an object of the fibre product
  $\bbS^{n} \times_{\Dop} \hbbS_{1,\op}$ is a pair $((i,j), (i',j'))$
  with $0 \leq i \leq j \leq n$ and $0 \leq i' \leq j' \leq
  j-i$. Identifing this with the quadruple $(i,i'+i,j'+i,j)$ we get a
  bijection between the objects of
  $\bbS^{n} \times_{\Dop} \hbbS_{1,\op}$ and the set of quadruples
  $(a,b,c,d)$ with $0 \leq a \leq b \leq c \leq d \leq n$.

  A morphism $((i,j),(i',j')) \to ((k,l),(k',l'))$ is unique if it
  exists, and corresponds to the inequalities $(i,j) \leq (k,l)$ and
  $(k'+k-i,l'+k-i) \leq (i',j')$ (since the corresponding inclusion
  $[l-k] \hookrightarrow [j-i]$ in $\simp$ is given by
  $t \mapsto t+k-i$), \ie{}
  \[ i \leq k \leq l \leq j, \qquad k'+k-i \leq i' \leq j' \leq l'+k-i,\] 
  which we can rewrite as
  \[ i \leq k \leq k'+k \leq i'+i \leq j'+j \leq l'+k \leq l \leq j.\]
  Equivalently, there is a unique morphism $(a,b,c,d) \to
  (a',b',c',d')$ \IFF{} $a \leq a' \leq b' \leq b \leq c \leq c' \leq
  d' \leq d$, as required to prove (i).

  To prove (ii), observe that an object of $\Twl(\bbS^{n})$ is a
  morphism $(i,j) \to (i',j')$ in $\bbS^{n}$, which we can identify
  with a quadruple $(i,i',j',j)$ with
  $0 \leq i \leq i' \leq j' \leq j \leq n$. Now a morphism from
  $(i,j) \to (i',j')$ to $(k,l) \to (k',l')$ in $\Twl(\bbS^{n})$ is a
  commutative diagram
  \[
    \begin{tikzcd}
      (i,j) \arrow{d} & (k,l) \arrow{d} \arrow{l} \\
      (i',j') \arrow{r} & (k',l'),
    \end{tikzcd}
  \]
  which corresponds to the inequalities
  \[ i \leq i' \leq j' \leq j,\qquad i' \leq k' \leq l' \leq j'\]
  \[ k \leq i \leq j \leq l,\qquad k \leq k' \leq l' \leq l,\]
  which we can combine into the single chain of inequalities
  \[ k \leq i \leq i' \leq k' \leq l' \leq j' \leq j \leq l,\]
  which proves (ii).

  To prove (iii), observe that the fibre product $\Twl(\bbS^{n})
  \times_{\Dop} \hbbS_{1,\op}$ is equivalently the fibre product
  $\Twl(\bbS^{n}) \times_{\bbS^{n}} (\bbS^{n} \times_{\Dop}
  \hbbS_{1,\op})$ of the categories considered in (i) and (ii). We can
  therefore identify an object of this category with a pair
  $((a,b,c,d), (i,j,k,l))$ where $(b,c) = (i,l)$, or equivalently a
  sextuple $(a,b,j,k,c,d)$ with $0 \leq a \leq b \leq j \leq k \leq c
  \leq d \leq n$. The inequalities in (i) and (ii) also combine to
  give the inequalities in (iii) as the criterion for a morphism to
  exist in this partially ordered set.
\end{proof}

\begin{defn}
  Let $\mathbb{T}^{n}$ denote the partially ordered set of quadruples
  $(a,c,d,f)$ with $0 \leq a \leq c \leq d \leq f \leq n$, where
  $(a,c,d,f) \leq (a',c',d',f')$ \IFF{} $a' \leq a \leq f \leq f'$ and
  $c' \leq c \leq d \leq d'$.  We then define functors
  $\alpha_{n} \colon \CbbS^{n}
  \to \mathbb{T}^{n}$ and
  $\beta_{n} \colon \bbS^{n,\op} \to \mathbb{T}^{n}$ by
  \[ \alpha_{n}(a,b,c,d,e,f) = (a,c,d,f),\]
  \[ \beta_{n}(i,j) = (i,i,j,j). \] We also define
  $\gamma_{n} \colon \mathbb{T}^{n} \to \bbS^{n,\op}$ by
  $(a,c,d,f) \mapsto (a,f)$; note that
  $\gamma_{n} \circ \beta_{n} = \id$.  For
  $I = ([i_{1}],\ldots,[i_{n}]) \in \Dnop$ we set
  $\mathbb{T}^{I} := \prod_{j} \mathbb{T}^{i_{j}}$, and we similarly define
  $\alpha_{I}$, $\beta_{I}$, and $\gamma_{I}$ as products.
\end{defn}

\begin{propn}\label{propn:markedanod}
  Let $C_{n}$ denote the set of morphisms $(a,c,d,f) \to (a',c',d',f')$ in
  $\mathbb{T}^{n}$ such that $c=c',d=d'$, and similarly let $C_{I}$
  denote the product of these morphisms viewed as morphisms in
  $\mathbb{T}^{I}$. Then composition with the functor $\beta_{I}$ induces an equivalence
  \[ \Map_{C/\bbS^{I,\op}}(\mathbb{T}^{I}, \mathcal{F}) \to \Map(\bbS^{I,\op},
    \mathcal{F}),\]
  where $\Map_{C/\bbS^{I,\op}}(\mathbb{T}^{I}, \mathcal{F})$ denotes the space of
  commutative squares
  \csquare{\mathbb{T}^I}{\mathcal{F}}{\bbS^{I,\op}}{\mathcal{C}^{\op}}{f}{\gamma_I}{}{}
  where $f$ takes the morphisms in $C_{I}$ to cocartesian morphisms
  in $\mathcal{F}$.
\end{propn}
\begin{proof}
  To prove this we will show that the morphism of marked simplicial sets
  \[ (\mathrm{N}\bbS^{I,\op})^{\flat} \to (\mathrm{N}\mathbb{T}^{I},
    C_{I})\] is marked anodyne in the cocartesian sense, \ie{} dual to
  that of \cite{HTT}*{Definition 3.1.1.1}. Marked anodyne morphisms
  are closed under the cartesian product of marked simplicial sets by
  \cite{HTT}*{Proposition 3.1.2.3}, so it suffices to prove the case
  $n = 1$. We will do this using a filtration of
  $\mathrm{N}\mathbb{T}^{I}$; to define this it is convenient to first
  make up some terminology and notation:
  \begin{itemize}
  \item We say a simplex of $\mathrm{N}\mathbb{T}^{I}$ is \emph{old}
    if it is contained in the simplicial subset
    $\mathrm{N}\bbS^{I,\op}$, and \emph{new} otherwise.
  \item If $\sigma \colon \Delta^{n} \to \mathrm{N}\mathbb{T}^{I}$ is
    a non-degenerate new simplex, corresponding to a sequence of
    morphisms $A_{0}\xto{f_{1}} A_{1} \to \cdots \to A_{n}$, we define
    $\nu(\sigma)$ to be the integer such that $A_{i} \in
    \beta_{n}(\mathrm{N}\bbS^{I,\op})$ for $i < \nu(\sigma)$ and
    $A_{\nu(\sigma)} \notin \beta_{n}(\mathrm{N}\bbS^{I,\op})$.
  \item If $\sigma$ is a non-degenerate new $n$-simplex as above, we
    say that $\sigma$ is \emph{long} if $\nu(\sigma) > 0$ and the
    morphism $A_{\nu(\sigma)-1} \to A_{\nu(\sigma)}$ is in $C_{n}$,
    and \emph{short} otherwise.
  \item If $\sigma$ is a long new non-degenerate $(n+1)$-simplex then
    we say that $\sigma$ is \emph{associated} to the short new
    non-degenerate $n$-simplex $d_{\nu(\sigma)-1}\sigma$. Observe that
    for every short new non-degenerate $n$-simplex there is a
    \emph{unique} long new non-degenerate $(n+1)$-simplex associated
    to it.
  \end{itemize}
  We let $\mathfrak{F}_{n}$ be the smallest simplicial subset of
  $\mathrm{N}\mathbb{T}^{I}$ containing $\mathfrak{F}_{n-1}$ (where we
  start with $\mathfrak{F}_{-1}$ containing only the old simplices) together
  with the short new non-degenerate $n$-simplices and the long new
  non-degenerate $(n+1)$-simplices. We then have a filtration of
  marked simplicial sets
  \[ \beta_{n}(\mathrm{N}\bbS^{I,\op}) = \mathfrak{F}_{-1} \subseteq
    \mathfrak{F}_{0} \subseteq \cdots \subseteq
    \mathrm{N}\mathbb{T}^{I},\] where we implicitly regard all these
  simplicial sets as marked by those of their edges that lie in
  $C_{n}$. Since $\mathrm{N}\mathbb{T}^{I}$ is the union of the
  simplicial subsets $\mathfrak{F}_{i}$, it suffices to show that
  the morphisms $\mathfrak{F}_{i-1} \hookrightarrow
  \mathfrak{F}_{i}$ are all marked anodyne.
  
  Next, we define a subsidiary filtration
  \[ \mathfrak{F}_{N-1} = \mathfrak{G}_{N,N+1} \subseteq
    \mathfrak{G}_{N,N} \subseteq \cdots \subseteq \mathfrak{G}_{N,0}
    = \mathfrak{F}_{N},\]
  where $\mathfrak{G}_{N,m}$ contains
  $\mathfrak{F}_{N-1}$ together with those short new non-degenerate
  $N$-simplices $\sigma$ such that $\nu(\sigma) \geq m$, as well as their
  associated $(N+1)$-simplices. Then it
  suffices to show that the inclusions $\mathfrak{G}_{N,m}
  \hookrightarrow \mathfrak{G}_{N,m-1}$ are all marked anodyne.
  
  Consider now a short new non-degenerate $N$-simplex $\sigma$ with
  associated $(N+1)$-simplex $\sigma'$. Then we observe that
  \begin{itemize}
  \item $d_{\nu(\sigma)}\sigma' = \sigma$,
  \item $d_{i}\sigma'$ is a long $N$-simplex if $i \neq \nu(\sigma),
    \nu(\sigma)+1$, and so is in $\mathfrak{F}_{N-1}$,
  \item $\nu(d_{\nu(\sigma)+1}\sigma') = \nu(\sigma)+1$, so
    $d_{\nu(\sigma)+1}\sigma'$ lies in $\mathfrak{G}_{N,\nu(\sigma)+1}$.
  \end{itemize}
  Thus we have pushouts \nolabelcsquare{\coprod_{\sigma}
    \Lambda^{N+1}_{m}}{\coprod_{\sigma}
    \Delta^{N+1}}{\mathfrak{G}_{N,m+1}}{\mathfrak{G}_{N,m},} where the
  coproducts are over all short new non-degenerate $N$-simplices
  $\sigma$ such that $\nu(\sigma)=m$. If $m > 0$ then the top
  horizontal morphism is inner anodyne, and if $m = 0$ then for every
  $\sigma$ the edge $0 \to 1$ in $\Lambda^{N+1}_{0}$ is sent to an
  edge of $\mathrm{N}\mathbb{T}^{n}$ that lies in $C_{n}$, hence the
  top horizontal morphism is still marked anodyne.
\end{proof}

Let us also write $\Map_{C/\bbS^{I,\op}}(\CbbS^{I}, \mathcal{F})$ for the space of
commutative squares
\csquare{\CbbS^I}{\mathcal{F}}{\bbS^{I,\op}}{\mathcal{C}^{\op}}{f}{}{}{g^\op}
where $f$ takes the morphisms that lie over cartesian morphisms in
$\hbbS_{n,\op}$ to cocartesian morphisms in $\mathcal{F}$. Then
composition with $\alpha_{I}$ and $\beta_{I}$ give natural maps
\[ \Map_{C/\bbS^{I,\op}}(\CbbS^{I}, \mathcal{F}) \from
\Map_{C/\bbS^{I,\op}}(\mathbb{T}^{I}, \mathcal{F}) \isoto
\Map(\bbS^{I,\op}, \mathcal{F}).\]
Now we define $\Map_{C/\bbS^{I,\op}}^{\txt{cocart}}(\CbbS^{I}, \mathcal{F})$ to be the subspace of such
squares where
\begin{itemize}
\item $g \colon \bbS^{I} \to \mathcal{C}$ is cartesian,
\item for every morphism $i \colon A \to B$ in $\bbS^{I}$, the
  diagram
  \[ \bbS^{\pi_{I}(B),\op} \simeq \{i\} \times_{\Dnop} \hbbS_{n,\op}
    \to \{g(A)\} \times_{\mathcal{C}^{op}} \mathcal{F} \simeq
    \phi(g(A))\] is cocartesian,
\end{itemize}
and we also define
$\Map_{C/\bbS^{I,\op}}^{\txt{cocart}}(\mathbb{T}^{I}, \mathcal{F})$ to
be the subspace of functors that restrict under $\beta_{I}$ to
cocartesian functors $\bbS^{I,\op} \to \mathcal{F}$.
\begin{propn}\label{propn:Tcocartftr}
  The maps given by composition with $\alpha_{I}$ and $\beta_{I}$
  restrict to maps
  \[ \Map_{C/\bbS^{I,\op}}^{\txt{cocart}}(\CbbS^{I}, \mathcal{F}) \from
\Map_{C/\bbS^{I,\op}}^{\txt{cocart}}(\mathbb{T}^{I}, \mathcal{F})
\isoto \Map^{\txt{cocart}}(\bbS^{I,\op}, \mathcal{F}).\]
\end{propn}

We need the well-known description of colimits in a cocartesian
fibration, which we spell out as follows:
\begin{lemma}\label{lem:cocartcolim}
  Suppose $\pi \colon \mathcal{E} \to \mathcal{B}$ is a cocartesian fibration
  and $\mathcal{I}$ is a small \icat{} such that
  \begin{enumerate}[(i)]
  \item $\mathcal{B}$ has colimits of shape $\mathcal{I}$,
  \item each fibre $\mathcal{E}_{b}$ has colimits of shape
    $\mathcal{I}$,
  \item the cocartesian pushforward functor $f_{!} \colon
    \mathcal{E}_{b} \to \mathcal{E}_{b'}$ preserves colimits of shape
    $\mathcal{I}$ for all morphisms $f \colon b \to b'$ in $\mathcal{B}$.
  \end{enumerate}
  Then $\mathcal{E}$ has colimits of shape $\mathcal{I}$. The colimit
  of a diagram $p \colon \mathcal{I} \to \mathcal{E}$ is computed by
  \begin{enumerate}[(1)]
  \item extending $\pi p \colon \mathcal{I} \to \mathcal{B}$ to a
    colimit diagram $q \colon \mathcal{I}^{\triangleright} \to
    \mathcal{B}$,
  \item taking the cocartesian pushforward $p' \colon \mathcal{I} \to
    \mathcal{E}_{q(\infty)}$ of the diagram $p$ along the morphisms
    $q(i) \to q(\infty)$,
  \item computing the colimit of $p'$ in the fibre $\mathcal{E}_{q(\infty)}$.
  \end{enumerate}
\end{lemma}
\begin{proof}
  By \cite{HTT}*{Corollary 4.3.1.11} the assumptions imply that there
  exists a lift
  $\bar{p} \colon \mathcal{I}^{\triangleright} \to \mathcal{E}$ over
  $q$, which is a $\pi$-colimit diagram. Combining
  \cite{HTT}*{Propositions 4.3.1.9 and 4.3.1.10}, we see that this
  $\pi$-colimit is equivalent to the colimit of the pushed-forward
  diagram $p'$ in the fibre $\mathcal{E}_{q(\infty)}$. On the other
  hand, since $q$ is a colimit diagram in $\mathcal{B}$,
  \cite{HTT}*{Proposition 4.3.1.5(2)} shows that $\bar{p}$ is a
  $\pi$-colimit diagram \IFF{} it is a colimit diagram in $\mathcal{E}$.
\end{proof}

\begin{proof}[Proof of Proposition~\ref{propn:Tcocartftr}]
  We must check that composition with $\alpha_{I}$ takes a commutative
  square in
  $\Map_{C/\bbS^{I,\op}}^{\txt{cocart}}(\mathbb{T}^{I}, \mathcal{F})$
  to one in
  $\Map_{C/\bbS^{I,\op}}^{\txt{cocart}}(\CbbS^{I}, \mathcal{F})$.  Since
  $\mathcal{F} \to \mathcal{C}^{\op}$ is the cocartesian fibration
  corresponding to a functor $\mathcal{C}^{\op} \to \CatI^{\txt{po}}$,
  Lemma~\ref{lem:cocartcolim} implies that a commutative square in $\mathcal{F}$
  is a pushout \IFF{} it projects to a pushout square in
  $\mathcal{C}^{\op}$ and its cocartesian pushforward to the fibre
  over the terminal object is a pushout square in that fibre.  This
  implies in particular that composition with
  $\mathcal{F} \to \mathcal{C}^{\op}$ takes cocartesian diagrams in
  $\mathcal{F}$ to cocartesian diagrams in $\mathcal{C}^{\op}$. Thus
  it remains only to show that for every morphism $i \colon A \to B$
  in $\bbS^{I}$, the diagram
  \[ \bbS^{\pi_{I}(B),\op} \simeq \{i\} \times_{\Dnop} \hbbS_{n,\op}
    \to \mathbb{T}^{I} \times_{\bbS^{I,\op}} \{A\}
    \to \{g(A)\} \times_{\mathcal{C}^{op}} \mathcal{F} \simeq
    \phi(g(A))\] is cocartesian. But this diagram is a cocartesian
  pushforward to the fibre $g(A)$ of the diagram
  \[ \bbS^{\pi_{I}(B),\op} \to \bbS^{I,\op} \xto{\beta_{I}}
    \mathbb{T}^{I} \to \mathcal{F},\] which is cocartesian by
  \cite{spans}*{Proposition 5.9}, and is therefore cocartesian, using
  again the description of pushouts in $\mathcal{F}$.
\end{proof}

Consequently we see that the functors $\alpha_{I}$ and $\beta_{I}$
induce a morphism of $n$-uple Segal spaces
\[ \COSPAN_{n}(\mathcal{F}) \to \SPAN_{n}(\mathcal{C};
  \COSPAN_{n}(F)).\]
To see that this is an equivalence, we need the following observation:
\begin{lemma}
  The functor $\alpha_{1}\colon \CbbS^{1} \to \mathbb{T}^{1}$ exhibits
  $\mathbb{T}^{1}$ as the localization of $\CbbS^{1}$ at the morphisms $(0,0,0,0,1,1) \to
  (0,\ldots,0,1)$ and $(0,0,1,1,1,1) \to (0,1,\ldots,1)$.
\end{lemma}
\begin{proof}
  We can depict the partially ordered set $\CbbS^{1}$
  as
  \[\mbox{\small $(0,\ldots,0) \to (0,\ldots,0,1) \from (0,\ldots,0,1,1) \to
    (0,0,0,1,1,1) \from (0,0,1,\ldots,1) \to (0,1,\ldots,1) \from (1,\ldots,1)$}\]
  and $\mathbb{T}^{1}$ as
  \[ (0,0,0,0) \to (0,0,0,1) \to (0,0,1,1) \from (0,1,1,1) \from
    (1,1,1,1).\] The result is clear from this description, since both
  decompose as pushouts of free categories.
\end{proof}

\begin{proof}[Proof of Proposition~\ref{propn:spancospandesc}]
  We have a morphism of $n$-uple Segal spaces \[\COSPAN_{n}(\mathcal{F}) \to \SPAN_{n}(\mathcal{C};
  \COSPAN_{n}(F)).\] To see that this is an equivalence, it suffices to
  show that it is an equivalence on fibres $\COSPAN_{n}(\mathcal{F})_{I} \to \SPAN_{n}(\mathcal{C};
  \COSPAN_{n}(F))_{I}$ where $I = ([i_{1}],\ldots,[i_{n}])$ with
  $i_{j} = 0$ or $1$ for all $j$, which follows from the previous lemma.
\end{proof}

\begin{cor}\label{cor:SPANCOSPAN}
  Let $F \colon \mathcal{C}^{\op} \to \CatI$ be a functor such that
  $F(x)$ has finite colimits for $x\in \mathcal{C}$ and $F(f) \colon
  F(x) \to F(y)$ preserves finite colimits for every morphism $f
  \colon x \to y$ in $\mathcal{C}^{\op}$. Then there is a symmetric
  monoidal equivalence of symmetric monoidal $(\infty,n)$-categories
  \[ \Cospan_{n}(\mathcal{F}) \isoto \Span_{n}(\mathcal{C}; \Cospan_{n}(F)).\]
\end{cor}
\begin{proof}
  Since the functors $\alpha_{I}$ and $\beta_{I}$ are defined as
  cartesian products, we have a commutative diagram of equivalences
  \[
    \begin{tikzcd}
      \Map^{\txt{cocart}}_{C/\bbS^{k,I,\op}}(\CbbS^{k,I}, \mathcal{F})
      \arrow{r}{\sim}& \Map^{\txt{cocart}}_{C/\bbS^{k,\op}}(\CbbS^{k},
    \Map^{\txt{cocart}}_{C/\bbS^{I,\op}}(\CbbS^{I},    \mathcal{F})) \\
    \Map^{\txt{cocart}}_{C/\bbS^{k,I,\op}}(\mathbb{T}^{k,I},\mathcal{F}) \arrow{u}{\sim}
    \arrow{d}[left]{\sim} \arrow{r}{\sim} & \Map^{\txt{cocart}}_{C/\bbS^{k,\op}}(\mathbb{T}^{k},
    \Map^{\txt{cocart}}_{C/\bbS^{I,\op}}(\mathbb{T}^{I}, \mathcal{F})) \arrow{u}[right]{\sim}
    \arrow{d}[right]{\sim} \\
        \Map(\bbS^{k,I,\op}, \mathcal{F}) 
        \arrow{r}{\sim} & 
        \Map(\bbS^{k,\op}, \Map(\bbS^{I,\op}, \mathcal{F})),
    \end{tikzcd}
  \]
  with the notation in the right-hand column interpreted so that it
  makes sense. Setting $k = 1$ and taking the fibres (via the maps
  $\bbS^{0} \amalg \bbS^{0} \to \bbS^{1}$, etc.) at the constant maps
  to the initial object (which we denote with the subscript $(\emptyset,\emptyset)$), we get a commutative diagram of equivalences
  \[
    \begin{tikzcd}
      \Map^{\txt{cocart}}_{C/\bbS^{1,I,\op}}(\CbbS^{1,I}, \mathcal{F})_{(\emptyset,\emptyset)}
      \arrow{r}{\sim}& \Map^{\txt{cocart}}_{C/\bbS^{I,\op}}(\CbbS^{I},    \mathcal{F}) \\
    \Map^{\txt{cocart}}_{C/\bbS^{1,I,\op}}(\mathbb{T}^{1,I},\mathcal{F})_{(\emptyset,\emptyset)} \arrow{u}{\sim}
    \arrow{d}[left]{\sim} \arrow{r}{\sim} &
    \Map^{\txt{cocart}}_{C/\bbS^{I,\op}}(\mathbb{T}^{I}, \mathcal{F})
    \arrow{u}[right]{\sim}
    \arrow{d}[right]{\sim} \\
        \Map(\bbS^{1,I,\op}, \mathcal{F})_{(\emptyset,\emptyset)} 
        \arrow{r}{\sim} & 
        \Map(\bbS^{I,\op}, \mathcal{F}).
    \end{tikzcd}
  \]
  From this we see that on the underlying $n$-fold Segal objects the
  equivalences of Proposition~\ref{propn:spancospandesc} are compatible under
  delooping, \ie{} we have commutative squares of equivalences
  \[
    \begin{tikzcd}
    \Span_{n+1}(\mathcal{C}; \Cospan_{n+1}(F))(\emptyset,\emptyset)
    \arrow{r}{\sim} & \Span_{n}(\mathcal{C}; \Cospan_{n}(F)) \\
    \Cospan_{n+1}(\mathcal{F})(\emptyset,\emptyset) \arrow{r}{\sim}
    \arrow{u}{\sim}& \Cospan_{n}(\mathcal{F}). \arrow{u}{\sim}
  \end{tikzcd}
\]
  It follows that the equivalence $\Span_{n}(\mathcal{C};
  \Cospan_{n}(F)) \isoto \Cospan_{n}(\mathcal{F})$ is symmetric
  monoidal, as required.
\end{proof}

\begin{remark}
  Let $F$ be as above, and suppose
  $\sigma \colon \mathcal{C}^{\op} \to \mathcal{F}$ is a section that
  takes finite limits in $\mathcal{C}$ to colimits in $\mathcal{F}$. Then
  it follows from the equivalence of Corollary~\ref{cor:SPANCOSPAN}
  that $\sigma$ induces a symmetric monoidal functor of
  $(\infty,n)$-categories
  \[ \Span_{n}(\mathcal{C}) \to \Span_{n}(\mathcal{C}; \Cospan_{n}(F)).\]
  \label{rmk:SPANCOSPANsection}
\end{remark}

\subsection{Review of Higher Morita  Categories}\label{subsec:revmorita} 
In this subsection we will briefly recall the definition of the higher
Morita category of $\mathbb{E}_{n}$-algebras in an
$\mathbb{E}_{n}$-monoidal \icat{}, as constructed in \cite{nmorita}.

\begin{defn}
  A \emph{$\Dn$-monoidal \icat{}} is a cocartesian fibration
  $\mathcal{V}^{\otimes} \to \Dnop$ such that the corresponding
  functor $\Dnop \to \CatI$ is an $n$-uple monoid in $\CatI$, in the
  sense of Definition~\ref{defn:uplemonoid}. We will abuse notation by
  writing $\mathcal{V}$ for $\mathcal{V}^{\otimes}_{(1,\ldots,1)}$ and
  just saying that ``$\mathcal{V}$ is a $\Dn$-monoidal \icat{}''.
\end{defn}

\begin{remark}
  As a special case of Remark~\ref{rmk:EkvsDn}, the notion of $\Dn$-monoidal
  \icat{} is equivalent to that of $\mathbb{E}_{n}$-monoidal \icat{}
  considered in \cite{HA}.
\end{remark}

\begin{notation}
  We say a morphism in $\simp^{n} := \simp^{\times n}$ is inert or
  active if each of its components in $\simp$ is inert or active,
  respectively, in the sense of Definition~\ref{def:inert}.
\end{notation}

\begin{defn}
  Suppose $\mathcal{V}$ is a $\Dn$-monoidal \icat{}. Then a
  \emph{$\simp^{n,\op}$-algebra} in $\mathcal{V}$ is a section
  \[
    \begin{tikzcd}
      \mathcal{V}^{\otimes} \arrow{d} \\
      \Dnop \arrow[bend left=40]{u}{A}
    \end{tikzcd}
  \]
  such that $A$ takes inert morphisms in $\Dnop$ to cocartesian
  morphisms in $\mathcal{V}^{\otimes}$.
\end{defn}

\begin{remark}
  It follows from the Dunn--Lurie additivity theorem that
  $\Dnop$-algebras in $\mathcal{V}$ are the same thing as
  $\mathbb{E}_{n}$-algebras; see \cite{nmorita}*{Corollary A.27}.
\end{remark}

\begin{defn}
  More generally, if $\mathcal{O}$ is an \icat{} over $\Dnop$ with a
  suitable notion of inert morphisms living over the inert morphsims
  in $\Dnop$, we can define \emph{$\mathcal{O}$-algebras} in a
  $\Dn$-monoidal \icat{} $\mathcal{V}$ as commutative triangles
  \[
    \begin{tikzcd}
      \mathcal{O} \arrow{rr}{A} \arrow{dr} & & \mathcal{V}^{\otimes}
      \arrow{dl}\\
      & \Dnop,
    \end{tikzcd}
  \]
  where $A$ takes inert morphisms in $\mathcal{O}$ to cocartesian
  morphisms in $\mathcal{V}^{\otimes}$. In particular, this definition
  makes sense if $\mathcal{O}$ is a \emph{(generalized)
    $\Dn$-$\infty$-operad}, in the sense of \cite{nmorita}*{Definition
    5.8}. We write $\Alg^{n}_{\mathcal{O}}(\mathcal{V})$ for the full
  subcategory of $\Fun_{/\Dnop}(\mathcal{O}, \mathcal{V}^{\otimes})$
  spanned by the $\mathcal{O}$-algebras.
\end{defn}

\begin{ex}\label{ex:DnIop}
  For any object $I$ in $\Dnop$, the slice category
  $\DnIop := ((\simp^{ \times n})_{/I})^{\op}$ is a generalized
  $\Dn$-$\infty$-operad via the forgetful functor. Algebras for
  $\Dop_{/[1]}$ in $\mathcal{V}$ correspond to a pair of associative
  algebras and a bimodule between them, while $\DnIop$-algebras where
  $I = (1,\ldots,1,0,\ldots,0)$ with $k$ 1's correspond to $k$-fold
  iterated bimodules in $\mathbb{E}_{n-k}$-algebras in $\mathcal{V}$;
  these are the $k$-morphisms in the higher Morita category. On the
  other hand, $\Dop_{/[2]}$-algebras correspond to a triple of
  associative algebras $A_{0},A_{1},A_{2}$, together with
  $A_{i}$-$A_{j}$-bimodules $M_{ij}$ for all $0 \leq i < j \leq 2$,
  as well as an $A_{1}$-bilinear map $M_{01} \otimes M_{12} \to M_{02}$, or
  equivalently a map $M_{01}\otimes_{A_{1}} M_{12}\to M_{02}$ of
  $A_{0}$-$A_{2}$-bimodules.
\end{ex}

\begin{ex}
  Let $\bbLambda_{/[n]}$ denote the full subcategory of $\simp_{/[n]}$
  spanned by the morphisms $\phi \colon [m] \to [n]$ such that
  $\phi(i+1)-\phi(i) \leq 1$ for all $i$. For
  $I = (i_{1},\ldots,i_{n})$ in $\simp^{n}$, we set
  $\LnIop := \prod_{t = 1}^{n}\bbLambda_{/[i_{t}]}^{\op}$; then
  $\LnIop$ is a generalized $\Dn$-$\infty$-operad via the forgetful
  functor to $\Dnop$. A $\bbLambda^{\op}_{/[2]}$-algebra in
  $\mathcal{V}$ corresponds to a triple of assocative algebras,
  $A_{0}$, $A_{1}$, $A_{2}$, together with an $A_{0}$-$A_{1}$-bimodule
  $M_{01}$ and an $A_{1}$-$A_{2}$-bimodule $M_{12}$.
\end{ex}

\begin{defn}
  If $S$ is some class of \icats{}, we say that a $\Dn$-monoidal
  \icat{} $\mathcal{V}^{\otimes}$ is \emph{compatible with $S$-shaped
    colimits} if $\mathcal{V}$ has $S$-shaped colimits and the tensor
  product functor
  \[ \mathcal{V}^{\times 2} \simeq
    \mathcal{V}^{\otimes}_{(2,1,\ldots,1)} \to \mathcal{V}\] coming
  from the map $(2,1\ldots,1) \to (1,\ldots,1)$ preserves $S$-shaped
  colimits in each variable. (The $n$ tensor products obtained in this
  way by permuting $(2,1,\ldots,1)$ can all be shown to be equivalent,
  so the definition does not depend on the choice of this map.)
\end{defn}

\begin{defn}
  Let $\tau_{I} \colon \LnIop \to \DnIop$ be the
  inclusion. Composition with $\tau_{I}$ induces a functor
  $\tau_{I}^{*} \colon \Alg^{n}_{\DnIop}(\mathcal{V}) \to
  \Alg^{n}_{\LnIop}(\mathcal{V})$. If $\mathcal{V}$ is compatible with
  $\Dop$-colimits, then this functor has a fully faithful left adjoint
  $\tau_{I,!}$. We say a $\DnIop$-algebra is \emph{composite} if it is
  in the essential image of this functor, or equivalently if the
  counit map $\tau_{I,!}\tau_{I}^{*}A \to A$ is an equivalence.
\end{defn}

\begin{ex}
  A $\simp^{\op}_{/[2]}$-algebra as in Example~\ref{ex:DnIop} is
  composite \IFF{} the morphism
  $M_{01} \otimes_{A_{1}} M_{12} \to M_{02}$ is an equivalence, \ie{}
  \IFF{} the $\simp^{\op}_{/[2]}$-algebra presents $M_{02}$ as the
  \emph{composite} of $M_{01}$ and $M_{12}$ in the higher Morita
  category.
\end{ex}

\begin{defn}\label{defn:ALGn}
  Suppose $\mathcal{V}$ is a $\Dn$-monoidal \icat{} compatible with
  $\Dnop$-colimits. There is a functor $\Dnop \to \CatI$ taking $I$ to
  $\Alg_{\DnIop}^{n}(\mathcal{V})$ and a morphism $\phi \colon I \to J$ to the
  functor given by composition with the functor $\DnIop \to
  \Dnop_{/J}$ defined by composing with $\phi$. We let 
  \[\overline{\mathfrak{ALG}}_{n}(\mathcal{V}) \to \Dnop\] be the corresponding cocartesian
  fibration, and write
  $\mathfrak{ALG}_{n}(\mathcal{V})$ for the full subcategory of
  $\overline{\mathfrak{ALG}}_{n}(\mathcal{V})$ spanned by the
  composite $\DnIop$-algebras for all $I$.
\end{defn}

We can now state the main result of \cite{nmorita}:
\begin{thm}[\cite{nmorita}*{Theorem 5.31}]
  For $\mathcal{V}$ as above, the restricted functor $\mathfrak{ALG}_{n}(\mathcal{V}) \to \Dnop$
  is a cocartesian fibration, and the corresponding functor is an
  $n$-uple category object in $\CatI$. \qed
\end{thm}

\begin{remark}
  The assumption that $\mathcal{V}$ is compatible with $\Dop$-colimits
  can be weakened to the assumption that $\mathcal{V}$ ``has good
  relative tensor products'' in the sense of
  \cite{nmorita}*{Definition 5.18}. In particular, it is not necessary
  that $\mathcal{V}$ has all simplicial colimits, only those that
  occur when forming relative tensor products. For example, if
  $\mathcal{V}$ is equipped with the cocartesian symmetric monoidal
  structure, then the relative tensor products are given by pushouts,
  and it is enough to assume that $\mathcal{V}$ has finite colimits.
\end{remark}

\begin{remark}
  An $n$-uple category object in $\CatI$ gives, by viewing \icats{} as
  complete Segal spaces, an $(n+1)$-uple Segal space. From this we can
  obtain an $(n+1)$-fold Segal space via
  Proposition~\ref{propn:Useg}.
\end{remark}

\begin{notation}
  We write $\mathfrak{Alg}_{n}(\mathcal{V})$ for the completion of the
  underlying $(n+1)$-fold Segal space
  $U_{\Seg}^{n+1}\mathfrak{ALG}_{n}(\mathcal{V})$ of
  $\mathfrak{ALG}_{n}(\mathcal{V})$. Thus
  $\mathfrak{Alg}_{n}(\mathcal{V})$ is an $(\infty,n+1)$-category; we
  write $\mathfrak{alg}_{n}(\mathcal{V})$ for its underlying
  $(\infty,n)$-category. Equivalently,
  $\mathfrak{alg}_{n}(\mathcal{V})$ is the completion of the
  underlying $n$-fold Segal space of the $n$-uple Segal space
  corresponding to the left fibration obtained by forgetting the
  non-cocartesian morphisms in $\mathfrak{ALG}_{n}(\mathcal{V})$.
\end{notation}

We then have the following results from \cite{nmorita}, which we state
for $\mathfrak{alg}_{n}(\mathcal{V})$, this being the version of the
higher Morita category relevant in this paper:
\begin{thm}[\cite{nmorita}*{Theorem 5.49}]
  $\mathfrak{alg}_{n}(\mathcal{V})(A,B) \simeq
  \mathfrak{alg}_{n-1}(\txt{Mod}_{A,B}(\mathcal{V}))$. \qed
\end{thm}

\begin{cor}
  If $\mathcal{V}$ is an $\mathbb{E}_{n+m}$-monoidal \icat{}, then
  $\mathfrak{alg}_{n}(\mathcal{V})$ is
  $\mathbb{E}_{m}$-monoidal. In particular, if $\mathcal{V}$ is
  symmetric monoidal, so is
  $\mathfrak{alg}_{n}(\mathcal{V})$.
\end{cor}

We now discuss two conjectures that will be relevant to our understanding
of the higher category of derived Poisson stacks.

\begin{conjecture}\label{conj:moritaadjoints}
  Suppose $\mathcal{V}$ is a symmetric monoidal \icat{} compatible
  with $\Dnop$-colimits. Then the symmetric monoidal
  $(\infty,n)$-category $\mathfrak{alg}_{n}(\mathcal{V})$ has duals
  (in the sense of Definition~\ref{defn:hasduals}, \ie{} its objects
  are dualizable and all $i$-morphisms have adjoints for
  $1 \leq i < n$). In particular, all objects of
  $\mathfrak{alg}_{n}(\mathcal{V})$ are fully dualizable.
\end{conjecture}

\begin{remark}\label{rmk:moritaadjoints}
  This conjecture has been proved by Gwilliam and Scheimbauer in
  \cite{GwilliamScheimbauer} for a closely related model
  $\mathfrak{alg}^{FA}_{n}(\mathcal{V})$ of the higher Morita
  $(\infty, n)$-category, defined using factorization algebras. It is expected that there is an equivalence
  $\mathfrak{alg}^{FA}_{n}(\mathcal{V}) \simeq
  \mathfrak{alg}_{n}(\mathcal{V})$ when $\mathcal{V}$ is
  \emph{pointed}, \ie{} the unit of the monoidal structure is the
  initial object, and more generally that there is an equivalence
  \[\mathfrak{alg}^{FA}_{n}(\mathcal{V}) \simeq
    \mathfrak{alg}_{n}(\mathcal{V}_{I/}).\] Since the forgetful
  functor $\mathcal{V}_{I/} \to \mathcal{V}$ induces a symmetric
  monoidal functor
  $\mathfrak{alg}_{n}(\mathcal{V}_{I/}) \to
  \mathfrak{alg}_{n}(\mathcal{V})$, the result of
  Gwilliam--Scheimbauer together with such a hypothetical equivalence
  would imply that $\mathfrak{alg}_{n}(\mathcal{V})$ has duals (since
  duals and adjoints are preserved by any functor, and $i$-morphisms
  in $\mathfrak{alg}_{n}(\mathcal{V})$ for $i < n$ are naturally
  pointed, and so lift uniquely to
  $\mathfrak{alg}_{n}(\mathcal{V}_{I/})$).
\end{remark}

\begin{conjecture}\label{conj:pointedcomplete}
  If $\mathcal{V}$ is a pointed $\mathbb{E}_{n}$-monoidal \icat{}
  (\ie{} the unit is the initial object) then the $(n+1)$-fold Segal
  space $U^{n+1}_{\Seg}\mathfrak{ALG}_{n}(\mathcal{V})$ is complete.
\end{conjecture}

\begin{remark}
  Completeness of an $n$-fold Segal space $\mathcal{X}$ is equivalent
  to completeness of the underlying Segal space
  $\mathcal{X}_{\bullet,0,\ldots,0}$ and of the $(n-1)$-fold Segal
  spaces of maps $\mathcal{X}(x,y)$. In the case of
  $U^{n+1}_{\Seg}\mathfrak{ALG}_{n}(\mathcal{V})$ both the underlying
  Segal space and the $(n-1)$-fold Segal spaces of maps can themselves
  be described as higher Morita categories (pointed if $\mathcal{V}$
  is pointed). By induction, this means that it suffices
  to prove the conjecture in the case $n = 1$.
\end{remark}

\begin{remark}
  It is shown in \cite{ScheimbauerThesis}*{\S 3.2.9} that for a
  pointed monoidal \icat{} the degeneracy map from the space of
  objects of $U_{\Seg}^{2}\mathfrak{ALG}_{1}(\mathcal{V})$ to the
  space of equivalences is surjective on $\pi_{0}$, \ie{} in the
  pointed case every Morita equivalence comes from an equivalence of
  algebras in $\mathcal{V}$. (More precisely, Scheimbauer proves the
  analogue of this statement for the factorization algebra model, but
  the proof also works for the algebraic model.)
  Conjecture~\ref{conj:pointedcomplete} then amounts to the assertion
  that this essentially surjective map is in fact an equivalence.
\end{remark}

\subsection{Iterated Cospans as a Higher Morita
  Category}\label{subsec:cospanmor}
Suppose $\mathcal{C}$ is an \icat{} with finite colimits. Then we can
define an $(\infty,n)$-category $\Cospan_{n}(\mathcal{C})$ of iterated
cospans in $\mathcal{C}$ as in \S\ref{subsec:revspans}. We can also view
$\mathcal{C}$ as a symmetric monoidal \icat{} via coproducts, and
hence define an $(\infty,n)$-category
$\mathfrak{alg}_{n}(\mathcal{C}^{\amalg})$ of
$\mathbb{E}_{n}$-algebras in $\mathcal{C}$. In this subsection we will
show that there is an equivalence
\[ \Cospan_{n}(\mathcal{C}) \simeq
  \mathfrak{alg}_{n}(\mathcal{C}^{\amalg}).\] For $I$ in $\simp^{n}$,
the space $\Cospan_{n}(\mathcal{C})_{I}$ is defined as a subspace of
the underlying space of the \icat{} $\Fun(\bbS^{I}, \mathcal{C})$,
while $\mathfrak{alg}_{n}(\mathcal{C}^{\amalg})_{I}$ is similarly
obtained (before completion) from
$\Alg^{n}_{\DnIop}(\mathcal{C}^{\amalg})$. We will prove the
equivalence of $(\infty,n)$-categories by finding a natural
equivalence of \icats{}
\[ \Alg^{n}_{\DnIop}(\mathcal{C}^{\amalg}) \isoto \Fun(\bbS^{I}, \mathcal{C})\]
and a compatible equivalence between $\LnIop$-algebras and functors
from $\bbL^{I}$.

We first consider the case $n = 1$, which breaks down into three steps:
\begin{enumerate}[(1)]
\item We define a $\simp$-\iopd{} $\BM_{i}$ and a natural map $\Dop_{/[i]} \to
  \BM_{i}$ of generalized $\simp$-\iopds{} such that for any monoidal
  \icat{} $\mathcal{V}$ the induced functor
  \[ \Alg^{1}_{\BM_{i}}(\mathcal{V}) \to
    \Alg^{1}_{\Dop_{/[i]}}(\mathcal{V})\]
  is an equivalence.
\item We define a unital $\simp$-\iopd{} $\BM_{i}^{*}$ and a natural map
  $\BM_{i} \to \BM_{i}^{*}$ such that for any monoidal \icat{}
  $\mathcal{V}$ the induced functor
  \[ \Alg^{1}_{\BM_{i}^{*}}(\mathcal{V}) \to
    \Alg^{1}_{\BM_{i}}(\mathcal{V})_{I/}\]
  is an equivalence.
\item We have a natural equivalence $(\BM_{i}^{*})_{[1]} \simeq \bbS^{i,\op}$, so
  using the non-symmetric version of \cite{HA}*{Proposition 2.4.3.16}
  for any \icat{} $\mathcal{C}$ with coproducts we have an equivalence
  \[ \Alg^{1}_{\BM_{i}^{*}}(\mathcal{C}^{\amalg}) \simeq
    \Fun(\bbS^{i,\op}, \mathcal{C}).\]
\end{enumerate}
The case of $n > 1$ will then be obtained from this by induction.

\begin{defn}
  Let $\txt{BM}_{n}$ be the non-symmetric operad with 
objects $x_{ij}$ where $0 \leq i \leq j \leq n$ and multimorphisms
given by 
\[ \Hom(x_{i_{1}j_{1}},\ldots, x_{i_{k}j_{k}}; x_{st}) =
\begin{cases}
  *, & s = i_{1}, j_{1}=i_{2},\ldots, j_{k}=t, k > 0,\\
  *, & s = t, k = 0,\\
  \emptyset, & \txt{otherwise}.
\end{cases}
\]
If $\txt{BM}^{\otimes}_{n}$ denotes its (non-symmetric) category
of operators, there is a natural map
\[ \Dop_{/[n]} \to \txt{BM}^{\otimes}_{n}\]
over $\Dop$, taking $(i_{0},\ldots,i_{k})$ to
$(x_{i_{0}i_{1}}, \ldots x_{i_{k-1}i_{k}})$.
\end{defn}

\begin{lemma}\label{lem:DoptoBM}
  Composition with the functor $\Dop_{/[n]} \to
  \txt{BM}^{\otimes}_{n}$ induces an equivalence
  \[ \Alg^{1}_{\txt{BM}_{n}}(\mathcal{V}) \to
  \Alg^{1}_{\Dop_{/[n]}}(\mathcal{V})\]
  for all monoidal \icats{} $\mathcal{V}$.
\end{lemma}
\begin{proof}
  Without loss of generality we may assume that $\mathcal{V}$ is
  compatible with colimits, as any monoidal \icat{} is a full
  subcategory of one that is (possibly passing to a larger universe if
  $\mathcal{V}$ is large and not presentable).
  Then we have a commutative square
  \[
    \begin{tikzcd}
      \Alg^{1}_{\txt{BM}_{n}}(\mathcal{V}) \arrow{r} \arrow{d} &
      \Alg^{1}_{\Dop_{/[n]}}(\mathcal{V}) \arrow{d} \\
      \Fun((\txt{BM}_{n})_{[1]}, \mathcal{V}) \arrow{r} &
      \Fun((\Dop_{/[n]})_{[1]}, \mathcal{V}).
    \end{tikzcd}
  \]
  Here the vertical arrows are both monadic right adjoints (\eg{} by
  \cite{enr}*{Corollary A.5.6}), and the
  bottom horizontal arrow is an equivalence, since
  $(\Dop_{/[n]})_{[1]}$ and $(\txt{BM}_{n})_{[1]}$ are both the set of
  pairs $(i,j)$ with $0 \leq i \leq j \leq n$. To show that the top
  horizontal arrow is an equivalence it now suffices by 
  \cite{HA}*{Corollary 4.7.3.16} to check that the natural map between
  free algebras for the two monads is an equivalence. From the formula
  for (non-symmetric) operadic left Kan extensions (see \cite{enr}*{\S
    A.4}) we see that the corresponding monads are given by 
  \[ T_{\Dop_{/[n]}}\Phi(i,i') \simeq
  \coprod_{k=0}^{\infty}\coprod_{\substack{(j_{0},\ldots,j_{k})\\i=j_{0},j_{k}=i'}}
    \Phi(j_{0},j_{1})\otimes\cdots\otimes\Phi(j_{k-1},j_{k}) \simeq
    T_{\txt{BM}_{n}}\Phi,\] which gives the desired equivalence.
\end{proof}

\begin{defn}
  Let $\txt{BM}_{n}^{*}$ be the non-symmetric operad with objects
  $x_{ij}$ with $0 \leq i \leq j \leq n$, and multimorphisms given by 
\[ \Hom(x_{i_{1}j_{1}},\ldots, x_{i_{k}j_{k}}; x_{st}) =
\begin{cases}
  *, & s \leq i_{1} \leq j_{1} \leq i_{2} \leq\cdots \leq j_{k} \leq t, k > 0,\\
  *, & s \leq t, k = 0,\\
  \emptyset, & \txt{otherwise}.
\end{cases}
\]
There is an obvious map
$\pi_{n} \colon \txt{BM}_{n} \to \txt{BM}_{n}^{*}$. We let
$\BM_{n}^{*,\otimes} \to \Dop$ be the category of operators for
$\BM_{n}^{*}$ and denote the induced map $\BM_{n}^{\otimes} \to
\BM_{n}^{*,\otimes}$ also by $\pi_{n}$.
\end{defn}

\begin{remark}
  The operad $\txt{BM}_{n}^{*}$ is unital, \ie{} every object has a
  unique nullary operation. By the non-symmetric variant of
  \cite{HA}*{Proposition 2.3.1.11}, this means that for every monoidal
  \icat{} $\mathcal{V}$ the forgetful functor
  \[\Alg^{1}_{\BM_{n}^{*}}(\mathcal{V})_{I/} \simeq
    \Alg^{1}_{\BM_{n}^{*}}(\mathcal{V}_{I/}) \to
    \Alg^{1}_{\BM_{n}^{*}}(\mathcal{V})\]
  is an equivalence. In particular, the unit $I$, equipped with its
  unique $\BM_{n}^{*}$-algebra structure, is initial in
  $\Alg^{1}_{\BM_{n}^{*}}(\mathcal{V})$.
\end{remark}

\begin{propn}\label{propn:BMtoBMstar}
  The functor $\pi_{n}$ induces an equivalence
  \[ \Alg_{\BM_{n}^{*}}(\mathcal{V}) \to
    \Alg_{\BM_{n}}(\mathcal{V})_{I/}\]
  for every monoidal \icat{} $\mathcal{V}$.
\end{propn}
\begin{proof}
  Since $I \in \Alg_{\BM_{n}}(\mathcal{V})$ is the image of the initial object of
  $\Alg_{\BM_{n}^{*}}(\mathcal{V})$, the functor
  \[\pi_{n}^{*} \colon \Alg_{\BM_{n}^{*}}(\mathcal{V}) \to
  \Alg_{\BM_{n}}(\mathcal{V})\] factors uniquely through the forgetful
  functor from $\Alg_{\BM_{n}}(\mathcal{V})_{I/}$.

  We may again assume, without loss of generality, that $\mathcal{V}$
  is compatible with small colimits. Then we have a commutative square
  \[
  \begin{tikzcd}
    \Alg^{1}_{\txt{BM}_{n}}(\mathcal{V})_{I/} \arrow{r} \arrow{d} &
    \Alg^{1}_{\txt{BM}_{n}^{*}}(\mathcal{V}) \arrow{d} \\
    \Fun((\txt{BM}_{n})_{[1]}, \mathcal{V}) \arrow{r} &
    \Fun((\txt{BM}_{n}^{*})^{\simeq}_{[1]}, \mathcal{V}).
  \end{tikzcd}
  \]
  Here the vertical arrows are both monadic right adjoints; for the
  left one this is because it factors as a composite
  $\Alg^{1}_{\txt{BM}_{n}}(\mathcal{V})_{I/} \to
  \Alg^{1}_{\txt{BM}_{n}}(\mathcal{V}) \to \Fun((\txt{BM}_{n})_{[1]},
  \mathcal{V})$ where both functors are not only monadic right
  adjoints but also preserve sifted colimits. Moreover, the bottom
  horizontal arrow is clearly an isomorphism (note that we use the
  underlying \emph{groupoid} of
  $(\txt{BM}_{n}^{*})_{[1]}$). Therefore, we may again use
  \cite{HA}*{Corollary 4.7.3.16} to show that the top horizontal
  morphism is an equivalence by comparing the free algebras for the
  two monads. The left adjoint to the left-hand functor takes $\Phi$
  to $F_{\txt{BM}_{n}}(\Phi) \amalg I$, where the coproduct is taken
  in $\BM_{n}$-algebras. The formula for
  $F_{\txt{BM}_{n}}$ identifies $I$ with $F_{\txt{BM}_{n}}(\delta)$
  where 
  \[\delta(i,j) \simeq
  \begin{cases}
    I, & j=i+1,\\
    \emptyset, & \txt{otherwise}.
  \end{cases}
  \]
  Since $F_{\txt{BM}_{n}}$ preserves colimits, this means we have
  \[ F_{\txt{BM}_{n}}(\Phi) \amalg I \simeq F_{\txt{BM}_{n}}(\Phi \amalg
  \delta),\]
  and so 
  \[ 
  \begin{split}
    (F_{\txt{BM}_{n}}(\Phi) \amalg I)(i,i') & \simeq
    \coprod_{k=0}^{\infty}
    \coprod_{\substack{(j_{0},\ldots,j_{k})\\i=j_{0},j_{k}=i'}} (\Phi
    \amalg \delta)(j_{0},j_{1})\times\cdots\times (\Phi \amalg
    \delta)(j_{k-1},j_{k})  \\
    & \simeq \coprod_{k = 0}^{\infty}
    \coprod_{\substack{(j_{0},\ldots,j_{k})\\i=j_{0},j_{k}=i'}}
    \coprod_{\substack{S \subseteq \{1,\ldots,k\} \\ j_{s}=j_{s-1}+1, s
        \in S}} \bigotimes_{s \notin S} \Phi(j_{s-1},j_{s}).
  \end{split}
  \]
  In this coproduct we have a term of the form $\Phi(i_{1},j_{1})
  \otimes \cdots \otimes \Phi(i_{k},j_{k})$ whenever 
  \[i \leq i_{1} \leq j_{1} \leq i_{2} \leq \cdots \leq i_{k} \leq
  j_{k}\leq i',\]
  corresponding to $(i, i+1, \ldots i_{1}-1,i_{1},i_{2},i_{2}+1,\ldots,
  j_{k},j_{k}+1,\ldots,j-1,i')$ with $S$ identifying the pairs not
  of the form $(i_{t},j_{t})$. This gives equivalences
  \[ (F_{\txt{BM}_{n}}(\Phi) \amalg I)(i,i') \simeq
  I \amalg \coprod_{k=1}^{\infty} \coprod_{i \leq i_{1} \leq \cdots \leq
    j_{k}\leq i'} \Phi(i_{1},j_{1})\otimes \cdots \otimes
  \Phi(i_{k},j_{k}) \simeq F_{\txt{BM}_{n}^{*}}(i,i'),\]
  where the second equivalence again comes from the formula for
  operadic Kan extensions.
\end{proof}

\begin{cor}\label{cor:DopnbbSeq}
  If $\mathcal{C}$ is an \icat{} with finite coproducts, then  there
  is a natural equivalence of \icats{}
  \[ \Alg^{1}_{\Dop_{/[n]}}(\mathcal{C}^{\amalg}) \simeq
  \Fun(\bbS^{n,\op}, \mathcal{C}).\]
\end{cor}
\begin{proof}
  Since $\mathcal{C}^{\amalg}$ is unital, by
  Lemma~\ref{lem:DoptoBM} and Proposition~\ref{propn:BMtoBMstar} 
  we have natural equivalences
  $\Alg^{1}_{\Dop_{/[n]}}(\mathcal{C}^{\amalg}) \simeq
  \Alg^{1}_{\txt{BM}_{n}^{*}}(\mathcal{C}^{\amalg})$. Using the
  non-symmetric analogue of \cite{HA}*{Proposition 2.4.3.9} 
  it follows that the restriction functor
  $\Alg^{1}_{\txt{BM}_{n}^{*}}(\mathcal{C}^{\amalg}) \to
  \Fun((\BM_{n}^{*})_{[1]}, \mathcal{C})$
  is an equivalence. (Alternatively, we can use \cite{HA}*{Proposition
    2.4.3.9}  together with the formula for symmetrizations of
  ordinary non-symmetric operads from \cite{enr}*{Corollary 3.7.8},
  which does not change the fibre over $[1]$.) 
  Finally, observe that by definition $(\BM_{n}^{*})_{[1]}$ is
  precisely the partially ordered set $\bbS^{n,\op}$.
\end{proof}

\begin{remark}\label{rmk:Lambdaeq}
A variant of the same argument gives a similar equivalence
\[\Alg^{1}_{\bbL^{\op}_{/[n]}}(\mathcal{C}^{\amalg}) \simeq
\Fun(\bbL^{n,\op}, \mathcal{C}),\] compatible with that of
Corollary~\ref{cor:DopnbbSeq} in the sense
that we have a commutative square
\[
  \begin{tikzcd}
    \Alg^{1}_{\simp^{\op}_{/[n]}}(\mathcal{C}^{\amalg}) \arrow{r}{\sim}
    \arrow{d} & \Fun(\bbS^{n,\op}, \mathcal{C}) \arrow{d} \\
    \Alg^{1}_{\bbL^{\op}_{/[n]}}(\mathcal{C}^{\amalg}) \arrow{r}{\sim}
    & \Fun(\bbL^{n,\op}, \mathcal{C}).
  \end{tikzcd}
\]
Passing to left adjoints, we see that under these equivalences the
composite $\Dop_{/[n]}$-algebras in $\mathcal{C}^{\amalg}$ correspond
precisely to the functors $\bbS^{n,\op} \to \mathcal{C}$ that are left
Kan extended from $\bbL^{n,\op}$. Since the equivalences are natural
in $[n] \in \simp$, this implies:
\end{remark}
\begin{cor}
  If $\mathcal{C}$ is an \icat{} with finite colimits, we have a
  natural equivalence
  \[ \mathfrak{ALG}_{1}(\mathcal{C}^{\amalg}) \to
    \COSPAN^{+}_{1}(\mathcal{C})\]
  of category objects in $\CatI$, and so an equivalence
  \[ \mathfrak{alg}_{1}(\mathcal{C}^{\amalg}) \to
    \Cospan_{1}(\mathcal{C}) \]
  of \icats{}.
\end{cor}

We can now prove the general case by induction:
\begin{cor} 
  If $\mathcal{C}$ is an \icat{} with finite coproducts,
  then we have a natural equivalence 
  \[\mathfrak{ALG}_{n}(\mathcal{C}^{\amalg}) \simeq
    \txt{COSPAN}_{n}^{+}(\mathcal{C})\]
  of $n$-uple category objects in $\CatI$.
\end{cor}
\begin{proof}
  If $\mathcal{V}$ is a $\simp^{n+m}$-monoidal \icat{}, then
  $\Alg_{\simp^{m,\op}_{/J}}(\mathcal{V})$ has a natural
  $\simp^{n}$-monoidal structure, given objectwise by the tensor
  product in $\mathcal{V}$, such that there is
  a natural equivalence
  \[ \Alg_{\DnIop}^{n}(\Alg^{m}_{\simp^{m,\op}_{/J}}(\mathcal{V})) \simeq
    \Alg^{n+m}_{\simp^{n+m,\op}_{/(I,J)}}(\mathcal{V}),\]
  by \cite{nmorita}*{Corollary A.77}.

  Suppose we have a natural equivalence
  \[ \Alg_{\simp^{n-1,\op}_{/J}}(\mathcal{C}^{\amalg}) \simeq
    \Fun(\bbS^{J,\op}, \mathcal{C}).\]
  The canonical symmetric monoidal structure on the left-hand side corresponds
  to the cocartesian structure on the right, since this is the unique
  symmetric monoidal structure given objectwise in $\bbS^{J,\op}$ by
  the coproduct in $\mathcal{C}$. For $I = ([i], J)$ in $\simp^{n}$,
  using Corollary~\ref{cor:DopnbbSeq} we then have a natural
  equivalence
  \[
    \begin{split}
    \Alg^{n}_{\DnIop}(\mathcal{C}^{\amalg}) & \simeq
    \Alg^{1}_{\Dop_{/[i]}}(\Alg^{n-1}_{\simp^{n-1,\op}_{/J}}(\mathcal{C}^{\amalg}))
    \simeq \Alg^{1}_{\Dop_{/[i]}}(\Fun(\bbS^{J,\op},
    \mathcal{C})^{\amalg}) \\
    & \simeq \Fun(\bbS^{i,\op}, \Fun(\bbS^{J,\op}, \mathcal{C})) \simeq
    \Fun(\bbS^{I,\op}, \mathcal{C}).      
    \end{split}
\]
  As in Remark~\ref{rmk:Lambdaeq} we also have a compatible equivalence
  $\Alg^{n}_{\LnIop}(\mathcal{C}^{\amalg}) \simeq \Fun(\bbL^{I,\op},
  \mathcal{C})$ and hence a natural equivalence
  \[ \mathfrak{ALG}_{n}(\mathcal{C}^{\amalg})_{I} \simeq
    \COSPAN^{+}_{n}(\mathcal{C})_{I},\]
  as required.
\end{proof}

Passing to underlying $n$-fold Segal spaces we get, since the
symmetric monoidal structures are defined by delooping in both cases:
\begin{cor} \label{cor:cospanmorita}
  If $\mathcal{C}$ is an \icat{} with finite coproducts, then we have
  an equivalence of symmetric monoidal $(\infty,n)$-categories
  \[\mathfrak{alg}_{n}(\mathcal{C}^{\amalg}) \simeq
  \txt{Cospan}_{n}(\mathcal{C}).\]
\end{cor}

\section{Higher Categories of Coisotropic Correspondences}

Our goal in this section is to introduce the notion of (iterated)
coisotropic correspondences, and to construct higher categories where
these are the (higher) morphisms. In \S\ref{sect:derivedstacks} we
give a brief outline of the theory of formal localization in derived
algebraic geometry, as developed in \cite{CPTVV}. We then review the
notions of Poisson structures on derived stacks and coisotropic
structures on morphisms of derived stacks, also from \cite{CPTVV}, in
\S\ref{sec:poissonstr}.  We will avoid going into the technical
details of the various constructions, and we refer to \cite{CPTVV} and
to \cite{PantevVezzosi} for a more complete and precise treatment of
the subject. In \S\ref{subsec:cois} we first define coisotropic
correspondences between derived Poisson stacks, and then use the
results of the previous section to construct $(\infty,n)$-categories
of derived Poisson stacks and iterated coisotropic correspondences. We
finish by briefly discussing the expected relation of our higher
categories to higher categories of symplectic derived stacks in
\S\ref{sect:LagrangianCorrespondences}.

\subsection{Derived Stacks and Formal Localization}
\label{sect:derivedstacks}

We fix a base field $k$ of characteristic $0$. Let $\cdga^{\leq 0}$
denote the \icat{} of commutative algebras in non-positively graded
cochain complexes of $k$-modules. We write $\dSt$ for the \icat{} of
\emph{derived stacks}, \ie{} \'etale sheaves of (large) spaces on
$\cdga^{\leq 0}$. Representable (pre)sheaves give a fully faithful
functor $(\cdga^{\leq 0})^{\op} \to \dSt$, and we write
$\dAff \simeq (\cdga^{\leq 0})^{\op}$ for its image; objects of
$\dAff$ will be called \emph{derived affine schemes}.

We denote by $\dArt\subset \dSt$ the full subcategory of derived Artin stacks locally of finite presentation. This is a convenient \icat{} of derived stacks $X$ which admit perfect cotangent complexes $\bL_X$. The dual will be denoted by $\bT_X$.

Consider the 
inclusion functor
\[ i: \mathrm{calg}^{\red} \to \cdga^{\leq 0}, \] where
$\mathrm{calg}^{\red}$ is the full sub-\icat{} of discrete reduced commutative
$k$-algebras. The \icat{} $\mathrm{calg}^{\red}$ can be endowed with
the \'etale topology, and we let $\mathrm{St}_{\red}$ be the \icat{}
of stacks on the associated site. By restriction, we immediately get
a functor of \icats{}
\[ i^*\colon \dSt \to \mathrm{St}_{\red}  \]
which has both a left adjoint $i_!$ and a right adjoint $i_*$, as $i$ is both continous and cocontinous

\begin{defn}$ $
	\begin{itemize}
		\item The functor
		\[ (-)_{\dR} := i_*i^* : \dSt \to \dSt \] is called the \emph{de Rham stack functor}.
		\item The functor \[(-)_{\red} := i_!i^* : \dSt \to \dSt \] is called the \emph{reduced stack functor}.
	\end{itemize}
\end{defn}

Note that by adjunction, for any $X \in \dSt$ we have canonical
morphisms $X \to X_{\dR}$ and $X_{\red} \to X$. One can prove that if
$X\in \dSt$ is a derived stack, then $X_{\dR}$ is simply given by
\[\begin{array}{cccc} X_{\dR}: & (\dAff)^{\op}& \longrightarrow& \mathcal{S} \\
& A &\longmapsto &X(A^{\red}),
\end{array}\]
where $A^{\red}$ is the reduced $k$-algebra $\mathrm{H}^0(A)/\mathrm{Nilp}(\mathrm{H}^0(A))$. On the other hand, if $\Spec A \in \dAff$ is affine, then $(\Spec A)_{\red} \simeq \Spec(A^{\red})$.

The theory of formal localization mainly deals with the study of the projection $X \to X_{\dR}$. This map is of particular interest, as its fibers are precisely the formal completions of $X$ at its points. More concretely, let $\Spec A \to X_{\dR}$ be an $A$-point of $X_{\dR}$, and let $X_A$ be the fiber product
\[
\begin{tikzcd}
X_A \arrow{r}  \arrow{d} & X \arrow{d} \\
\Spec A \arrow{r} &X_{\dR}.
\end{tikzcd}
\]

It can be shown (see \cite[Proposition 2.1.8]{CPTVV}) that $X_A$ is equivalent to the formal completion of the map $\Spec A^{\red} \to X \times \Spec A$. This is easily seen to imply that $(X_A)_{\red} \simeq \Spec A^{\red}$. In other words, one can think of $X_A$ as a sort of ``formal thickening'' of $\Spec A^{\red}$. 
By the properties of the de Rham stack, the map $\Spec A \to X_{\dR}$ corresponds to a map $\Spec A^{\red} \to X$, which is induced by the map $\Spec A^{\red} \simeq (X_A)_{\red} \to X_A$, so that we get a commutative diagram
\[
\begin{tikzcd}
	& X_A \arrow{r} \arrow{d} & X \arrow{d} \\
	\Spec A^{\red} \arrow{r} \ar[ur] & \Spec A \arrow{r} &X_{\dR}
\end{tikzcd}
\]
of derived stacks, where the square on the right is cartesian.

The upshot of the above discussion is that we can think of $X \to X_{\dR}$ as a family of formal derived stacks, and, more explicitly, as the family of formal completions of $X$ at its closed points. By the general theorem of \cite{DAG-X}, these formal completions correspond to dg Lie algebras. However, these dg Lie algebras do not extend to form a sheaf of dg Lie algebras over $X_{\dR}$. Instead, the Chevalley--Eilenberg complexes of these dg Lie algebras extend globally, thus producing a sheaf of graded mixed algebras over $X_{\dR}$.

We are interested in studying prestacks on $X_{\dR}$, that is to say
functors out of the \icat{} $(\dAff_{/X_{\dR}})^{op}$. Let $\Mix$ be
the \icat{} of graded mixed dg modules (\ie{} the \icat{} underlying
the model category of these considered in \cite{CPTVV}). For notational convenience, we give the following definition.

\begin{defn} Let $X$ be a derived stack.
	\begin{itemize}
		\item We denote by $\cD_X$ the \icat{} of prestacks of ind-objects in graded mixed dg modules on $X_{\dR}$, that is to say
		\[ \cD_X := \Fun((\dAff_{/X_{\dR}})^{\op},
                  \Ind(\Mix)).\]
                We consider this as a symmetric monoidal \icat{} with
                respect to the pointwise tensor product coming from $\Ind(\Mix)$.
              \item We denote by $\cA_X$ the \icat{} of prestacks of
                graded mixed cdgas in ind-objects on $X_{\dR}$, that
                is to say
		\[ \cA_X := \Fun((\dAff_{/X_{\dR}})^{\op},
                  \CAlg(\Ind(\Mix))).\]
                Equivalently, since the tensor product on
                $\mathcal{D}_{X}$ is pointwise, we have
                \[ \cA_{X}\simeq \CAlg(\cD_{X}).\]
	\end{itemize}
\end{defn}

Notice that both assignments $X \mapsto \cD_X$ and $X \mapsto \cA_X$ are functorial, in the sense that if we have a map $f: X \to Y$ of derived stacks, we immediately get a functor $f^* : \cD_Y \to \cD_X$ (and similarly for $\cA_X$), simply given by pullback of prestacks. Equivalently, we can encode these functors into cocartesian fibrations $\cD \to \dSt^{\op}$ and $\cA\to \dSt^{\op}$.

Consider the following Ind-object in the \icat{} $\Mix$:
\[  k(\infty):= \{ k(0) \to k(1) \to \cdots \to k(i) \to k(i+1) \to \cdots  \}  \]
where $k(i)$ is the graded mixed module simply given by $k$ sitting in degree 0 and weight $i$, together with the trivial mixed structure. The maps $k(i) \to k(i+1)$ are the canonical morphisms in the \icat{} of graded mixed modules. 

The Ind-object $k(\infty)$ is a commutative algebra in the category $\Ind(\Mix)$, and it can be used to define two fundamental prestacks on $X_{\dR}$.

\begin{defn}
	\begin{enumerate}
		\item The \emph{twisted crystalline structure sheaf} of $X$ is defined to be
		\[\begin{array}{cccc} \bD_{X_{\dR}}^{\infty} : & (\dAff_{/X_{\dR}})^{\op} & \longrightarrow & \CAlg(\Ind(\Mix)) \\
		& (\Spec A \to X_{\dR}) &\longmapsto & \DR(A^{\red}/A) \otimes_k k(\infty),
		\end{array}\]
		\item The \emph{twisted prestack of principal parts} of $X$ is defined as
		\[\begin{array}{cccc} \cP_{X}^{\infty} : & (\dAff_{/X_{\dR}})^{\op} & \longrightarrow & \CAlg(\Ind(\Mix)) \\
		& (\Spec A \to X_{\dR}) &\longmapsto &  \DR(\Spec A^{\red} / X_A) \otimes_k k(\infty).
		\end{array}\]
	\end{enumerate}
\end{defn}

Both prestacks $\bD_{X_{\dR}}^{\infty}$ and $\cP_X^{\infty}$ are functorial in $X$, in the sense that they can be interpreted as sections of the coartesian fibration $\cA \to \dSt^{\op}$. We will denote the corresponding sections by $\bD^{\infty}$ and $\cP^{\infty}$ respectively. Notice however that given a map of derived stacks $f:X 	\to Y$, we have $f^*\bD_{Y_{\dR}}^{\infty} \simeq \bD_{X_{\dR}}^{\infty}$ but in general $f^*\cP_Y^\infty$ is not equivalent to $\cP_X^\infty$.
In other words, the section $\bD^{\infty}$ is cocartesian, while $\cP^{\infty}$ is not. We remark however that there is always an induced map $f^*_{\cP} : f^*\cP_Y^\infty \to \cP_X^\infty$.

For every derived stack $X$, there is a natural map $\bD_{X_{\dR}}^{\infty} \to \cP_X^{\infty}$ in the category $\cA_X$, which one can view as endowing $\cP_X^{\infty}$ with the structure of a $\bD_{X_{\dR}}^{\infty}$-algebra. 

For notational convenience, we give the following definition.

\begin{defn}
	The cocartesian fibration associated to the functor
	\[ X \mapsto \Mod_{\bD_{X_{\dR}}^\infty}(\cD_X) \]
	will be denoted $\cM \to \dSt^{\op}$.
\end{defn}

Notice that by definition we have an equivalence $\cA_X \simeq \CAlg(\mathcal{D}_{X})$
which in turn gives an equivalence
\[ (\cA_X)_{\bD^\infty_{X_{\dR}}/} \simeq
  \CAlg(\Mod_{\bD^\infty_{X_{\dR}}}(\mathcal{D}_{X})) \simeq
  \CAlg(\cM_X). \]
Thus, $\cP_X^\infty$ can be viewed as an object of $\CAlg(\cM_X)$, and
$\mathcal{P}^{\infty}$ as a section of the cocartesian fibration
$\cM_{\CAlg} \to \dSt^{\op}$ corresponding to $\CAlg(\cM_{(\blank)})$.

By a slight abuse of notation we denote by $\cM_{\CAlg}\rightarrow \dArt^{\op}$ the restriction of the cocartesian fibration $\cM_{\CAlg}\rightarrow \dSt^{\op}$ to the full subcategory of derived Artin stacks locally of finite presentation. The following is the key input we will need to apply the results of
the previous section to coisotropic correspondences:
\begin{propn}\label{propn:Ppullbacks}
	Suppose that the diagram
	\[ \begin{tikzcd}
		W \arrow{r}{f} \arrow{d}{p} & X \arrow{d}{q} \\
		Y \arrow{r}{g} & Z 
		\end{tikzcd}
		\]
	is a pullback of derived Artin stacks locally of finite presentation. Then the induced diagram
	\[ \begin{tikzcd} f^*q^*\cP^\infty_Z \arrow{r} \arrow{d} & f^*\cP^\infty_X \arrow{d} \\
		p^*\cP^\infty_Y \arrow{r} & \cP_W^\infty
	\end{tikzcd} \]
	is a pushout in the \icat{} $\CAlg(\cM_W)$.
\end{propn}
\begin{proof}
  Since $\CAlg(\mathcal{M}_{W}) \simeq
  \CAlg(\mathcal{D}_{W})_{\bD^\infty_{W_{\dR}}/}$ and the forgetful
  functor to $\CAlg(\mathcal{D}_{W})$ detects weakly contractible
  colimits, it suffices to show that the underlying diagram in
  $\CAlg(\mathcal{D}_{W})$ is a pushout. But by definition
  $\mathcal{D}_{W}$ is a functor \icat{}, equipped with the pointwise
  symmetric monoidal structure, and so we have an equivalence
  \[ \CAlg(\mathcal{D}_{W}) \simeq \Fun((\dAff_{/W_{\dR}})^{\op},
    \CAlg(\Ind(\Mix))).\] It is therefore enough to check that the
  diagram is a pushout when evaluated at each object of
  $\dAff_{/W_{\dR}}$. In other words, given a point
  $\Spec A \to W_{\dR}$, we need to show that the diagram
		\[ \begin{tikzcd} \cP^\infty_Z(A) \arrow{r} \arrow{d} & \cP^\infty_X(A) \arrow{d} \\
	\cP^\infty_Y (A) \arrow{r} & \cP_W^\infty (A)
	\end{tikzcd} \]
	is a pushout in $\CAlg(\Ind(\Mix))$. Unraveling the definition of the twisted prestack of principal parts, we are left with proving that the diagram
	 	\[ \begin{tikzcd} \DR(\Spec A^{\red}/Z_A) \arrow{r} \arrow{d} & \DR(\Spec A^{\red}/X_A) \arrow{d} \\
	 \DR(\Spec A^{\red}/Y_A) \arrow{r} & \DR(\Spec A^{\red}/W_A)
	 \end{tikzcd} \]
	 is a pushout of graded mixed commutative algebras. The forgetful functor
	 \[ \CAlg(\Mix) \longrightarrow \CAlg(\mathrm{dg}^\mathrm{gr}) \]
	 creates colimits, hence it suffices to show that the above square is a pushout in the category of graded commutative algebras.
	 The derived stacks $X_A, Y_A, Z_A, W_A$ are algebraisable in the sense of \cite[Definition 2.2.1]{CPTVV}, so by \cite[Proposition 2.2.7]{CPTVV} we have an equivalence $\DR(\Spec A^{\red}/X_A)\simeq \Sym_{A^\red}(\bL_{\Spec A^\red / X_A}[-1])$ of graded commutative algebras and similarly for other stacks.
	 
	  Therefore we need to prove that the square 
	 \[ \begin{tikzcd} \Sym_{A^{\red}}(\bL_{\Spec A^\red / Z_A}[-1]) \arrow{r} \arrow{d} & \Sym_{A^{\red}}(\bL_{\Spec A^\red / X_A}[-1]) \arrow{d} \\
	 \Sym_{A^{\red}}(\bL_{\Spec A^\red / Y_A}[-1]) \arrow{r} & \Sym_{A^{\red}}(\bL_{\Spec A^\red / W_A}[-1])
	 \end{tikzcd} \]
	 is a pushout of graded commutative algebras. Since the functor $\Sym_{A^\red}(-)$ commutes with colimits, it is enough to prove that 
	 \[ \begin{tikzcd} \bL_{\Spec A^\red / Z_A} \arrow{r} \arrow{d} & \bL_{\Spec A^\red / X_A} \arrow{d} \\
	 \bL_{\Spec A^\red / Y_A} \arrow{r} & \bL_{\Spec A^\red / W_A}
	 \end{tikzcd} \]
	 is a pushout square. But this follows directly from \cite[Lemma 3.5]{MelaniSafronov2}.
\end{proof}

\begin{cor}\label{cor:Pfincolim}
  The section $\mathcal{P}^{\infty} \colon \dArt^{\op} \to
  \mathcal{M}_{\CAlg}$ preserves finite colimits.
\end{cor}
\begin{proof}
  Proposition~\ref{propn:Ppullbacks} implies, via
  Lemma~\ref{lem:cocartcolim}, that $\mathcal{P}^{\infty}$ preserves
  pushouts. It thus only remains to show that it preserves the initial
  object, \ie{} that $\mathcal{P}^{\infty}_{\Spec k}$ is the initial
  object of $\CAlg(\mathcal{M}_{\Spec k})$, or equivalently that the
  canonical map $\bD_{(\Spec k)_{\dR}}^{\infty} \to
  \mathcal{P}^{\infty}_{\Spec k}$ is an equivalence. The functor $\bD_{(\Spec k)_{\dR}}^{\infty}$ sends $\Spec A\in\dAff$ to $\DR(A^{\red}/A)\otimes_k k(\infty)$. Similarly, the functor $\mathcal{P}^{\infty}_{\Spec k}$ sends $\Spec A\in\dAff$ to $\DR(A^{\red}/(\Spec k)_A)\otimes_k k(\infty)$. But by definition $(\Spec k)_A \cong \Spec A$, so the map $\bD_{(\Spec k)_{\dR}}^{\infty} \to \mathcal{P}^{\infty}_{\Spec k}$ is an equivalence.
\end{proof}

\subsection{Poisson and Coisotropic Structures}\label{sec:poissonstr}
In this subsection we recall the notions of Poisson and coisotropic
structures in the context of derived algebraic geometry. Let
$\dg$ be the symmetric monoidal model category of
cochain complexes of $k$-modules. We will often work with an arbitrary symmetric monoidal
\icat{} $\cC$ satisfying a set of assumptions (see \cite[Section
1.1]{CPTVV}; in particular, we refer there for a proof that the
\icats{} we consider here satisfy the assumptions).

\begin{assumption}\label{ModelCatAssumption}
Let $\mathbf{C}$ be a symmetric monoidal model category which is combinatorial as a model category. Assume the following:
\begin{enumerate}
\item $\mathbf{C}$ is tensored over $\dg$ compatibly with the model and symmetric monoidal structures.

\item For any cofibration $j\colon X\rightarrow Y$, any object $A\in \mathbf{C}$ and any morphism $u\colon A\otimes X\rightarrow C$ the pushout square
\[
  \begin{tikzcd}
    C  \arrow{r} & D \\
    A\otimes X \arrow{u}{u} \arrow{r}{\id\otimes j} & A\otimes Y \arrow{u}
  \end{tikzcd}
\]
is a homotopy pushout.

\item For a cofibrant object $X\in \mathbf{C}$, the functor $X\otimes (-)\colon \mathbf{C}\rightarrow \mathbf{C}$ preserves weak equivalences.

\item $\mathbf{C}$ is a tractable model category.

\item Weak equivalences in $\mathbf{C}$ are stable under filtered colimits and finite products.
\end{enumerate}
We denote by $\cC$ the localization of $\mathbf{C}$ with respect to
weak equivalences, which is a $k$-linear presentably symmetric
monoidal \icat{}. We will abuse notation and just say ``$\mathcal{C}$
satisfies Assumption~\ref{ModelCatAssumption}'', without explicitly
mentioning the model category $\mathbf{C}$.
\end{assumption}

Recall that $\bP_{s+1}$ is the dg operad controlling $s$-shifted
Poisson algebras (i.e. commutative algebras together with a compatible
Lie bracket of degree $-s$); the notation is chosen so that $\bP_{n}$
is the cohomology of the little discs operad $\mathbb{E}_{n}$ for $n
\geq 2$. The operad $\bP_{s+1}$ can be used to define Poisson structures on commutative algebras (see \cite[Theorem 3.2]{MelaniPois}, \cite[Theorem 1.4.9]{CPTVV} and \cite[Theorem 4.5]{MelaniSafronov1}).

\begin{defn}
Let $\cC$ be a $k$-linear symmetric monoidal \icat{} satisfying Assumption \ref{ModelCatAssumption}. We define $\Alg_{\bP_{s+1}}(\cC)$ to be the localization of the category of $\bP_{s+1}$-algebras in $\mathbf{C}$ along weak equivalences.
\end{defn}

By construction we have a forgetful functor
\[\Alg_{\bP_{s+1}}(\cC)\longrightarrow \CAlg(\cC).\]

\begin{defn}\label{defn:pois}
Let $\cC$ be a symmetric monoidal \icat{} as above. Let $A \in \CAlg(\cC)$ be a commutative algebra. The space $\Pois(A,s)$ of $s$-shifted Poisson structures on $A$ is the fiber of
\[\Alg_{\bP_{s+1}}(\cC) \to \CAlg(\cC)\]
taken at the point corresponding to the given commutative structure on $A$.
\end{defn}

Note that the operad $\bP_{s+1}$ has an involution given by changing
the sign of the bracket which preserves the map from the commutative
operad. Therefore, it induces an involution on $\Alg_{\bP_{s+1}}(\cC)$
which we consider as passing to the opposite $\bP_{s+1}$-algebra and,
similarly, an involution on $\Pois(A, s)$ that we denote by
$\pi_A\mapsto -\pi_A$. 

Let $X$ be a derived Artin stack locally of finite presentation. Recall from the previous section that one can associate to $X$ an \icat{} $\cM_X$, which in the language of \cite{CPTVV} corresponds to $\bD_{X_{\dR}}^\infty$-modules. Moreover, one has a canonical object in $\CAlg(\cM_X)$, given by $\cP_X^\infty$. We can define Poisson structures on $X$ in the following way (see \cite[Theorem 3.1.2]{CPTVV}).

\begin{defn}
\label{defn:poisstacks}
With notations as above, the space $\Pois(X, s)$ of $s$-shifted Poisson structures on $X$ is defined to be the space $\Pois(\cP_X^{\infty},s)$, where $\cP_X^{\infty}$ is considered as a commutative algebra in the \icat{} $\cM_X=\Mod_{\bD^\infty_{X_{\dR}}}(\cD_X)$ of $\bD_{X_{\dR}}^\infty$-modules.
\end{defn}

The notion of shifted Poisson structure on a derived stack admits a
relative version. To state this we will use the following result (Poisson
additivity) proved in \cite[Theorem 2.22]{SafronovAdditivity}.

\begin{thm}\label{thm:add}
Let $\cC$ be a symmetric monoidal \icat{} satisfying Assumption \ref{ModelCatAssumption}. Then there is an equivalence
\[\Alg_{\bP_{s+1}}(\cC) \simeq \Alg(\Alg_{\bP_s}(\cC))\]
of symmetric monoidal \icats{} satisfying the following compatibilities:
\begin{enumerate}
\item It is equivariant with respect to the involution on
  $\Alg_{\bP_{s+1}}(\cC)$ given by passing to the opposite
  $\bP_{s+1}$-algebra and the involution on $\Alg(\Alg_{\bP_s}(\cC))$
  given by passing to the opposite associative algebra.

\item It is compatible with the forgetful functors to $\CAlg(\cC)$, i.e. the diagram
\[
\begin{tikzcd}
  \Alg_{\bP_{s+1}}(\cC) \arrow{r}{\sim} \arrow{d}& \Alg(\Alg_{\bP_s}(\cC)) \arrow{d} \\
  \CAlg(\cC) \arrow{r}{\sim} & \Alg(\CAlg(\cC))
\end{tikzcd}
\]
commutes.
\end{enumerate}
\end{thm}

\begin{cor}
Let $\cC$ be a symmetric monoidal \icat{} satisfying Assumption \ref{ModelCatAssumption}. Then there is an equivalence
\[\Alg_{\bP_{s+n}}(\cC)\simeq \Alg_{\bE_n}(\Alg_{\bP_s}(\cC))\]
of symmetric monoidal \icats{}.
\label{cor:Enadditivity}
\end{cor}

For a symmetric monoidal \icat{} $\cC$ we denote by $\LMod(-)$ the $\infty$-category of pairs $(A, M)$ of an algebra $A\in \Alg(\cC)$ and a left $A$-module $M\in\cC$. Note that there is an equivalence
\[\LMod(\CAlg(\cC)) \simeq \Mor(\CAlg(\cC))\]
of \icats{} by \cite{HA}*{Proposition 2.4.3.16}, since the tensor
product in $\CAlg(\cC)$ is the coproduct. As a consequence, we get a forgetful functor
\[\LMod(\Alg_{\bP_s}(\cC)) \to \Mor(\CAlg(\cC))\]
defined for every integer $s$.

\begin{defn}\label{defn:cois}
Let $\cC$ be a symmetric monoidal \icat{} satisfying Assumption \ref{ModelCatAssumption}. Let $\phi\colon A\to B$ be a morphism of commutative algebras in $\cC$. The space $\Cois(\phi,s)$ of $s$-shifted coisotropic structures on $\phi$ is the fiber of
\[ \LMod(\Alg_{\bP_s}(\cC)) \to \Mor(\CAlg(\cC)) \]
taken at the point corresponding to $\phi$.
\end{defn}

We have a forgetful functor
\[\LMod(\Alg_{\bP_s}(\cC))\longrightarrow \Alg(\Alg_{\bP_s}(\cC))\simeq \Alg_{\bP_{s+1}}(\cC)\]
where the last equivalence is given by Theorem \ref{thm:add}, and this is compatible with the forgetful functor to $\CAlg(\cC)$. Therefore, we obtain a forgetful map
\[\Cois(\phi, s)\longrightarrow \Pois(A, s),\]
i.e. an $s$-shifted coisotropic structure on $A\rightarrow B$ encodes an $s$-shifted Poisson structure on $A$ together with some extra data.

\begin{remark}
It is also possible to give another definition of a shifted coisotropic structure. Namely, in \cite{SafronovPoisson} and \cite{MelaniSafronov1} the authors describe a 2-colored operad $\bP_{[s+1,s]}$. An important feature of this operad is that any $\bP_{[s+1,s]}$-algebra $(A,B)$ has an underlying morphism $A\to B$ of commutative algebras.

More specifically, there is a natural morphism of 2-colored operads $\Comm^{\Delta^1} \to \bP_{[s+1,s]}$, where $\Comm^{\Delta^1}$ is the operad of morphisms of commutative algebras. In turn, this induces a forgetful functor $\Alg_{\bP_{[s+1,s]}} \to \Mor(\CAlg)$, and one can define $s$-shifted coisotropic structures in terms of the fiber of this functor.

This alternative definition has the advantage of being somewhat more explicit, and it is proved to be equivalent to Definition \ref{defn:cois} in \cite[Section 3]{SafronovAdditivity}.
\end{remark}

Similarly to what we did in Definition \ref{defn:poisstacks}, we can now extend the notion of shifted coisotropic structure to general morphisms of derived stacks. Let $f\colon Z \rightarrow X$ be a morphism between derived Artin stacks locally of finite presentation. We have an induced symmetric monoidal functor $f^*\colon \cM_X \to \cM_Z$ and a natural map $f^*_{\cP}\colon f^*\cP_X^{\infty} \rightarrow \cP^\infty_Z$ in $\CAlg(\cM_Z)$. We can now give the following definition, which is \cite[Definition 2.1]{MelaniSafronov2}.

\begin{defn}
Let $f\colon Z \to X$ be a morphism of derived Artin stacks locally of finite presentation. The space $\Cois(f,s)$ of $s$-shifted coisotropic structures on $f$ is the pullback
\[
\begin{tikzcd}
	\Cois(f, s) \arrow{r} \arrow{d} & \Pois(X,s) \arrow{d} \\
	\Cois(f^*_{\cP}, s) \arrow{r} & \Pois(f^*\cP_X^{\infty}, s).
\end{tikzcd}
\]
\end{defn}

In other words, an $s$-shifted coisotropic structure on a map $f\colon Z \to X$ of derived stacks is given by an $s$-shifted Poisson structure on $X$, together with a compatible $\bP_{[s+1,s]}$-structure on the morphism $f^*_\cP \colon f^* \cP^\infty_X \to \cP^\infty_Z$.

\subsection{Coisotropic Correspondences}\label{subsec:cois}
We begin by giving the definition of what we mean by a shifted coisotropic correspondence.

\begin{defn}
Let
\[\begin{tikzcd} X & \ar[l, "f"'] Z \arrow{r}{g} & Y \end{tikzcd}\]
be a correspondence of derived Artin stacks locally of finite presentation. The space $\Cois(f, g; s)$ of $s$-shifted coisotropic structures on the correspondence $(f,g)$ is the pullback
\[
  \begin{tikzcd}
    \Cois(f,g;s) \arrow{r} \arrow{d} & \Pois(X,s) \times \Pois(Y,s) \arrow{d} \\
    \Cois((f,g) ; s) \arrow{r} & \Pois(X \times Y, s),
  \end{tikzcd}
\]
where $(f,g)$ is the induced map $Z \to X \times Y$, and the vertical morphism on the right sends a pair of Poisson structures $(\pi_X,\pi_Y)$ on $X$ and $Y$ to the Poisson structure $\pi_X-\pi_Y$ on $X\times Y$.
\end{defn}

The notion of an $s$-shifted coisotropic correspondence can be
reinterpreted in a nice algebraic manner. Namely, consider an
$s$-shifted coisotropic correspondence
\[\begin{tikzcd} X & \ar[l, "f"'] Z \arrow{r}{g} & Y 	\end{tikzcd}\]
between derived Artin stacks locally of finite presentation. The
$s$-shifted Poisson structures on $X$ and $Y$ correspond to
$\bP_{s+1}$-structures on $\cP_X^\infty$ and $\cP_Y^\infty$. By
Poisson additivity we can think of $\cP_X^\infty$ and $\cP_Y^\infty$
as associative algebras in the \icat{} of $\bP_s$-algebras. In other
words, they are objects of the \icats{} $\Alg(\Alg_{\bP_s}(\cM_X))$
and $\Alg(\Alg_{\bP_s}(\cM_Y))$ respectively. Moreover, these
$s$-shifted Poisson structures allow us to enhance
$f^*\cP_X^\infty\otimes g^*\cP_Y^\infty$ to an algebra object in
$\Alg_{\bP_s}(\cM_Z)$.

Next, the $s$-shifted coisotropic structure on $Z \to X \times Y$
endows $\cP_Z^\infty$ with a left module structure over
$f^*\cP_X^\infty\otimes g^*\cP_Y^\infty$ in $\Alg_{\bP_s}(\cM_Z)$. In
other words, $\cP_Z^\infty$ becomes an
$(f^*\cP_X^\infty, g^*\cP_Y^\infty)$-bimodule.

In this sense, coisotropic correspondences give a geometric
incarnation of bimodules. This fact is the main motivation for our
Morita approach to the construction of the \icat{} of coisotropic
correspondences.

\begin{remark}\label{rmk:coismorascorr}
  Note that a coisotropic morphism from $X$ to $Y$ corresponds to $X$
  viewed as a coisotropic correspondence from $\Spec k$ to $Y$.
\end{remark}

Following \S\ref{subsec:revspans}, we have a symmetric monoidal $(\infty, n)$-category $\Span_n(\dArt)$ which has the following informal description:
\begin{itemize}
\item Its objects are derived Artin stacks locally of finite presentation.

\item A 1-morphism from $X$ to $Y$ is given by a correspondence $X\leftarrow Z\rightarrow Y $.

\item Higher morphisms are given by iterated correspondences.
\end{itemize}

The symmetric monoidal structure on $\Span_n(\dArt)$ is given by the
product of derived Artin stacks with the unit given by the terminal object $\pt
= \Spec k$. Each object $X\in\Span_n(\dArt)$ is canonically self-dual with the evaluation and coevaluation maps given by
\[\begin{tikzcd} X\times X & \ar[l, "\Delta"'] X \arrow{r} & \pt \end{tikzcd},\qquad \begin{tikzcd} \pt & \arrow{l} X \arrow{r}{\Delta} & X\times X \end{tikzcd}\]
see \cite[Lemma 12.3]{spans}.

Next, using the notation introduced in the same section, we have a functor
\[\mathfrak{C}_{n} := \Cospan_n(\CAlg(\cM))\colon \dArt^{\op}\longrightarrow \Cat_{(\infty,n)}.\]
This sends a derived stack $X$ to the $(\infty, n)$-category
$\mathfrak{C}_{n}(X) := \Cospan_n(\CAlg(\cM_X))$, which has the following informal description:
\begin{itemize}
\item Its objects are commutative algebra objects in $\cM_X$.

\item A 1-morphism from $A$ to $B$ is given by a cospan $A \rightarrow C \leftarrow B$ of commutative algebras in $\cM_X$.

\item Higher morphisms are given by iterated cospans.
\end{itemize}

Following Section \ref{subsec:spancoeff}, we can also combine the two
$(\infty, n)$-categories we have introduced above into a symmetric
monoidal $(\infty, n)$-category $\Span_n(\dArt; \mathfrak{C}_{n})$
whose objects are pairs $(X, A)$ of a derived stack $X\in\dArt$ and a
commutative algebra $A\in\cM_X$. By Corollary \ref{cor:Pfincolim} the section $\cP^\infty\colon \dArt^{\op}\rightarrow \cM_\CAlg$ preserves finite colimits, so by Corollary \ref{cor:SPANCOSPAN} and Remark \ref{rmk:SPANCOSPANsection} it induces a symmetric monoidal functor
\[\Span_n(\dArt)\longrightarrow \Span_n(\dArt; \mathfrak{C}_{n}).\]

The cocartesian monoidal structure on $\CAlg(\cM_X)$ corresponds to the usual tensor product of algebras, so by Corollary \ref{cor:cospanmorita} we have an equivalence of diagrams of symmetric monoidal $(\infty, n)$-categories
\[\mathfrak{C}_{n} \simeq \mathfrak{alg}_n(\CAlg(\cM)),\]
where $\mathfrak{alg}_n(-)$ is the Morita $(\infty, n)$-category of
$\bE_n$-algebras. Therefore, we have an equivalence of symmetric monoidal $(\infty, n)$-categories
\[\Span_n(\dArt; \mathfrak{alg}_n(\CAlg(\cM)))\simeq\Span_n(\dArt; \mathfrak{C}_{n}).\]

Next, the forgetful functor
\[\Alg_{\bP_{s-n+1}}(\cM_X)\longrightarrow \CAlg(\cM_X)\]
is symmetric monoidal, so we obtain a forgetful functor of diagrams of symmetric monoidal $(\infty, n)$-categories
\[ \mathfrak{P}_{n}^{s} :=
  \mathfrak{alg}_n(\Alg_{\bP_{s-n+1}}(\cM))\longrightarrow
  \mathfrak{alg}_n(\CAlg(\cM)) \simeq \mathfrak{C}_{n},\]
and hence a forgetful functor of symmetric monoidal $(\infty, n)$-categories
\[\Span_n(\dArt; \mathfrak{P}_{n}^{s}) \longrightarrow \Span_n(\dArt; \mathfrak{C}_{n}).\]

\begin{defn}
The $(\infty, n)$-category $\CoisCorrns$ of $s$-shifted coisotropic correspondences is the pullback
\[
\begin{tikzcd}
\CoisCorrns \arrow{d} \arrow{r} & \Span_n(\dArt; \mathfrak{P}_{n}^{s}) \arrow{d} \\
\Span_n(\dArt) \arrow{r} & \Span_n(\dArt; \mathfrak{C}_{n})
\end{tikzcd}
\]
of $(\infty, n)$-categories.
\end{defn}

Let $\cM_{\bP_{s+1}}\rightarrow \dArt^{\op}$ be the cocartesian
fibration corresponding to the functor
\[\Alg_{\bP_{s+1}}(\cM)\simeq \Alg_{\bE_n}(\Alg_{\bP_{s+1-n}}(\cM)),\]
where the latter equivalence is given by Corollary
\ref{cor:Enadditivity}. For a complete $n$-fold Segal space $\cC$ we
denote by $\cC^\simeq := \cC_{0, \ldots, 0}\in\cS$ the space of objects
of $\cC$. The space $(\CoisCorrns)^{\simeq}$ is given by the pullback of spaces of objects
obtained from the defining pullback of $(\infty,n)$-categories. Since
the $n$-fold Segal spaces of iterated spans are already complete, we
have $\Span_{n}(\dArt)^{\simeq} \simeq \dArt^{\simeq}$, and by
Proposition~\ref{propn:spancospandesc} we have
\[\Span_n(\dArt; \mathfrak{C}_{n})^{\simeq} \simeq
  \cM_{\CAlg}^{\simeq}.\]
Using Lemma~\ref{lm:spandualadjoints} we can similarly identify  $\Span_n(\dArt;
\mathfrak{alg}_n(\Alg_{\bP_{s-n+1}}(\cM)))^{\simeq}$ in terms of the
fibration for $\mathfrak{alg}_n(\Alg_{\bP_{s-n+1}}(\cM))^{\simeq}$.
The symmetric monoidal \icat{} $\Alg_{\bP_{s-n+1}}(\cM)$ is pointed,
so if we assume Conjecture~\ref{conj:pointedcomplete} then we can
identify this space with $\cM_{\bP_{s+1}}^\simeq$. With this
assumption we thus get a pullback square

\[
\begin{tikzcd}
(\CoisCorrns)^\simeq \arrow{d} \arrow{r} & \cM_{\bP_{s+1}}^\simeq \arrow{d} \\
\dArt^\simeq \arrow{r}{\cP^\infty} & \cM_\CAlg^\simeq
\end{tikzcd}
\]
of spaces. Therefore, the space of objects of
$\CoisCorrns$ coincides with the space of derived Artin
stacks $X$ equipped with a lift of $\cP_X^\infty\in\CAlg(\cM_X)$ to a
$\bP_{s+1}$-algebra in $\cM_X$, i.e. an $s$-shifted Poisson structure
$\pi\in\Pois(X, s)$. One may analyze in a similar way the space of
1-morphisms, so let us present an informal summary:
\begin{itemize}
\item Objects of $\CoisCorrns$ are derived Artin stacks $X\in\dArt$ together with an $s$-shifted Poisson structure $\pi_X\in \Pois(X, s)$.

\item Its morphisms from $(X, \pi_X)$ to $(Y, \pi_Y)$ are given by correspondences $X\xleftarrow{f} Z\xrightarrow{g} Y$ of derived Artin stacks equipped with an $s$-shifted coisotropic structure $\gamma_Z\in \Cois(f, g; s)$ compatible with the given $s$-shifted Poisson structures $\pi_X$ and $\pi_Y$.

\item Higher morphisms are given by iterated correspondences.
\end{itemize}

Note that the diagram defining $\CoisCorrns$ is a
diagram of symmetric monoidal $(\infty, n)$-categories. Therefore,
$\CoisCorrns$ carries a natural symmetric monoidal
structure. This symmetric monoidal structure can also be defined by
delooping, \ie{} we have equivalences
\[ \CoisCorrns(\pt,\pt) \simeq \CoisCorr_{(\infty,n-1)}^{s-1}.\]

\begin{thm}\label{thm:coiscorradjoints}
  Assuming Conjecture \ref{conj:moritaadjoints}, the symmetric
  monoidal $(\infty, n)$-category $\CoisCorrns$ has
  duals (\ie{} its objects are dualizable and all $i$-morphisms for $i
  < n$ have adjoints).
\end{thm}
\begin{proof}
  The symmetric monoidal $(\infty, n)$-category $\Span_n(\dArt)$ has
  duals by \cite[Theorem 12.4 and Corollary 12.5]{spans}. 

Assuming Conjecture~\ref{conj:moritaadjoints}, the symmetric monoidal $(\infty,
n)$-categories $\mathfrak{C}_{n}(X)$ and $\mathfrak{P}_{n}^{s}(X)$ have duals
for any derived Artin stack $X\in\dArt$. Thus, by Lemma
\ref{lm:spandualadjoints} the symmetric monoidal $(\infty,
n)$-categories $\Span_n(\dArt; \mathfrak{P}_{n}^{s})$ and
$\Span_{n}(\dArt; \mathfrak{C}_{n})$ have duals. 

The claim therefore follows from Corollary \ref{cor:dualsadjointspullback}.
\end{proof}

\begin{remark}
  Unwinding the definitions, if $X$ is an $s$-shifted derived Poisson
  stack, viewed as an object of $\CoisCorr^{s}_{1}$, then the dual
  $X^{\vee}$ has the same underlying derived stack $X$, but its 
  Poisson structure corresponds to the reversed multiplication on
  $\mathcal{P}^{\infty}_{X}$, viewed as an assocative algebra in
  $\Alg_{\bP_{s}}(\mathcal{M}_{X})$. By Theorem~\ref{thm:add}, in
  terms of $\Alg_{\bP_{s+1}}(\mathcal{M}_{X})$ this amounts to taking
  the negative of the Poisson bracket. Thus a coisotropic
  correspondence from $X$ to $Y$ is equivalent to a coisotropic
  correspondence from $\Spec\ k$ to $X^{\vee}\times Y$, or (using
  Remark~\ref{rmk:coismorascorr}) a coisotropic morphism to $X^{\vee}
  \times Y$.
\end{remark}

\subsection{Relationship with Lagrangian Correspondences}
\label{sect:LagrangianCorrespondences}

In this section we sketch a conjectural relationship between our $(\infty, n)$-category of $s$-shifted coisotropic correspondences and the $(\infty, n)$-category of $s$-shifted Lagrangian correspondences from \cite{spans} and \cite{aksz}.

Let $\cC$ be a symmetric monoidal \icat{} satisfying Assumption \ref{ModelCatAssumption}. Then one has the de Rham functor (see \cite[Section 1.3]{CPTVV})
\[\DR\colon \CAlg(\cC)\longrightarrow \CAlg(\Mix)\]
which sends a commutative algebra $A$ to the graded commutative algebra $\Hom_{\cC}(\mathbf{1}, \Sym_A(\bL_A[-1]))$ equipped with the de Rham differential. One can therefore define the functors
\[\cA^2(s), \cAcl(s)\colon \CAlg(\cC)\rightarrow \CAlg(\cS)\]
of $s$-shifted two-forms and closed $s$-shifted two-forms by
\begin{align*}
\cA^2(A, s) &= \Hom_{\dg^\mathrm{gr}}(k(2)[-s-2], \DR(A)) \\
\cAcl(A, s) &= \Hom_{\Mix}(k(2)[-s-2], \DR(A)),
\end{align*}
where $k(2)[-s-2]$ is the unit object concentrated in weight $2$ and cohomological degree $s+2$ with the trivial mixed structure. Note that by construction we have a natural forgetful map $\cAcl(s)\rightarrow \cA^2(s)$.

Applying the above construction to $\cC=\dg$, the $\infty$-category of complexes of $k$-modules, we obtain functors
\[\cA^2(s)\colon \cdga^{\leq 0}\longrightarrow \CAlg(\cS), \qquad \cAcl(s)\colon \cdga^{\leq 0}\longrightarrow \CAlg(\cS).\]
Let $\cA^2(s), \cAcl(s)\colon \dArt^{\op}\rightarrow \CAlg(\cS)$ be the corresponding right Kan extensions.

\begin{defn}
The $(\infty, n)$-category $\IsotCorrns$ of $s$-shifted isotropic correspondences is
\[\IsotCorrns := \Span_n(\dArt; \cAcl(s))\simeq \Span_{n}(\dArt_{/\cAcl(s)}).\]
\end{defn}

Now suppose $D$ is a finite category with an initial object
$\emptyset\in D$ and let $D^{\triangleright}$ be the category obtained
by formally adjoining a terminal object $\ast\in
D^{\triangleright}$. Suppose $X\colon D\rightarrow \dArt_{/\cAcl(s)}$
is a diagram of derived Artin stacks equipped with closed $s$-shifted
two-forms. Then we obtain a diagram $\bT_X\colon D\rightarrow
\QCoh(X_\emptyset)$ whose value on $d\in D$ is given by pulling back
$\bT_{X_d}$ along the unique map $X_\emptyset\rightarrow X_d$. Using
the closed two-forms we can extend this to a diagram $\bT_X\colon
D^{\triangleright}\rightarrow \QCoh(X_\emptyset)$ whose value on the
final object is $(\bT_X)_{\ast} := \bL_{X_\emptyset}[s]$. We say the
diagram $X\colon D\rightarrow \dArt_{/\cAcl(s)}$ is nondegenerate if
$\bT_X\colon D^{\triangleright}\rightarrow \QCoh(X_\emptyset)$ is a
limit diagram.

\begin{defn}
The $(\infty, n)$-category $\Lagns$ of $s$-shifted Lagrangian
correspondences is the subcategory $\Lagns\subset \IsotCorrns$ consisting of nondegenerate diagrams $\bbS^{i_1, \dots, i_n}\rightarrow \dArt_{/\cAcl(s)}$.
\end{defn}

Let $C$ be a symmetric monoidal category satisfying Assumption \ref{ModelCatAssumption} and $\cC$ its localization. We define $\Alg_{\bP_{s+1}}(C)^\omega$ to be the category whose objects are $\bP_{s+1}$-algebras equipped with a strictly closed two-form $\omega$. We have the following two functors
\[F_1, F_2\colon \Alg_{\bP_{s+1}}(C)^\omega\longrightarrow \Alg_{\bP_{s+1}}(\Mod_{k[\hbar]/\hbar^2}(C))\]
\begin{itemize}
\item Given a $\bP_{s+1}$-algebra $A\in\Alg_{\bP_{s+1}}(C)$, we define $F_1(A)$ to be the commutative algebra $A[\hbar]/\hbar^2$ equipped with the bracket $\{a, b\}_\hbar = (1+\hbar)\{a, b\}$ for $a,b\in A$.

\item Given a $\bP_{s+1}$-algebra $A\in\Alg_{\bP_{s+1}}(C)$ equipped with a closed two-form $\omega = \sum_i f_i\ddr g_i\wedge \ddr h_i$, we define $F_2(A)$ to be the commutative algebra $A[\hbar]/\hbar^2$ equipped with the bracket $\{a, b\}_\hbar = \{a, b\} \pm \hbar \sum_i f_i\{g_i, a\} \{h_i, b\}$ with the sign determined by the Koszul sign rule.
\end{itemize}

Note that both $F_1$ and $F_2$ modulo $\hbar$ are given by the forgetful functor $\Alg_{\bP_{s+1}}(C)^\omega\rightarrow \Alg_{\bP_{s+1}}(C)$ and they preserve weak equivalences. Therefore, after localization they give rise to a diagram of symmetric monoidal \icats{}
\[
\begin{tikzcd}
\Alg_{\bP_{s+1}}(\cC)^\omega \arrow[r, shift left=1.5, "F_1"] \arrow[r, shift right=1.5, "F_2"{below}] & \Alg_{\bP_{s+1}}(\Mod_{k[\hbar]/\hbar^2}(\cC)) \arrow[r] & \Alg_{\bP_{s+1}}(\cC)
\end{tikzcd}
\]
where the last functor is given by evaluating at $\hbar=0$. We denote the limit of the above diagram by $\Alg_{\bP_{s+1}}(\cC)^{\compat}$. This is the $\infty$-category of compatible pairs, see \cite[Definition 1.4.20]{CPTVV} and \cite[Definition 1.24]{PridhamPoisson}.

\begin{defn}
The $(\infty, n)$-category $\CompCorrns$ of $s$-shifted compatible correspondences is the pullback
\[
\begin{tikzcd}
\CompCorrns \arrow{d} \arrow{r} & \Span_n(\dArt; \mathfrak{alg}_n(\Alg_{\bP_{s-n+1}}(\cM)^{\compat})) \arrow{d} \\
\Span_n(\dArt) \arrow{r} & \Span_n(\dArt; \mathfrak{C}_{n})
\end{tikzcd}
\]
of $(\infty, n)$-categories.
\end{defn}

Note that by construction we have a symmetric monoidal forgetful functor
\[\CompCorrns \longrightarrow \CoisCorrns.\]

We expect that one may define nondegenerate coisotropic
correspondences $\CoisCorrnsnd\subset \CoisCorrns$ similarly to
$\Lagns \subset \IsotCorrns$. Denote
\[\CompCorrnsnd := \CoisCorrnsnd\times_{\CoisCorrns} \CompCorrns.\]

\begin{conjecture}$ $
\begin{enumerate}
\item The projection $\CompCorrnsnd\rightarrow \CoisCorrnsnd$ is an equivalence.

\item There is a symmetric monoidal functor $\CompCorrnsnd \rightarrow \IsotCorrns$.

\item The previous functor restricts to an equivalence $\CompCorrnsnd\rightarrow \Lagns$.
\end{enumerate}
\end{conjecture}
This conjecture would establish the existence of a symmetric monoidal
functor of $(\infty, n)$-categories
\[\Lagns \longrightarrow \CoisCorrns\]
which is an equivalence onto the subcategory $\CoisCorrnsnd$.

Let us note that there is a forgetful functor from
$\Alg_{\bP_{s+1}}(\cC)^\omega$ to the $\infty$-category of commutative
algebras equipped with a closed $s$-shifted two-form. Thus, the second
claim is closely related to \cite[Conjecture 1.3.1]{bvq}. The claims
(1) and (3) on the level of objects have been proven in \cite[Theorem
3.2.4]{CPTVV} and \cite[Theorem 3.33]{PridhamPoisson}. The same claims
on the level of 1-morphisms have been proven in
\cite{PridhamLagrangian} and \cite[Theorem 4.22]{MelaniSafronov2}.

\begin{remark}
  In \cite{aksz} it is also shown that every symplectic derived stack
  determines an oriented extended TQFT using the \emph{AKSZ
    construction} (defined in the derived algebro-geometric context in
  \cite{PTVV}). It is tempting to speculate that there exists an analogue
  of the AKSZ construction for derived Poisson stacks
  (cf.~\cite{JohnsonFreydAKSZ}), and that this can be used to
  construct, for every derived Poisson stack, oriented extended TQFTs
  \[ \txt{Bord}_{0,n}^{\txt{or}} \to \CoisCorrns.\]
\end{remark}

\appendix
\section{Twisted Arrows and Bifibrations}
Our goal in this appendix is to prove two somewhat technical
results, Corollary~\ref{cor:bifibseclfib} and
Proposition~\ref{propn:Funcartcocart}, which will allow us to describe
the higher category of spans with coefficients in cospans in
Proposition~\ref{propn:Ospancospanbifib}.

\subsection{Bifibrations}\label{subsec:bifib}
We begin with a preliminary discussion of bifibrations, in the
following sense:

\begin{defn}\label{defn:bifib}
  A \emph{bifibration} $(p,q) \colon \mathcal{E} \to \mathcal{A}
  \times \mathcal{B}$ consists of a cartesian fibration $p$ and a
  cocartesian fibration $q$ such that a morphism $f$ in $\mathcal{E}$ is
  \begin{itemize}
  \item $p$-cartesian \IFF{} $q(f)$ is an equivalence,
  \item $q$-cocartesian \IFF{} $p(f)$ is an equivalence.
  \end{itemize}
\end{defn}
\begin{remark}
  This definition is a model-independent version of
  \cite{HTT}*{Definition 2.4.7.2}.
\end{remark}

\begin{lemma}\label{lem:bifibautocart}
  Consider a commutative triangle of \icats{}
  \opctriangle{\mathcal{E}}{\mathcal{E}'}{\mathcal{A} \times
    \mathcal{B},}{f}{(p,q)}{(p',q')} where $(p,q)$ and $(p',q')$ are
  bifibrations. Then $f$ takes $q$-cocartesian morphisms to
  $q'$-cocartesian morphism, and $p$-cartesian morphisms to
  $p'$-cartesian morphisms.
\end{lemma}
\begin{proof}
  This is immediate from the definition as $f$ takes
  a morphism $\phi$ in $\mathcal{E}$ such that $p(\phi)$ is an
  equivalence to the morphism $f(\phi)$ where $p'f(\phi) \simeq
  p(\phi)$ is an equivalence, and similarly for $q$.
\end{proof}

\begin{propn}\label{propn:cartcocartfib}
  Suppose $(p,q) \colon \mathcal{E} \to \mathcal{A} \times
  \mathcal{B}$ is a functor such that $p$ is a cartesian fibration,
  $q$ is a cocartesian fibration, $p$ takes $q$-cocartesian morphisms
  to equivalences, and $q$ takes $p$-cartesian morphisms to
  equivalences. Then:
  \begin{enumerate}[(i)]
  \item The functor $q_{a}\colon \mathcal{E}_{a} \to \mathcal{B}$ on
    fibres at $a \in \mathcal{A}$ is a cocartesian fibration, and a
    morphism in $\mathcal{E}_{a}$ is $q_{a}$-cocartesian \IFF{} its
    image in $\mathcal{E}$ is $q$-cocartesian.
  \item The functor $p_{b} \colon \mathcal{E}_{b} \to \mathcal{A}$ on
    fibres at $b \in \mathcal{B}$ is a cartesian fibration, and a
    morphism in $\mathcal{E}_{b}$ is $p_{b}$-cartesian \IFF{} its
    image in $\mathcal{E}$ is $p$-cartesian.
  \end{enumerate}
\end{propn}
\begin{proof}
  We prove (i); the proof of (ii) is the same. Suppose $x \xto{\phi} x'$ is a
  morphism in $\mathcal{E}_{a}$, \ie{} a morphism in $\mathcal{E}$
  over $b \to b'$ in $\mathcal{B}$ and $\id_{a}$ in
  $\mathcal{A}$. Then for $y \in \mathcal{E}$ we have a commutative
  diagram
  \[
    \begin{tikzcd}
      \Map_{\mathcal{E}}(x', y) \arrow{r} \arrow{d} &
      \Map_{\mathcal{E}}(x,y) \arrow{d} \\
      \Map_{\mathcal{A}}(a,py) \times \Map_{\mathcal{B}}(b', qy)
      \arrow{d} \arrow{r} & \Map_{\mathcal{A}}(a,py) \times
      \Map_{\mathcal{B}}(b, qy) \arrow{d} \\
      \Map_{\mathcal{B}}(b', qy)
      \arrow{r} & 
      \Map_{\mathcal{B}}(b, qy).\\
    \end{tikzcd}
  \]
  Here the bottom square is cartesian (since $p\phi$ is an
  equivalence in $\mathcal{A}$), and so the top square
  is cartesian \IFF{} the outer square is cartesian.

  Suppose first that $\phi$ is $q$-cocartesian, so that the outer
  square is cartesian
  for any $y$. If $py \simeq a$, then we can take fibres in the top square at $\id_{a}
  \in \Map_{\mathcal{A}}(a,a) \simeq \Map_{\mathcal{A}}(a,py)$, giving
  a square
  \[
    \begin{tikzcd}
      \Map_{\mathcal{E}_{a}}(x',y) \arrow{d} \arrow{r} &
      \Map_{\mathcal{E}_{a}}(x,y) \arrow{d} \\
      \Map_{\mathcal{B}}(b', qy)
      \arrow{r} & 
      \Map_{\mathcal{B}}(b, qy),\\
    \end{tikzcd}
  \]
  which is cartesian since the top square is cartesian. This exhibits
  $\phi$ as $q_{a}$-cocartesian. Moreover, since $q$-cocartesian
  morphisms exist, so do $q_{a}$-cocartesian morphisms, \ie{} $q_{a}$
  is a cocartesian fibration.

  Now suppose that $\phi$ is $q_{a}$-cocartesian. To show that $\phi$
  is also $q$-cocartesian we must prove that the top square in the
  diagram above is cartesian for all $y \in \mathcal{E}$. For a given
  $y$ this will follow if we can show that for every map $\psi \colon
  a \to py$ the square
  \[
    \begin{tikzcd}
      \Map_{\mathcal{E}}(x',y)_{\psi} \arrow{d} \arrow{r} &
      \Map_{\mathcal{E}}(x,y)_{\psi} \arrow{d} \\
      \Map_{\mathcal{B}}(b', qy)
      \arrow{r} & 
      \Map_{\mathcal{B}}(b, qy)\\
    \end{tikzcd}
  \]
  of fibres at $\phi$ is cartesian. Let $\bar{\psi} \colon \psi^{*}y
  \to y$ be a $p$-cartesian morphism over $\psi$; then $q\bar{\psi}$
  is an equivalence, so this square is equivalent to 
  \[
    \begin{tikzcd}
      \Map_{\mathcal{E}_{a}}(x',\psi^{*}y) \arrow{d} \arrow{r} &
      \Map_{\mathcal{E}_{a}}(x,\psi^{*}y) \arrow{d} \\
      \Map_{\mathcal{B}}(b', qy)
      \arrow{r} & 
      \Map_{\mathcal{B}}(b, qy),\\
    \end{tikzcd}
  \]
  and this is cartesian since $\phi$ is by assumption $q_{a}$-cocartesian.
\end{proof}

\begin{cor}\label{cor:bifibcrit}
  Suppose
  $(p,q) \colon \mathcal{E} \to \mathcal{A} \times \mathcal{B}$ is as
  in Proposition~\ref{propn:cartcocartfib}. Then the following are
  equivalent:
  \begin{enumerate}[(1)]
  \item $(p,q)$ is a bifibration.
  \item $q_{a}$ is a left fibration for all $a \in \mathcal{A}$.
  \item $p_{b}$ is a right fibration for all $b \in \mathcal{B}$.
  \item The fibre $\mathcal{E}_{a,b}$ is an $\infty$-groupoid for all
    $a \in \mathcal{A}, b \in \mathcal{B}$.
  \end{enumerate}
\end{cor}

\begin{proof}
  Part (i) of Proposition~\ref{propn:cartcocartfib} implies that $(p,q)$ is a
  bifibration \IFF{} every morphism in $\mathcal{E}_{a}$ is
  $q_{a}$-cocartesian for all $a$, \ie{} $q_{a}$ is a left
  fibration. Similarly, part (ii) implies that (1) is equivalent to
  (3). Finally, since $q_{a}$ is by assumption a cocartesian
  fibration, it is a left fibration \IFF{} its fibres
  $\mathcal{E}_{a,b}$ are $\infty$-groupoids for all $b \in
  \mathcal{B}$, so (2) is equivalent to (4).
\end{proof}

We will now show that we can replace bifibrations by left fibrations,
and vice versa, using the following constructions:
\begin{construction}\label{constr:bifibdual}\ 
  \begin{enumerate}[(i)]
  \item Suppose
    $(p,q) \colon \mathcal{E} \to \mathcal{A} \times \mathcal{B}$ is a
    bifibration. Then we have a commutative triangle
    \[
      \begin{tikzcd}
        \mathcal{E} \arrow{rr}{(p,q)} \arrow{dr}{p} & & \mathcal{A} \times
        \mathcal{B} \arrow{dl} \\
         & \mathcal{A}
      \end{tikzcd}
      \]
    where the diagonal maps are cartesian fibrations, and the
    horizontal map takes $p$-cartesian morphisms to cartesian
    morphisms for the projection $\mathcal{A} \times \mathcal{B} \to
    \mathcal{A}$, as these are precisely the morphisms that project to
    equivalences in $\mathcal{B}$. Let $p^{\vee} \colon \mathcal{E}^{\ell} \to
    \mathcal{A}$ be the cocartesian fibration dual to $p$, then
    dualization gives a commutative triangle
    \[
      \begin{tikzcd}
        \mathcal{E}^{\ell} \arrow{rr} \arrow{dr}{p^{\vee}} & & \mathcal{A}^{\op} \times
        \mathcal{B} \arrow{dl} \\
         & \mathcal{A}^{\op}
      \end{tikzcd}
    \]
    where the diagonal maps are cocartesian fibrations and the
    horizontal map preserves cocartesian morphisms.
  \item Suppose $(p,q) \colon \mathcal{F} \to \mathcal{A}^{\op} \times
    \mathcal{B}$ is a left fibration. Then we have a commutative
    triangle
    \[
      \begin{tikzcd}
        \mathcal{F} \arrow{rr}{(p,q)} \arrow{dr}{p} & & \mathcal{A}^{\op} \times
        \mathcal{B} \arrow{dl} \\
         & \mathcal{A}^{\op}
      \end{tikzcd}
    \]
    where the diagonal maps are cocartesian fibrations. A morphism
    $\phi \colon x
    \to x'$ in
    $\mathcal{F}$ is $p$-cocartesian \IFF{} $q(\phi)$ is an
    equivalence in $\mathcal{B}$: In the commutative diagram
    \[
      \begin{tikzcd}
        \Map_{\mathcal{F}}(x', y) \arrow{d} \arrow{r} &
        \Map_{\mathcal{F}}(x,y) \arrow{d} \\
        \Map_{\mathcal{A}^{\op}}(px',py)\times \Map_{\mathcal{B}}(qx',qy)
        \arrow{r} \arrow{d} &  \Map_{\mathcal{A}^{\op}}(px,py)\times
        \Map_{\mathcal{B}}(qx,qy) \arrow{d}\\
        \Map_{\mathcal{A}^{\op}}(px',py) \arrow{r} & \Map_{\mathcal{A}^{\op}}(px,py)        
      \end{tikzcd}
    \]
    the top square is cocartesian since $(p,q)$ is a left fibration,
    while the bottom square is cartesian if $q(\phi)$ is an
    equivalence, hence such a morphism is $p$-cocartesian; since such
    $p$-cocartesian morphisms always exist, by uniqueness all
    $p$-cocartesian morphisms must map to equivalences in
    $\mathcal{B}$. Thus $(p,q)$ preserves cocartesian morphisms in the
    triangle above, and so if $p^{\vee}\colon \mathcal{F}^{b} \to
    \mathcal{A}$ denotes the cartesian fibration dual to $p$, we get a dual triangle
    \[
      \begin{tikzcd}
        \mathcal{F}^{b} \arrow{rr}{(p^{\vee},q')} \arrow{dr}{p^{\vee}} & & \mathcal{A} \times
        \mathcal{B} \arrow{dl} \\
         & \mathcal{A},
      \end{tikzcd}
    \]
    where the diagonal maps are cartesian fibrations and the
    horizontal map preserves cartesian morphisms.
  \end{enumerate}
\end{construction}

\begin{propn}\label{propn:bifibLR}
  We keep the notation of Construction~\ref{constr:bifibdual}.
  \begin{enumerate}[(i)]
  \item Suppose $(p,q) \colon \mathcal{E} \to \mathcal{A} \times
    \mathcal{B}$ is a bifibration. Then $(p^{\vee},q') \colon
    \mathcal{E}^{\ell} \to \mathcal{A}^{\op} \times \mathcal{B}$ is a left fibration.
  \item Suppose $(p,q) \colon \mathcal{F} \to \mathcal{A}^{\op} \times
    \mathcal{B}$ is a left fibration. Then $(p^{\vee},q') \colon
    \mathcal{F}^{b} \to \mathcal{A} \times \mathcal{B}$ is a bifibration.
  \end{enumerate}
\end{propn}

We prove general versions of the criteria we will use to establish
this proposition:
\begin{lemma}\label{lem:firstoftwoCart}
  Suppose given a commutative triangle
  \opctriangle{\mathcal{E}}{\mathcal{D}}{\mathcal{C}}{f}{p}{q} 
  of functors between \icats{} such that:
  \begin{enumerate}
  \item $p$ and $q$ are cartesian fibrations.
  \item $f$ takes $p$-cartesian edges to
    $q$-cartesian edges.
  \item For each object $c \in \mathcal{C}$ the induced map on fibres
    $f_{c} \colon \mathcal{E}_{c} \to \mathcal{D}_{c}$ is a cartesian fibration.
  \item Suppose given a commutative square
    \csquare{\phi^{*}e'}{e'}{\phi^{*}e}{e}{\alpha}{\beta}{\gamma}{\delta}
    in $\mathcal{E}$ lying over the degenerate square
    \csquare{c'}{c}{c'}{c}{\phi}{\id_{c'}}{\phi}{\id_{c}}
    in $\mathcal{C}$, where $\alpha$ and $\delta$ are $p$-cartesian
    edges and $\gamma$ is $f_{c}$-cartesian. Then $\beta$ is
    $f_{c'}$-cartesian. (In other words, the induced functor $\phi^{*} \colon
    \mathcal{E}_{c} \to \mathcal{E}_{c'}$ takes $f_{c}$-cartesian
    edges to $f_{c'}$-cartesian edges.)
  \end{enumerate}
  Then $f$ is also a cartesian fibration.
\end{lemma}
\begin{proof}
  Suppose given $e
  \in \mathcal{E}$ lying over $d \in \mathcal{D}$ and $c \in
  \mathcal{C}$ (i.e. $d \simeq f(e)$ and $c \simeq p(e) \simeq q(d)$)
  and a morphism $\delta \colon d' \to d$ in $\mathcal{D}$ lying over
  $\gamma \colon c' \to c$ in $\mathcal{C}$. Then we must show that
  there exists an $f$-cartesian morphism $e' \to e$ over $\delta$.

  Since $p$ is a cartesian fibration, there exists a $p$-cartesian
  morphism $\beta \colon \gamma^{*}e \to e$ over $\gamma$, and as $f$ takes
  $p$-cartesian edges to $q$-cartesian edges, its image in
  $\mathcal{D}$ is a $q$-cartesian edge $f(\beta) \colon \gamma^{*}d \to d$. There is
  then an essentially unique factorization of $\delta$ through
  $f(\beta)$, as 
  \[ d' \xto{\alpha} \gamma^{*}d \xto{f(\beta)} d.\]
  Now $\alpha$ is a morphism in $\mathcal{D}_{c'}$, so since $f_{c'}$
  is a cartesian fibration there exists an $f_{c'}$-cartesian edge
  $\epsilon \colon \alpha^{*}\gamma^{*}e \to \gamma^{*}e$. We will
  show that the composite $\beta \circ \epsilon \colon
  \alpha^{*}\gamma^{*}e \to \gamma^{*}e \to e$ is an $f$-cartesian
  morphism over $\delta$.

  To see this, we consider the commutative diagram
  \[ %
\begin{tikzpicture}[cross line/.style={preaction={draw=white, -,
line width=6pt}}] %
\matrix (m) [matrix of math nodes,row sep=3em,column sep=2.5em,text height=1.5ex,text depth=0.25ex]
{ 
\Map_{\mathcal{E}}(x, \alpha^{*}\gamma^{*}e) & \Map_{\mathcal{E}}(x,
\gamma^{*}e) & \Map_{\mathcal{E}}(x,e) \\
\Map_{\mathcal{D}}(f(x), d') & \Map_{\mathcal{D}}(f(x), \gamma^{*}d) &
\Map_{\mathcal{D}}(f(x), d) \\
\Map_{\mathcal{C}}(p(x), c') & \Map_{\mathcal{C}}(p(x), c') &
\Map_{\mathcal{C}}(p(x), c),
 \\ };
\path[->,font=\footnotesize] %
(m-1-1) edge (m-1-2)
(m-1-2) edge (m-1-3)
(m-2-1) edge (m-2-2)
(m-2-2) edge (m-2-3)
(m-3-1) edge node[below] {$\id$} (m-3-2)
(m-3-2) edge (m-3-3)
(m-1-1) edge (m-2-1)
(m-1-2) edge (m-2-2)
(m-1-3) edge (m-2-3)
(m-2-1) edge (m-3-1)
(m-2-2) edge (m-3-2)
(m-2-3) edge (m-3-3)
;
\end{tikzpicture}%
\]%
where $x$ is an arbitary object of $\mathcal{E}$. By
\cite{HTT}*{Proposition 2.4.4.3} to see that $\beta
\circ \epsilon$ is $f$-cartesian we must show that the composite of
the two upper squares is cartesian. We will prove this by showing that
both of the upper squares are cartesian. By construction $\beta$ is
$p$-cartesian and $f(\beta)$ is $q$-cartesian, so the composite of the
two right squares and the bottom right square are both cartesian,
hence so is the upper right square.

Since a commutative square of spaces is cartesian \IFF{} the induced
maps on all fibres are equivalences, to see that the upper left square
is cartesian it suffices to show that the square
\nolabelcsquare{\Map_{\mathcal{E}}(x,
  \alpha^{*}\gamma^{*}e)_{\mu}}{\Map_{\mathcal{E}}(x,
  \gamma^{*}e)_{\mu}}{\Map_{\mathcal{D}}(f(x),
  d')_{\mu}}{\Map_{\mathcal{D}}(f(x), \gamma^{*}d)_{\mu}} obtained by
taking the fibre at $\mu \colon p(x) \to c'$ is cartesian for every
map $\mu$. Now taking $p$- and $q$-cartesian pullbacks along $\mu$ we
can (since $f$ takes $p$-cartesian morphisms to $q$-cartesian
morphisms) identify this with the square
\nolabelcsquare{\Map_{\mathcal{E}_{p(x)}}(x,
  \mu^{*}\alpha^{*}\gamma^{*}e)}{\Map_{\mathcal{E}_{p(x)}}(x,
  \mu^{*}\gamma^{*}e)}{\Map_{\mathcal{D}_{p(x)}}(f(x),
  \mu^{*}d')}{\Map_{\mathcal{D}_{p(x)}}(f(x), \mu{*}\gamma^{*}d).}
But this is cartesian since by assumption the map
$\mu^{*}\alpha^{*}\gamma^{*}e \to \mu^{*}\gamma^{*}e$ is
$f_{p(x)}$-cartesian (because $\epsilon$ is $f_{c'}$-cartesian).
\end{proof}
\begin{remark}\label{rmk:firstoftwoR}
  In the situation of Lemma~\ref{lem:firstoftwoCart}, if the maps on
  fibres $f_{c}$ are right fibrations for all $c \in \mathcal{C}$,
  then condition (4) is automatically satisfied, since every morphism
  is $f_{c}$-cartesian.
\end{remark}

\begin{lemma}\label{lem:projcart}
  Suppose $\pi \colon \mathcal{E} \to \mathcal{I} \times \mathcal{J}$ is a
  functor of \icats{} such that
  \begin{enumerate}[(i)]
  \item the composite $\pi_{\mathcal{I}} \colon \mathcal{E} \to \mathcal{I}$ is a cartesian
    fibration, and $\pi$ preserves cartesian morphisms,\footnote{The second part of this
      condition is missing in the published version of the paper.}
  \item for every $i \in \mathcal{I}$, the functor $\pi_{i} \colon \mathcal{E}_{i}
    \to \mathcal{J}$ on fibres over $i$ is a cocartesian fibration.
  \end{enumerate}
  Then the composite
  $\pi_{\mathcal{J}} \colon \mathcal{E} \to \mathcal{J}$ is a
  cocartesian fibration, and $\pi$ preserves cocartesian morphisms.
\end{lemma}
\begin{proof}
  Given $e \in \mathcal{E}$ lying over $j \in \mathcal{J}$ and a
  morphism $\phi \colon j \to j'$, we must
  show that there exists a cocartesian morphism in $\mathcal{E}$ over
  $\phi$ with source $e$. Suppose $e$ lies over $i \in \mathcal{I}$,
  and let $\bar{\phi} \colon e \to e'$ be a cocartesian morphism over $\phi$ in
  $\mathcal{E}_{i}$. We will show that $\bar{\phi}$ is also a
  cocartesian morphism in $\mathcal{E}$. Thus we wish to prove that
  the commutative square
  \csquare{\Map_{\mathcal{E}}(e', x)}{\Map_{\mathcal{E}}(e,
    x)}{\Map_{\mathcal{J}}(j', k)}{\Map_{\mathcal{J}}(j,
    k)}{\bar{\phi}^*}{}{}{\phi^*}
  is cartesian for every $x \in \mathcal{E}$ lying over $k \in
  \mathcal{J}$. It
  suffices to prove that the square
    \csquare{\Map_{\mathcal{E}}(e', x)}{\Map_{\mathcal{E}}(e,
    x)}{\Map_{\mathcal{J}}(j', k) \times \Map_{\mathcal{I}}(i, l)}{\Map_{\mathcal{J}}(j,
    k) \times \Map_{\mathcal{I}}(i,l)}{\bar{\phi}^*}{}{}{\phi^* \times
  \id}
  is cartesian, where $x$ lies over $l$ in $\mathcal{I}$. But to show
  this, it's enough to show the commutative square
  \csquare{\Map_{\mathcal{E}}(e', x)_f}{\Map_{\mathcal{E}}(e,
    x)_f}{\Map_{\mathcal{J}}(j', k)}{\Map_{\mathcal{J}}(j,
    k)}{\bar{\phi}^*}{}{}{\phi^*}
  on fibres over $f \colon i \to l$ is cartesian for all $f$. Since
  $\mathcal{E} \to \mathcal{I}$ is a cartesian fibration, we can
  rewrite this as
  \csquare{\Map_{\mathcal{E}_i}(e', f^* x)}{\Map_{\mathcal{E}_i}(e,
    f^* x)}{\Map_{\mathcal{J}}(j', k)}{\Map_{\mathcal{J}}(j,
    k)}{\bar{\phi}^*}{}{}{\phi^*}
  where $f^{*}x \to x$ is a cartesian morphism over $f$. But now this
  square is cartesian since $\bar{\phi}$ is by assumption cocartesian
  in $\mathcal{E}_{i}$. The assertion that $\pi$ preserves cocartesian
  morphisms amounts to $\pi$ taking $\pi_{\mathcal{J}}$-cocartesian
  morphisms to equivalences in $\mathcal{I}$, which is clear from our
  description of $\pi_{\mathcal{J}}$-cocartesian morphisms.
\end{proof}

\begin{proof}[Proof of Proposition~\ref{propn:bifibLR}]
  We first prove case (i). It follows from
  Corollary~\ref{cor:bifibcrit} and Lemma~\ref{lem:firstoftwoCart} (using
  Remark~\ref{rmk:firstoftwoR}) that $(p^{\vee},q')$ is a cocartesian
  fibration. Moreover, the fibre $\mathcal{E}^{\ell}_{a,b}$ is by
  construction equivalent to the fibre $\mathcal{E}_{a,b}$, which is
  an $\infty$-groupoid, hence $(p^{\vee},q')$ is a left fibration.

  In case (ii), Lemma~\ref{lem:projcart} implies that $q'$ is a
  cocartesian fibration, and that $q'$-cocartesian morphisms map to
  equivalences under $p^{\vee}$. Since we also know that $q'$ takes
  $p^{\vee}$-cartesian morphisms to equivalences,
  Corollary~\ref{cor:bifibcrit} implies that $(p^{\vee},q')$ is a
  bifibration since the fibres $(\mathcal{F}^{b})_{a,b} \simeq
  \mathcal{F}_{a,b}$ are $\infty$-groupoids.
\end{proof}

\begin{remark}
  Dually, we can replace a bifibration $\mathcal{E} \to \mathcal{A}
  \times \mathcal{B}$ by a right fibration $\mathcal{E}^{r} \to
  \mathcal{A} \times \mathcal{B}^{\op}$ and vice versa.
\end{remark}

\begin{remark}
  Let $\Cat_{\infty/\mathcal{A} \times \mathcal{B}}^{\txt{bifib}}$
  denote the full subcategory of $\Cat_{\infty/\mathcal{A} \times
    \mathcal{B}}$ spanned by the bifibrations, and let similarly
  $\Cat_{\infty/\mathcal{C}}^{L}$ and $\Cat_{\infty/\mathcal{C}}^{R}$
  denote the full subcategories of $\Cat_{\infty/\mathcal{C}}$ spanned
  by the left and right fibrations, respectively. Since dualizing
  fibrations is an equivalence of \icats{},  the constructions
  in Proposition~\ref{propn:bifibLR} and their dual versions give equivalences
  \[ \Cat_{\infty/\mathcal{A} \times \mathcal{B}}^{\txt{bifib}} \simeq
    \Cat_{\infty/\mathcal{A}^{\op} \times \mathcal{B}}^{\txt{L}}
    \simeq \Cat_{\infty/\mathcal{A} \times \mathcal{B}^{\op}}^{\txt{R}}.\]
\end{remark}

\subsection{Sections of Bifibrations}
In this subsection we will describe sections of a bifibration in terms
of the corresponding left and right fibrations.

\begin{propn}
  Let $\mathcal{I}$ be an \icat{}. Then the functor
  $(\txt{ev}_{0},\txt{ev}_{1}) \colon \mathcal{I}^{\Delta^{1}} \to
  \mathcal{I} \times \mathcal{I}$ is the \emph{free} bifibration on
  $\mathcal{I}$, in the sense that the map
  \[ \Map_{/\mathcal{I} \times \mathcal{I}}(\mathcal{I}^{\Delta^{1}},
    \mathcal{E}) \isoto \Map_{/\mathcal{I} \times
      \mathcal{I}}(\mathcal{I}, \mathcal{E}),\]
  induced by composition with the canonical map $\txt{const} \colon \mathcal{I} \to
  \mathcal{I}^{\Delta^{1}}$, is an equivalence for every bifibration
  $\mathcal{E} \to \mathcal{I} \times \mathcal{I}$.
\end{propn}
\begin{proof}
  By \cite{freepres}*{Theorem 4.5}, the functor $\txt{ev}_{0} \colon
  \mathcal{I}^{\Delta^{1}} \to \mathcal{I}$ is the free cartesian
  fibration on $\id_{\mathcal{I}}$. Composition with $\txt{const}$
  therefore induces an equivalence
  \[ \Map_{/\mathcal{I}}^{\txt{cart}}(\mathcal{I}^{\Delta^{1}},
    \mathcal{C}) \isoto \Map_{/\mathcal{I}}(\mathcal{I}, \mathcal{C})\]
  for any cartesian fibration $\mathcal{C} \to \mathcal{I}$. In our
  case we then
  have a commutative square
  \csquare{\Map_{/\mathcal{I}}^{\txt{cart}}(\mathcal{I}^{\Delta^{1}},
    \mathcal{E})}{\Map_{/\mathcal{I}}(\mathcal{I}, \mathcal{E})}%
  {\Map_{/\mathcal{I}}^{\txt{cart}}(\mathcal{I}^{\Delta^{1}},
    \mathcal{I} \times \mathcal{I})}%
  {\Map_{/\mathcal{I}}(\mathcal{I}, \mathcal{I}\times \mathcal{I}),}{\sim}{}{}{\sim}
  where the horizontal maps are equivalences.
  On the fibre over $(\txt{ev}_{1},\txt{ev}_{0}) \colon
  \mathcal{I}^{\Delta^{1}} \to \mathcal{I} \times \mathcal{I}$ (which
  corresponds to the diagonal $\Delta \colon \mathcal{I} \to
  \mathcal{I} \times \mathcal{I}$) we get an equivalence
  \[ \Map_{/\mathcal{I} \times \mathcal{I}}(\mathcal{I}^{\Delta^{1}},
    \mathcal{E}) \isoto \Map_{/\mathcal{I} \times
      \mathcal{I}}(\mathcal{I}, \mathcal{E}),\] since the morphisms in
  the source automatically preserve cartesian morphisms by
  Lemma~\ref{lem:bifibautocart}.
\end{proof}

Describing the spaces of sections of a bifibration in
terms of the corresponding left and right fibrations turns out
to involve the twisted arrow \icat{}:
\begin{defn}\label{def:Tw}
  If $\mathcal{C}$ is an \icat{}, we define 
  $\Tw^{r}(\mathcal{C})$ as the simplicial space
  \[ \Map([n] \star [n]^{\op}, \mathcal{C}).\]
  Restricting to the factors $[n]$ and $[n]^{\op}$ we get a projection
  \[ \Tw^{r}(\mathcal{C}) \to \mathcal{C} \times \mathcal{C}^{\op}.\]
  We also define $\Tw^{\ell}(\mathcal{C}) :=
  \Tw^{r}(\mathcal{C})^{\op}$, which as a simplicial space is
  $\Map([n]^{\op} \star [n], \mathcal{C})$.
\end{defn}

The following is essentially \cite{HA}*{Proposition 5.2.1.3} or
\cite{BarwickGlasmanBurnside}*{Proposition 1.1}. Since we have defined
$\Tw^{r}(\mathcal{C})$ as a Segal space rather than a quasicategory,
we briefly discuss how to adapt the proof to this setting.
\begin{propn}\label{propn:TwCicat}\
  \begin{enumerate}[(i)]
  \item If $\mathcal{C}$ is a Segal space, then so is
    $\Tw^{r}\mathcal{C}$.
  \item If $\mathcal{C}$ is a Segal space, then the morphism
    $\Tw^{r}\mathcal{C} \to \mathcal{C} \times
    \mathcal{C}^{\op}$ is a right fibration.
  \item If $\mathcal{C}$ is a complete Segal space, then so is $\Tw^{r}\mathcal{C}$.
  \end{enumerate}
\end{propn}
\begin{proof}
  To see that $\Tw^{r}\mathcal{C}$ is a Segal space, it
  suffices to prove that the morphisms
  $\epsilon^{r}(\Lambda^{n}_{i}) \to \epsilon^{r}(\Delta^{n})$ for $0
  < i < n$ are
  Segal equivalences (\ie{} local equivalences for the localization to
  Segal spaces), where $\epsilon^{r}$ denotes the colimit-preserving functor
  $\mathcal{P}(\simp) \to \mathcal{P}(\simp)$ extending $[n] \mapsto
  [n] \star [n]^{\op}$. This follows from the proof of \cite[Proposition
  5.2.1.3]{HA}, where this map is shown to be inner anodyne in
  simplicial sets.

  A morphism $\mathcal{E} \to \mathcal{B}$ of Segal spaces is a right
  fibration \IFF{} the commutative square
  \[
    \begin{tikzcd}
      \mathcal{E}_{1} \arrow{r} \arrow{d}{d_{0}} & \mathcal{B}_{1} \arrow{d}{d_{0}}
      \\
      \mathcal{E}_{0} \arrow{r} & \mathcal{B}_{0}
    \end{tikzcd}
  \]
  is cartesian. For $\Tw^{r}\mathcal{C} \to \mathcal{C} \times
  \mathcal{C}^{\op}$, we have $(\Tw^{r}\mathcal{C})_{1} \simeq
  \Map(\epsilon^{r}(\Delta^{1}), \mathcal{C})$ where
  $\epsilon^{r}(\Delta^{1}) \simeq \Delta^{1} \star \Delta^{1,\op}
  \simeq \Delta^{3}$, and the square can be rewritten as
  \[
    \begin{tikzcd}
      \Map(\Delta^{3}, \mathcal{C}) \arrow{r} \arrow{d} & \Map(\Delta^{\{0,1\}}
      \amalg \Delta^{\{2,3\}}, \mathcal{C}) \arrow{d} \\
      \Map(\Delta^{\{1,2\}},  \mathcal{C}) \arrow{r} & 
      \Map(\Delta^{\{1\}} \amalg \Delta^{\{2\}}, \mathcal{C}).
    \end{tikzcd}
  \]
  This is cartesian since \[\Delta^{\{0,1\}} \amalg_{\Delta^{\{1\}}}
    \Delta^{\{1,2\}} \amalg_{\Delta^{\{2\}}} \Delta^{\{2,3\}} \to
    \Delta^{3}\] is a (generating) Segal equivalence.
  
  It is easy to see that any right fibration $\mathcal{E} \to \mathcal{B}$ is conservative,
  and so gives a pullback square
  \[
    \begin{tikzcd}
      \mathcal{E}^{\txt{eq}}_{1} \arrow{r} \arrow{d}{d_{0}} & \mathcal{B}^{\txt{eq}}_{1} \arrow{d}{d_{0}}
      \\
      \mathcal{E}_{0} \arrow{r} & \mathcal{B}_{0}.
    \end{tikzcd}
  \]
  Thus if $\mathcal{B}$ is complete then so is $\mathcal{E}$, which
  means that 
  (ii) implies (iii).
\end{proof}

\begin{propn}\label{propn:Twrcartfib}
  The projection $\Tw^{r}(\mathcal{C}) \to \mathcal{C}^{\op}$ is the
  cartesian fibration corresponding to the cocartesian fibration
  $\txt{ev}_{1} \colon \mathcal{C}^{\Delta^{1}}\to \mathcal{C}$.
\end{propn}
\begin{proof}
  Let $\pi \colon \mathcal{E} \to \mathcal{C}^{\op}$ be this dual cartesian
  fibration. Observe that we have a commutative triangle
  \opctriangle{\mathcal{C}^{\Delta^{1}}}{\mathcal{C} \times
    \mathcal{C}}{\mathcal{C}}{(\txt{ev}_{0},
    \txt{ev}_{1})}{\txt{ev}_{1}}{}
  where the downward maps are cocartesian fibrations, and the
  horizontal map preserves cocartesian morphisms. Dualizing, this
  corresponds to a diagram
    \opctriangle{\mathcal{E}}{\mathcal{C} \times
    \mathcal{C}^{\op}}{\mathcal{C}^{\op}}{\phi}{\pi}{}
  where $\phi$ preserves cartesian morphisms. We claim that $\phi$ is
  in fact a right fibration. To prove this we first use
  \cite{HTT}*{Proposition 2.4.2.11} to see that $\phi$ is a locally
  cartesian fibration since fibrewise over $\mathcal{C}^{\op}$ it is
  given by $\mathcal{E}_{x} \simeq (\mathcal{C}^{\Delta^{1}})_{x}
  \simeq \mathcal{C}_{/x} \to \mathcal{C}$ which is a right
  fibration; since the fibres are moreover spaces, this implies that
  $\phi$ is a right fibration.

  We can now use \cite{HA}*{Corollary 5.2.1.22} to conclude that
  $\mathcal{E}$ is equivalent to $\Tw^{r}(\mathcal{C})$ over
  $\mathcal{C} \times \mathcal{C}^{\op}$ \IFF{}
  \begin{enumerate}[(i)]
  \item for $c \in \mathcal{C}$ the fibre $\mathcal{E}_{c,\mathcal{C}}$ has a
    terminal object,
  \item for $c \in \mathcal{C}^{\op}$ the fibre
    $\mathcal{E}_{c,\mathcal{C}^{\op}}$ has
    a terminal object,
  \item an object $x \in \mathcal{E}$ over $(a,b)$ is terminal in
    $\mathcal{E}_{a,\mathcal{C}}$ \IFF{} it is terminal in
    $\mathcal{E}_{b,\mathcal{C}^{\op}}$.
  \end{enumerate}
  In our case, the fibre $\mathcal{E}_{c,\mathcal{C}^{\op}}$ is
  equivalent to $\mathcal{C}_{/c}$, and
  the fibre at $c \in \mathcal{C}$ is $(\mathcal{C}_{c/})^{\op}$ (as
  this is the dualization of the fibre $\mathcal{C}_{c/}\to
  \mathcal{C}$ of $\mathcal{C}^{\Delta^{1}} \to \mathcal{C} \times
  \mathcal{C}$ at $c$, and dualization preserves pullbacks.) Both of these clearly have
  terminal objects. An element in the fibre over $(a,b) \in \mathcal{C}
  \times \mathcal{C}^{\op}$ can be identified with a morphism $b \to
  a$, and in both cases the criterion for this to be a fibrewise
  terminal object is that this morphism must be an equivalence.
\end{proof}

\begin{cor}
  The left and right fibrations corresponding to the bifibration
  $\mathcal{I}^{\Delta^{1}} \to \mathcal{I} \times \mathcal{I}$ are
  the left and right twisted arrow \icats{}
  \[ \Tw^{\ell}\mathcal{I} \to \mathcal{I}^{\op} \times \mathcal{I},
    \quad \Tw^{r}\mathcal{I} \to \mathcal{I} \times
    \mathcal{I}^{\op},\]
  respectively.
\end{cor}
\begin{proof}
  By Proposition~\ref{propn:Twrcartfib}, $\Tw^{r} \mathcal{I} \to
  \mathcal{I}^{\op}$ is the cartesian fibration corresponding to
  $\txt{ev}_{1} \colon \mathcal{I}^{\Delta^{1}} \to \mathcal{I}$, and
  similarly for $\Tw^{\ell}\mathcal{I}$, so this follows from
  Construction~\ref{constr:bifibdual}.
\end{proof}

From this we obtain a useful description of the sections of a bifibration:
\begin{cor}\label{cor:bifibseclfib}
  Suppose $\mathcal{E} \to \mathcal{A} \times \mathcal{B}$ is a
  bifibration. Then for functors $\alpha \colon \mathcal{C} \to
  \mathcal{A}, \beta \colon \mathcal{C} \to \mathcal{B}$, the
  space of sections
  \[
    \begin{tikzcd}
      {} & \mathcal{E} \arrow{d} \\
      \mathcal{C} \arrow{ur} \arrow{r}[below]{(\alpha,\beta)} & \mathcal{A}
      \times \mathcal{B}
    \end{tikzcd}
  \]
  is equivalent to the spaces of commutative squares
  \[
    \left\{
      \begin{tikzcd}
        \mathcal{C}^{\Delta^{1}} \arrow{r} \arrow{d} & \mathcal{E}
        \arrow{d} \\
        \mathcal{C} \times \mathcal{C} \arrow{r}{\alpha \times \beta} &
        \mathcal{A} \times \mathcal{B}
    \end{tikzcd}
  \right\}
  \simeq
\left\{ \begin{tikzcd}
  \Tw^{r}(\mathcal{C}) \arrow{r} \arrow{d} & \mathcal{E}^{r} \arrow{d}\\
  \mathcal{C} \times \mathcal{C}^{\op} \arrow{r}{\alpha \times
    \beta^{\op}} & \mathcal{A} \times \mathcal{B}^{\op}
\end{tikzcd}
 \right\}
 \simeq
 \left\{ \begin{tikzcd}
  \Tw^{\ell}(\mathcal{C}) \arrow{r} \arrow{d} & \mathcal{E}^{\ell} \arrow{d}\\
  \mathcal{C} \times \mathcal{C}^{\op} \arrow{r}{\alpha^{\op} \times
    \beta} & \mathcal{A}^{\op} \times \mathcal{B}
\end{tikzcd}
 \right\}.
\]
\end{cor}

\subsection{Fibrations of Functor $\infty$-Categories}
In this subsection we will prove the following result:
\begin{propn}\label{propn:Funcartcocart}
  Let $\mathcal{F} \to \mathcal{I}$ be the cocartesian fibration for a
  functor $F \colon \mathcal{I} \to \CatI$ and $\mathcal{G} \to
  \mathcal{I}$ be the cartesian fibration for a functor $G \colon
  \mathcal{I}^{\op} \to \CatI$. If $\mathcal{H} \to \mathcal{I}$ is
  the cocartesian fibration for the functor $H := \Fun(F(\blank),
  G(\blank)) \colon \mathcal{I} \to \CatI$, then there is a natural
  equivalence of \icats{}
  \[ \Fun_{\mathcal{I}}(\mathcal{I}, \mathcal{H}) \simeq
  \Fun_{\mathcal{I}}(\mathcal{F}, \mathcal{G}).\]
  Under this equivalence, the cocartesian sections of $\mathcal{H}$
  correspond to the functors $\mathcal{F} \to \mathcal{G}$ that take
  cartesian morphisms to cocartesian morphisms.
\end{propn}

The proof requires understanding a variant of the twisted
arrow category:
\begin{defn}\label{defn:Tw2}
  For an \icat{} $\mathcal{C}$, viewed as a complete Segal space, we
  define $\Tw_{2}(\mathcal{C})$ to be the simplicial space
  \[ \Tw^{r}_{2}(\mathcal{C})_{n} \simeq \Map([n] \star [n]^{\op} \star
  [n], \mathcal{C}).\]
  Since $[n] \star [n]^{\op} \star
  [n]$ can be identified with the pushout of \icats{} $([n] \star
  [n]^{\op}) \amalg_{[n]^{\op}} ([n]^{\op} \star [n])$, the simplicial
  space $\Tw^{r}_{2}(\mathcal{C})$ is given by the pullback
  \[
  \begin{tikzcd}
    \Tw^{r}_{2}(\mathcal{C}) \arrow{r} \arrow{d} &
    \Tw^{r}(\mathcal{C}) \arrow{d} \\
    \Tw^{r}(\mathcal{C})^{\op} \arrow{r} & \mathcal{C}^{\op}.
  \end{tikzcd}
  \]
  This implies in particular that $\Tw^{r}_{2}(\mathcal{C})$ is a
  complete Segal space, \ie{} an \icat{}.
\end{defn}

\begin{lemma}\label{lem:Twrfib}
  Let $f \colon \mathcal{E} \to \mathcal{B}$ be any functor of
  \icats{}. Then \[\Tw^{r}(\mathcal{B}) \times_{\mathcal{B}}
  \mathcal{E} \to \mathcal{B}^{\op}\] is a cartesian fibration,
  corresponding to the functor $\mathcal{B} \to \CatI$ given by
  \[ b \mapsto \mathcal{B}_{/b} \times_{\mathcal{B}} \mathcal{E}.\]
\end{lemma}
\begin{proof}
  This functor factors as the composite
  \[ \Tw^{r}(\mathcal{B}) \times_{\mathcal{B}} \mathcal{E} \to
    \mathcal{E} \times \mathcal{B}^{\op} \to \mathcal{B}^{\op},\]
  where the first functor is a cartesian fibration, being a pullback
  of $\Tw^{r}(\mathcal{B}) \to \mathcal{B} \times \mathcal{B}^{\op}$,
  and the second is obviously a cartesian fibration. Moreover, we can
  write $\Tw^{r}(\mathcal{B}) \times_{\mathcal{B}} \mathcal{E}$ as the
  fibre product
  $\Tw^{r}(\mathcal{B}) \times_{\mathcal{B} \times \mathcal{B}^{\op}}
  \mathcal{E} \times \mathcal{B}^{\op}$ of cartesian fibrations over
  $\mathcal{B}^{\op}$. This identifies the corresponding functors as
  the fibre product of the functors associated to the three factors;
  as $\Tw^{r}(\mathcal{B}) \to \mathcal{B}^{\op}$ corresponds to
  $b \mapsto \mathcal{B}_{/b}$ by Proposition~\ref{propn:Twrcartfib}
  and the two other fibrations correspond to constant functors, this
  gives the result.
\end{proof}

\begin{lemma}\label{lem:Twslice}
  There are natural equivalences of \icats{}
  \[ \Tw^{r}(\mathcal{C}_{/x}) \simeq (\mathcal{C}_{/x})^{\op}
  \times_{\mathcal{C}^{\op}} \Tw^{r}(\mathcal{C}),\]
  \[ \Tw^{r}(\mathcal{C}_{x/}) \simeq \mathcal{C}_{x/}
  \times_{\mathcal{C}} \Tw^{r}(\mathcal{C}).\]
\end{lemma}
\begin{proof}
  We will prove the first equivalence; the proof of the second is
  similar. By the universal property of $\mathcal{C}_{/x}$ and the
  definition of the twisted arrow \icat{}, we have a natural pullback
  square
  \[
  \begin{tikzcd}
    \Map([n], \Tw^{r}(\mathcal{C})_{/x}) \arrow{d} \arrow{r} &
    \Map([n] \star [n]^{\op} \star [0], \mathcal{C}) \arrow{d} \\
    \{x\} \arrow{r} & \Map([0], \mathcal{C}).
  \end{tikzcd}
  \]
  On the other hand, we have a pullback square
  \[
  \begin{tikzcd}
    \Map([n], (\mathcal{C}_{/x})^{\op}
  \times_{\mathcal{C}^{\op}} \Tw^{r}(\mathcal{C})) \arrow{d} \arrow{r} &
    \Map([n]^{\op}, \mathcal{C}_{/x}) \arrow{d} \\
    \Map([n] \star [n]^{\op}, \mathcal{C}) \arrow{r} & \Map([n]^{\op}, \mathcal{C}).
  \end{tikzcd}
  \]
  We can expand this to a commutative diagram
  \[
  \begin{tikzcd}
    \Map([n], (\mathcal{C}_{/x})^{\op}
  \times_{\mathcal{C}^{\op}} \Tw^{r}(\mathcal{C})) \arrow{d} \arrow{r} &
    \Map([n]^{\op}, \mathcal{C}_{/x}) \arrow{d} \arrow{r} & \{x\}
    \arrow{d} \\
    \Map([n] \star [n]^{\op} \star [0], \mathcal{C}) \arrow{r}
    \arrow{d} & \Map([n]^{\op} \star [0], \mathcal{C}) \arrow{d}
    \arrow{r} & \Map([0], \mathcal{C}) \\
    \Map([n] \star [n]^{\op}, \mathcal{C}) \arrow{r} & \Map([n]^{\op}, \mathcal{C}),
  \end{tikzcd}
  \]
  where all the squares are pullbacks. In particular the composite
  square in the top row is a pullback, which shows that $\Map([n], (\mathcal{C}_{/x})^{\op}
  \times_{\mathcal{C}^{\op}} \Tw^{r}(\mathcal{C}))$ is equivalent to
  $\Map([n], \Tw^{r}(\mathcal{C})_{/x})$, naturally in $[n]$ and $x$,
  as required.
\end{proof}

\begin{lemma}\label{lem:wcfibcofcoinit}
  Suppose $\pi \colon \mathcal{E} \to \mathcal{B}$ is a cartesian
  fibration whose fibres are all weakly contractible. Then $\pi$ is
  both cofinal and coinitial.
\end{lemma}
\begin{proof}
  The functor $\pi$ is cofinal by \cite{HTT}*{Lemma 4.1.3.2}. To see
  that it is also coinitial, observe that for any functor
  $F \colon \mathcal{B} \to \mathcal{C}$ the right Kan extension
  $\pi_{*}\pi^{*}F$ exists, and
  $\pi_{*}\pi^{*}F(b) \simeq \lim_{\mathcal{E}_{b}} F(b) \simeq F(b)$
  where the second equivalence uses that $\mathcal{E}_{b}$ is weakly
  contractible; thus $\pi_{*}\pi^{*}F \simeq F$. The limit of
  $\pi^{*}F$ over $\mathcal{E}$ is the limit over $\mathcal{B}$ of
  $\pi_{*}\pi^{*}F \simeq F$, hence $\pi$ is indeed coinitial.
\end{proof}

\begin{lemma}\label{lem:Twcof}
  For any \icat{} $\mathcal{C}$, the functors $\Tw^{r}(\mathcal{C}) \to
  \mathcal{C}, \mathcal{C}^{\op}$ are both cofinal and coinitial.
\end{lemma}
\begin{proof}
  We know that $\Tw^{r}(\mathcal{C}) \to \mathcal{C}$ and
  $\Tw^{r}(\mathcal{C}) \to \mathcal{C}^{\op}$ are cartesian
  fibrations, with fibres $(\mathcal{C}_{x/})^{\op}$ and
  $\mathcal{C}_{/x}$, respectively. These \icats{} are weakly
  contractible, hence these functors are both cofinal and coinitial by
  Lemma~\ref{lem:wcfibcofcoinit}.
\end{proof}

\begin{lemma}\label{lem:Tw2cocart}\ 
  Let $\pi_{0}, \pi_{2} \colon \Tw^{r}_{2}(\mathcal{C}) \to
  \mathcal{C}$ be the projections induced by restriction to the first
  and second copy of $[n]$ in $[n] \star [n]^{\op} \star [n]$,
  respectively. Then
  \begin{enumerate}[(i)]
  \item $\pi_{0}$ is a cartesian fibration, corresponding to the functor
    $x \mapsto \Tw^{r}(\mathcal{C}_{x/})^{\op}$.
  \item $\pi_{2}$ is a cocartesian fibration, corresponding to the
    functor $x \mapsto \Tw^{r}(\mathcal{C}_{/x})$,
  \end{enumerate}
\end{lemma}
\begin{proof}
  From the definition of $\Tw^{r}_{2}(\mathcal{C})$ we have a pullback
  square
  \[
  \begin{tikzcd}
    \Tw^{r}_{2}(\mathcal{C}) \arrow{r} \arrow{d} &
    \Tw^{r}(\mathcal{C}) \arrow{d} \\
    \Tw^{r}(\mathcal{C})^{\op} \arrow{r} & \mathcal{C}^{\op}.
  \end{tikzcd}
  \]
  Now Lemma~\ref{lem:Twrfib} applied to $\Tw^{r}(\mathcal{C})^{\op}
  \to \mathcal{C}^{\op}$ (using the equivalence $\Tw^{r}(\mathcal{C})
  \simeq \Tw^{r}(\mathcal{C}^{\op})$) gives that $\pi_{0}$ is a cartesian fibration
  corresponding to the functor
  \[ x \mapsto (\mathcal{C}_{x/})^{\op} \times_{\mathcal{C}^{\op}}
  \Tw^{r}(\mathcal{C})^{\op}.\]
  Similarly, using the op'ed version of Lemma~\ref{lem:Twrfib} we see
  that $\pi_{2}$ is the cocartesian fibration for the functor
  $x \mapsto (\mathcal{C}_{/x})^{\op} \times_{\mathcal{C}^{\op}}
  \Tw^{r}(\mathcal{C})$.
  Now Lemma~\ref{lem:Twslice} identifies these functors with
  $\Tw^{r}(\mathcal{C}_{\blank/})^{\op}$ and 
  $\Tw^{r}(\mathcal{C}_{/\blank})$, respectively.
\end{proof}

\begin{lemma}\label{lem:cocartcof}
  Consider a diagram \[
  \begin{tikzcd}
    \mathcal{E} \arrow{rr}{\phi} \arrow{dr}[below left]{p} & & \mathcal{F} \arrow{dl}{q} \\
     & \mathcal{B}
  \end{tikzcd}
\]
  where $p$ and $q$ are cocartesian fibrations and $\phi$ preserves
  cocartesian morphisms. If the functor $\phi_{b} \colon \mathcal{E}_{b} \to
  \mathcal{F}_{b}$ is cofinal for every $b \in \mathcal{B}$, then
  $\phi$ is cofinal, as is $\phi \times_{\mathcal{B}} \mathcal{B}'$
  for any functor $\mathcal{B}' \to \mathcal{B}$.
\end{lemma}
\begin{proof}
  It suffices to check that composition with $\phi^{\op}$ preserves
  limits for functors $f \colon \mathcal{F}^{\op} \to
  \mathcal{S}$. But here we have natural equivalences
  \[\lim_{\mathcal{F}^{\op}} f \simeq
    \lim_{b \in \mathcal{B}^{\op}} \lim_{\mathcal{F}_{b}^{\op}}
    f|_{\mathcal{F}_{b}^{\op}} \simeq \lim_{b \in \mathcal{B}^{\op}}
    \lim_{\mathcal{E}_{b}^{\op}} (f\phi)|_{\mathcal{F}_{b}^{\op}} \simeq
    \lim_{\mathcal{E}^{\op}} f\phi.\]
  Since the same condition holds for the pullback of $\phi$ along any
  map $\mathcal{B}' \to \mathcal{B}$, any such pullback of $\phi$ is
  also cofinal.
\end{proof}

\begin{lemma}\label{lem:Tw2cof}
  There is a natural inclusion of posets $[n] \times [1] \to [n] \star
  [n]^{\op} \star [n]$, extending the inclusion of two copies of
  $[n]$, which induces a functor of \icats{}
  \[ \Phi \colon\Tw^{r}_{2}(\mathcal{C}) \to \mathcal{C}^{\Delta^{1}}.\]
  This functor is both cofinal and coinitial.
\end{lemma}
\begin{proof}
  We have commutative diagrams
  \[
  \begin{tikzcd}
    \Tw^{r}_{2}(\mathcal{C}) \arrow{rr}{\Phi} \arrow{dr}[below left]{\pi_{2}} & &
    \mathcal{C}^{\Delta^{1}} \arrow{dl}{\txt{ev}_{1}} \\
    & \mathcal{C}
  \end{tikzcd}
  \qquad
  \begin{tikzcd}
    \Tw^{r}_{2}(\mathcal{C}) \arrow{rr}{\Phi} \arrow{dr}[below left]{\pi_{0}} & &
    \mathcal{C}^{\Delta^{1}} \arrow{dl}{\txt{ev}_{0}} \\
    & \mathcal{C}.
  \end{tikzcd}
  \]
  In the first diagram the diagonal morphisms are both cocartesian
  fibrations, while in the second they are cartesian fibrations;
  moreover, the functor $\Phi$ clearly preserves cocartesian and
  cartesian morphisms for these fibrations.  To show that the top
  morphism is cofinal or coinitial it therefore suffices by
  Lemma~\ref{lem:cocartcof} to show that the induced morphisms on
  fibres are all cofinal in the first diagram and coinitial in the
  second diagram.  At $x \in \mathcal{C}$ we can identify these with
  the projections $\Tw^{r}(\mathcal{C}_{/x}) \to \mathcal{C}_{/x}$ and
  $\Tw^{r}(\mathcal{C}_{x/})^{\op} \to \mathcal{C}_{x/}$,
  respectively. These are both cofinal and coinitial by
  Lemma~\ref{lem:Twcof}.
\end{proof}

\begin{proof}[Proof of Proposition~\ref{propn:Funcartcocart}]
  By \cite{freepres}*{Corollary 7.7} the \icat{}
  $\Fun_{\mathcal{I}}(\mathcal{I}, \mathcal{H})$ is the limit
  \[ \lim_{i \to j \in \Tw^{r}(\mathcal{I})^{\op}} \Fun(\mathcal{I}_{/i},
  H(j)) \simeq \lim_{i \to j \in \Tw^{r}(\mathcal{I})^{\op}}
  \Fun(\mathcal{I}_{/i} \times F(j), G(j)).\]
  Similarly, $\Fun_{\mathcal{I}}(\mathcal{F}, \mathcal{G})$ is the
  limit
  \[ \lim_{i \to j \in \Tw^{r}(\mathcal{I})^{\op}}
  \Fun(\mathcal{I}_{/i} \times_{\mathcal{I}} \mathcal{F},
  G(j)).\]
  Here $\mathcal{I}_{/i} \times_{\mathcal{I}} \mathcal{F} \to
  \mathcal{I}_{/i}$ is a cartesian fibration, equivalent by
  \cite{freepres}*{Corollary 7.6} to the colimit
  \[ \colim_{x \to y \in \Tw^{r}(\mathcal{I}_{/i})} \mathcal{I}_{/x}
  \times F(y).\]
  Thus the \icat{} $\Fun_{\mathcal{I}}(\mathcal{F}, \mathcal{G})$ is
  the limit
  \[\lim_{i \to j \in \Tw^{r}(\mathcal{I})^{\op}} \lim_{x \to y \in
    \Tw^{r}(\mathcal{I}_{/i})^{\op}} \Fun(\mathcal{I}_{/x} \times
  F(y), G(j)).\]
  Let $\Tw^{r}_{3}(\mathcal{I})$ denote the pullback $\Tw^{r}_{2}(\mathcal{I})
  \times_{\mathcal{I}} \Tw^{r}(\mathcal{I})$, where $\Tw_{2}^{r}(\mathcal{I})$
  is defined in Definition~\ref{defn:Tw2}; by
  Lemma~\ref{lem:Tw2cocart} the projection $\Tw^{r}_{3}(\mathcal{I}) \to
  \Tw^{r}(\mathcal{I})$ is then the cocartesian fibration for the
  functor taking $i \to j$ in $\Tw^{r}(\mathcal{I})$ to
  $\Tw^{r}(\mathcal{I}_{/i})$. Combining the limits in the expression
  above, we may therefore identify $\Fun_{\mathcal{I}}(\mathcal{F},
  \mathcal{G})$ with the limit
  \[ \lim_{x \to y \to i \to j \in \Tw_{3}^{r}(\mathcal{I})^{\op}}
  \Fun(\mathcal{I}_{/x} \times F(y), G(j)).\]
  We may also identify $\Tw_{3}(\mathcal{I})$ with the pullback
  $\Tw^{r}(\mathcal{I}) \times_{\mathcal{I}^{\op}}
  \Tw_{2}^{r}(\mathcal{I})^{\op}$.
  The functor whose limit we are taking clearly factors
  through
  \[\Tw^{r}(\mathcal{I}) \times_{\mathcal{I}^{\op}} \Phi^{\op} \colon
  \Tw^{r}(\mathcal{I}) \times_{\mathcal{I}^{\op}}
  \Tw_{2}^{r}(\mathcal{I})^{\op} \to \Tw^{r}(\mathcal{I})
  \times_{\mathcal{I}^{\op}} (\mathcal{I}^{\Delta^{1}})^{\op},\]
  where $\Phi$ is the functor of Lemma~\ref{lem:Tw2cof}. This functor
  is cofinal by Lemma~\ref{lem:cocartcof} since $\Phi$ is fibrewise
  cofinal and coinitial, and so this is the pullback of a fibrewise cofinal
  morphism of cocartesian fibrations over $\mathcal{I}^{\op}$. This
  means we may identify $\Fun_{\mathcal{I}}(\mathcal{F}, \mathcal{G})$
  with the limit
  \[ \lim_{x \to y \to j \in \Tw^{r}(\mathcal{I})^{\op}
    \times_{\mathcal{I}} \mathcal{I}^{\Delta^{1}}}
  \Fun(\mathcal{I}_{/x} \times F(y), G(j)).\]
  Now consider the commutative triangle
  \[
  \begin{tikzcd}
    \mathcal{C} \arrow{rr}{c} \arrow[equals]{dr} & &
    \mathcal{C}^{\Delta^{1}} \arrow{dl}{\txt{ev}_{0}} \\
     & \mathcal{C},
  \end{tikzcd}
  \]
  where $c$ is the functor induced by composition with
  $\Delta^{1} \to \Delta^{0}$, taking an object to its identity
  morphism. This is a morphism of cartesian fibrations, given on
  fibres by $\{x\} \to \mathcal{C}_{x/}$, which is clearly coinitial;
  hence $c$ is itself coinitial, as is its pullback along any morphism
  to the base $\mathcal{C}$. In particular, the induced functor
  $\Tw^{r}(\mathcal{I})^{\op} \to \Tw^{r}(\mathcal{I})^{\op}
  \times_{\mathcal{I}} \mathcal{I}^{\Delta^{1}}$
  is coinitial. Thus $\Fun_{\mathcal{I}}(\mathcal{F}, \mathcal{G})$
  can finally be identified with
  \[ \lim_{x \to j \in \Tw^{r}(\mathcal{I})^{\op}}    
  \Fun(\mathcal{I}_{/x} \times F(j), G(j)),\]
  which is the same as our first expression for
  $\Fun_{\mathcal{I}}(\mathcal{I}, \mathcal{H})$. To identify the
  cocartesian sections, observe that our work so far shows that the
  cocartesian fibration $\mathcal{H} \to \mathcal{I}$ has the same
  universal property as the cocartesian fibration given by (the dual
  of) \cite{HTT}*{Corollary 3.2.2.13}, whose cocartesian sections are
  shown there to be given by functors $\mathcal{F} \to \mathcal{G}$
  that take cartesian morphisms to cocartesian ones.
\end{proof}

\begin{cor}\label{cor:Funfib}
  Let $\mathcal{E} \to \mathcal{C}$ be a cartesian fibration
  corresponding to a functor
  $\epsilon \colon \mathcal{C}^{\op} \to \CatI$, and
  $\mathcal{F} \to \mathcal{D}$ be a cocartesian fibration
  corresponding to a functor $\phi \colon \mathcal{D} \to \CatI$. Then
  if $\mathcal{G} \to \mathcal{C} \times \mathcal{D}$ is the
  cocartesian fibration corresponding to $\Fun(\epsilon, \phi) \colon
  \mathcal{C} \times \mathcal{D} \to \CatI$, then $\mathcal{G}$
  satisfies
  \[ \Fun_{/\mathcal{C} \times \mathcal{D}}(\mathcal{I}, \mathcal{G})
    \simeq \Fun_{/\mathcal{D}}(\mathcal{I} \times_{\mathcal{C}}
    \mathcal{E}, \mathcal{F})\]
  for any functor $\mathcal{I} \to \mathcal{C} \times
  \mathcal{D}$. Under this equivalence, a cocartesian morphism in
  $\mathcal{G}$ corresponds to a functor
  \[ \Delta^{1} \times_{\mathcal{C}} \mathcal{E} \to \mathcal{F}\]
  that takes cartesian morphisms for $\Delta^{1} \times_{\mathcal{C}}
  \mathcal{E} \to \Delta^{1}$ to cocartesian morphisms in
  $\mathcal{F}$.
\end{cor}
\begin{proof}
  Apply Proposition~\ref{propn:Funcartcocart} to the pullback of the
  fibrations to $\mathcal{I}$.
\end{proof}

\end{document}